\documentclass[11pt,reqno]{amsart}

\usepackage{amssymb,amsmath,amsthm,amscd,latexsym,amsfonts}
\usepackage{mathtools}
\usepackage[T1]{fontenc}
\usepackage{graphicx,epstopdf}
\usepackage{graphicx}
\usepackage{xcolor}
\usepackage{comment}
\usepackage{cite}
\usepackage{float}
\usepackage[a4paper,top=2.5cm, bottom=2.5cm, left=2.5cm, right=2.5cm]{geometry}

\newtheorem{thm}{Theorem}

\newtheorem{lemma}{Lemma}
\newtheorem{pro}{Proposition}
\newtheorem{rk}{Remark}
\newtheorem{cor}{Corollary}

\numberwithin{equation}{section} 

\newcommand{\s}{{\sigma}}

\newcommand{\bea}{\begin{eqnarray}}
	\newcommand{\eea}{\end{eqnarray}}

\def\s{\sigma}

\def\s{\sigma}

\def\s{\sigma}

\begin{document}
	\title [Coupled Ising-Potts Model]{Coupled Ising-Potts Model: Rich Sets of Critical Temperatures and Translation-Invariant Gibbs Measures}
	
	\author{F.H. Haydarov, B.A. Omirov, U.A. Rozikov}
\address{F.H. Haydarov$^{c,g}$\begin{itemize}
			\item[$^c$] V.I.Romanovskiy Institute of Mathematics,  Uzbekistan Academy of Sciences, 9, Universitet str., 100174, Tashkent, Uzbekistan;
			\item[$^g$] New Uzbekistan University, 54, Mustaqillik Ave., Tashkent, 100007, Uzbekistan.
		\end{itemize}}
			\email{f.khaydarov@newuu.uz}
\address{Bakhrom A. Omirov \newline \indent
	Institute for Advanced Study in Mathematics,
	Harbin Institute of Technologies, Harbin 150001 \newline \indent
	uzhou Research Institute, Harbin Institute of Technologies, Harbin  215104, Suzhou, P.R. China}
\email{{\tt omirovb@mail.ru}}
\address{ U.A. Rozikov$^{c,d,f}$\begin{itemize}
		\item[$^c$] V.I.Romanovskiy Institute of Mathematics,  Uzbekistan Academy of Sciences, 9, Universitet str., 100174, Tashkent, Uzbekistan;
		\item[$^d$]  National University of Uzbekistan,  4, Universitet str., 100174, Tashkent, Uzbekistan.
		\item[$^f$] Karshi State University, 17, Kuchabag str., 180119, Karshi, Uzbekistan.
\end{itemize}}
\email{rozikovu@yandex.ru}

	\begin{abstract}

We consider a coupled Ising-Potts model on Cayley trees of order $ k \geq 2 $. This model involves spin vectors $ (s, \sigma) $, and generalizes both the Ising and Potts models by incorporating interactions between two types of spins: $ s = \pm 1 $ and $ \sigma = 1, \dots, q $. It is applicable to a wide range of systems, including multicomponent alloys, spin glasses, magnetic systems with multiple phases, biological systems, neural networks, and social models.

In this paper, we find some translation-invariant splitting Gibbs measures (TISGMs) and show, for $ k \geq 2 $, that at sufficiently low temperatures, the number of such measures is at least $ 2^{q} + 1 $. This is not an exact upper bound; for $ k = 2 $ and $ q = 5 $, we demonstrate that the number of TISGMs reaches the exact bound of 335, which is much larger than $ 2^5 + 1 = 33 $. We prove, for $q=5$ that there are 12 critical temperatures at which the number of TISGMs changes, and we provide the exact number of TISGMs for each intermediate temperature. Additionally, we identify temperature regions where three TISGMs, which are `close' to the free measure, are either extreme or non-extreme within the set of all Gibbs measures.

We also show that the coupled Ising-Potts model possesses properties that do not hold for the Ising and Potts models separately. In particular, we observe the following new phenomena:

1. In both the Ising and Potts models, if a Gibbs measure exists at some temperature $T_0 $, then it exists for all $T < T_0 $. However, in the coupled Ising-Potts model, some TISGMs may only exist at intermediate temperatures (neither very low nor very high).

2. It is known that the 5-state Potts model has 3 critical temperatures and up to 31 TISGMs. In this paper, we show that for $q = 5$, the coupled Ising-Potts model has four times as many critical temperatures as the Potts model and approximately 11 times as many TISGMs. Thus, our model alters the number of phases more rapidly and exhibits a significantly richer class of splitting Gibbs measures.
	
	\end{abstract}
	\maketitle
	
	\tableofcontents
	\section{Introduction}
	
	A Cayley tree $\Gamma^k=(V, L)$ (where $V$ is the set of vertices and $L$ is the set of edges) with branching factor $ k\geq 1 $ is a connected infinite graph, every vertex of which has exactly $ k+1 $ neighbors.  The graph $\Gamma^k$ is acyclic, meaning it has no loops or cycles.  The graph has no boundary (it is infinite), but often, in physical or computational studies, a finite truncation of the Cayley tree is used.
	
	The uniformity of the Cayley tree (every vertex has $ k+1 $ neighbors) simplifies certain calculations, particularly in statistical mechanics, where it is used to model systems on hierarchical lattices.		
	Additionally, Cayley trees are used in network theory, ecology, and computer science for various purposes like representing branching processes or hierarchical structures.
		
{\bf The model:}	In this paper we consider Ising model's spins $s(x)\in I=\{-1,1\}$ and  Potts model's spins $\sigma(x)\in
	\Phi:=\{1,2,\dots,q\}$, both are assigned to the vertices
	of Cayley tree.
		
	 A coupled configuration $(s,\sigma)$ on $V$ is then defined
	as a function $x\in V\mapsto (s(x), \sigma (x))\in I\times\Phi$;
	the set of all configurations is $I^V\times\Phi^V$.
	
	The following Hamiltonian is called the coupled Ising-Potts model:
	\begin{equation}\label{ph}
		H(s, \sigma)=-J\sum\limits_{\langle x,y\rangle\in L}
		s(x)s(y)\delta_{\sigma(x)\sigma(y)},
	\end{equation}
	where $J\in \mathbb R$ is a constant,
	$\langle x,y\rangle$ stands for nearest neighbor vertices and $\delta_{ij}$ is the Kroneker's
	symbol:
	$$\delta_{ij}=\left\{\begin{array}{ll}
		0, \ \ \mbox{if} \ \ i\ne j\\[2mm]
		1, \ \ \mbox{if} \ \ i= j.
	\end{array}\right.
	$$

 The coupled Ising-Potts model defined by the Hamiltonian $\eqref{ph}$ is an intriguing model considering the interaction between two distinct types of spins.  It is clear that if $s(x)\equiv -1$ (or 1) then $\eqref{ph}$ coincides with Potts model. Moreover, if $\sigma(x)\equiv i$ for some $i=1,\dots, q$ then  $\eqref{ph}$ coincides with the Ising model.

 We note that in \cite{Mc} the following mixed-spin Ising-Potts model is introduced: each site of a regular
 lattice is occupied by both a two-state Ising spin and a $q$-state Potts spin. The Hamiltonian is

 $$
 \begin{aligned}
 	 H(s, \sigma)=-\sum_{\langle i, j\rangle} K_1 s_i s_j & -\sum_i K_2 s_i \\
 	& -\sum_{\langle i, j\rangle}\left(L_1 \delta_{\sigma_i, \sigma_j} \delta_{s_i, 1} \delta_{s_j, 1}+L_2 \delta_{\sigma_i, \sigma_i} \delta_{s_i,-1} \delta_{s_i,-1}+L_3 \delta_{\sigma_i, \sigma_j} \delta_{s_i,-s_j}\right)
 \end{aligned}
 $$
  where $s_i= \pm 1$ specifies the Ising spin state at lattice site $i, \sigma_i=1, \ldots, q$ specifies the Potts state at site $i$.

 It is easy to see that this Hamiltonian coincides with (\ref{ph}) in case $K_1=K_2=0$ and $L_1=L_2=-L_3=J$.

 In \cite{Mc}, the author considered this model on a Face-Centered Cubic (FCC) lattice and argued that it can be used to represent solid-liquid equilibria in binary mixtures. In the model, the onset of long-range order in the ferromagnetic Potts spins corresponds to solidification, while the ordering of the Ising spins still represents binary phase separation, which can occur in either the solid or liquid phases. Simple lattice models of this type could serve as the foundation for semiempirical models of solid-liquid equilibria, similar to how simple Ising-type models have been used to develop successful liquid mixture excess property models \cite{Ab}. In \cite{Mc}, the Bragg-Williams mean-field method was applied, showing that this model generates phase diagrams typically observed in solid-liquid equilibrium. In our paper, we consider the model on a Cayley tree and explore the splitting of Gibbs measures. Furthermore, we derive the conditions for the (non-) extremality of such measures.

Recall the Ashkin-Teller (AT)  model (see \cite{Ao}, \cite{Do},  \cite{Gi} and the references therein) viewed as a pair of interacting Ising models: For a graph $G=(V, L)$, the AT model is supported on pairs of spin configurations $\left(\tau, \tau^{\prime}\right) \in\{ \pm 1\}^V \times\{ \pm 1\}^V$ and the Hamiltonian is defined by

$$
H(\tau, \tau^{\prime})=\sum\limits_{\langle x,y\rangle\in L} \left(J_\tau \tau_u \tau_v+J_{\tau^{\prime}} \tau_u^{\prime} \tau_v^{\prime}+U \tau_u \tau_u^{\prime} \tau_v \tau_v^{\prime}\right),
$$
where $J_\tau, J_{\tau^{\prime}}, U$ are real parameters.

The coupled Ising-Potts model can be compared to the AT model when\footnote{The coupled Ising-Potts model can be more broadly defined by adding the classic Hamiltonians of both the Ising and Potts models into equality \eqref{ph}. However, in this paper, we will concentrate exclusively on the case described by equality \eqref{ph}.}  $J_\tau=J_{\tau^{\prime}}=0$. In this case, the term at $ U$  depends on two Ising spin configurations. However, the coupled Ising-Potts model features a similar term within a more general framework that describes interactions between multiple types of spins (Ising and Potts).

Another model which should be compared with model (\ref{ph}) is
 the mixed Ashkin-Teller model (see \cite{Be} and the references therein) the Hamiltonian is given by:
 $$
 \begin{aligned}
 	H= & -K_2 \sum_{\langle i, j\rangle}\left(\sigma_i \sigma_j+S_i S_j\right) \\
 	& -K_4 \sum_{\langle i, j\rangle} \sigma_i \sigma_j S_i S_j-D \sum_i S_i^2
 \end{aligned}
 $$
  where the spins $S_i= \pm 1, 0$ and $\sigma_i= \pm 1 / 2$, are localized on the sites of a hypercubic lattice. The first term describes the bilinear interactions between the $\sigma$ and $S$ spins at sites $i$ and $j$, with the interaction parameter $K_2$. The second term describes the four spin interaction with strength $K_4$, and on each site there is a single ion potential $D$.
 	
{\bf Motivations:} In this paper, we examine the splitting of Gibbs measures associated with the Hamiltonian \eqref{ph}. The primary motivations for this investigation are:
	
1. {\it Interacting degrees of freedom}: 	The most immediate reason the coupled Ising-Potts model is interesting is that it introduces interactions between two different types of order parameters, which is common in many physical systems \cite{Ba}. The Ising model represents ferromagnetic ordering where spins take one of two values, while the Potts model allows for more than two states, making it applicable to systems with more complex interactions, such as multicomponent systems. The interaction term in the Hamiltonian, $-Js(x)s(y)\delta_{\sigma(x) \sigma(y)}$, ensures that the Ising spins only interact when the Potts spins at neighboring sites are identical. This coupling between different types of order parameters allows for the study of systems where multiple orderings coexist and influence each other.
	
2. {\it Coupling between different types of orders}:
	In many physical systems, multiple degrees of freedom interact. For example, in magnetic materials, different types of ordering (such as ferromagnetic and antiferromagnetic ordering) can coexist or compete. Similarly, in more complex systems like spin glasses or multicomponent alloys, multiple types of order (which may correspond to different spin models) are coupled. In this coupled Ising-Potts model, the Ising spins represent ferromagnetic interactions, while the Potts spins (see \cite{Rbp}, \cite{Wu}, \cite{Wuf}) can represent multiple types of phases or magnetic orientations. The competition or coexistence between these different types of order is a hallmark of many complex systems, and this model provides a theoretical framework for studying their interactions.
	
3. {\it Phase transitions and critical phenomena:}
	The Ising and Potts models are well-known for exhibiting phase transitions, and their coupling in this model allows the study of how these transitions may occur simultaneously or influence one another (see \cite{Du, Dum, Ga8, Ga13, Ge}, \cite{Os, OV, Per, Per3, Per5}, \cite{Rbp}, \cite{Wu}). The Ising model typically exhibits a second-order phase transition in the case of ferromagnetism, while the Potts model may have a richer set of phase transitions depending on the number of states $q$. Coupling the two models allows the investigation of more complex phase diagrams where both types of transitions could appear in the same system, potentially leading to new and richer critical behavior. The model may exhibit multiple types of critical phenomena, such as continuous or discontinuous transitions depending on the interaction strength and the number of Potts states, making it useful for studying more complex systems.
	
	4. {\it Lattice structure}:
	The fact that this model is defined on a Cayley tree (a hierarchical tree-like lattice) is important because the structure of the lattice can dramatically affect the nature of phase transitions and critical behavior. In statistical mechanics, the Cayley tree is often used as a solvable model to study the behavior of systems with no loops or cycles. The tree-like structure simplifies calculations while still capturing essential features of interacting spin systems. It is known that the critical temperature and other properties of spin systems on hierarchical lattices like the Cayley tree can differ from those on regular lattices, making this an important setting for theoretical studies (see \cite{BR}, \cite{KRK}, \cite{KR}, \cite{Ro}, \cite{Rbp}).

{\bf Examples of applicability}:
	
	a) Multicomponent systems:
	
	An alloy is\footnote{https://en.wikipedia.org/wiki/Alloy} a combination of metals or metals combined with one or more other elements.
	Examples are combining the metallic elements gold and copper produces red gold, gold and silver becomes white gold, and silver combined with copper produces sterling silver. 	
	
	By modeling of microstructural evolution one can study many metals processing
	companies because the alloys are usually designed computationally by tailoring their
	microstructural features (see \cite{AZ}, \cite{Pol} and references therein).
	
	The coupled Ising-Potts model can be used to describe systems where multiple types of components interact. For instance, in alloy systems where two different types of interactions (ferromagnetic and antiferromagnetic) are present, the Ising spins might model the ferromagnetic interaction, while the Potts spins might represent the different phases or types of atoms in the alloy. This generalization to multicomponent systems is essential for understanding materials where more than two types of magnetic ordering or other types of interaction are present simultaneously.
		
	b) Spin-glass systems:	In disordered systems, such as spin glasses \cite{Ni}, the interactions between spins are complex and typically random. The Potts model provides a natural framework to describe such systems, as it can account for multiple spin states. Coupling this with the Ising model allows for more general descriptions of systems with multiple interacting components, such as in random magnetic materials where different types of disorder can coexist and influence each other. The coupled Ising-Potts model can describe the way different types of disorder (e.g., random magnetic fields and competing interactions) affect the overall behavior of the system.
	
	c) Magnetic systems with multiple phases:
	The model is also applicable to magnetic systems where different phases of magnetic ordering exist simultaneously. For example, in certain high-temperature superconductors, magnetic ordering may occur in more than one phase, and these phases may interact. The Potts spins can represent different phases, while the Ising spins could represent ferromagnetic or antiferromagnetic interactions between the particles. The Hamiltonian would describe how these multiple phases influence the magnetic ordering, and how the system transitions between these phases.
	
d) Biological systems:
	In systems biology or ecological models, the Ising spins might represent the state of individual biological entities (such as genes, proteins, or cells), while the Potts spins could represent larger-scale phases or groupings of these entities (such as tissues or biological networks) (see \cite{Rbp} and references therein). For example, in developmental biology, the states of individual cells can be coupled with the overall state of the tissue, where the cells' behavior is influenced by the tissue phase. The model could help understand how these two types of order parameters (individual cells and tissue-level structures) interact during development or disease progression.
	
e) Neural networks:	Neural network models often involve multiple layers or states, where each layer may have multiple possible states \cite{Ho}. The Ising spins might represent the activation state of neurons in a binary network, while the Potts spins could represent different types of activity patterns or classes of neurons. The interaction between these two types of spins could model the influence of different levels of processing in a neural network. This can be useful for understanding the dynamics of learning and information propagation in networks with hierarchical or complex interactions.
	
f) Social and economic models:
	In social physics, the Ising spins could represent individuals' opinions on a particular issue (e.g., agree or disagree), and the Potts spins could represent different societal groups or factions \cite{Gr}. The model would describe how the opinion of individuals is influenced by the groups they belong to, and how group membership can change based on individual interactions. This could be particularly useful for modeling social influence, opinion dynamics, and the formation of consensus in societies with multiple factions.
	
{\bf Structure of the paper}: In Section 2, we derive a vector-valued functional equation involving unknown functions defined at the vertices of Cayley tree. Each solution to this equation corresponds to a splitting Gibbs measure. Notably, a solution (function) that is independent of the vertices of the tree defines a translation-invariant splitting Gibbs measure (TISGM).
	
	Section 3 focuses on translation-invariant splitting Gibbs measures (TISGMs). These measures are influenced by the parameters  $k \geq 2$, $q\geq 2$, and  $\theta>0$. Under certain conditions related to these parameters, we present several TISGMs. Specifically, for  $k = 2$  and  $q = 5$, we demonstrate that the maximum number of TISGMs can reach 335, depending on the value of  $\hat{I}>0$. It is important to note that this number is significantly larger than the 31 TISGMs observed in the 5-state Potts model \cite{KRK}.
	
	Subsection 4.1 devoted to find non-extremality conditions for a TISGM, where we examine the Kesten-Stigum condition based on the second-largest eigenvalue. For three of TISGMs we provide explicit formulas for the critical parameters at which the Kesten-Stigum condition holds. These conditions are not expected to be sharp in all cases, since the Kesten-Stigum bound is not sharp in general.
	
	In Subsection 4.2, using methods of \cite{MSW}, we find extremality conditions for TISGMs considered in subsection 4.1. 	
		
	\section{The compatibility of measures}
	
	Define a finite-dimensional distribution of a probability measure $\mu$ in the volume $V_n$ as
	\begin{equation}\label{p*}
		\mu_n(s_n,\sigma_n)=Z_n^{-1}\exp\left\{-\beta H_n(s_n,\sigma_n)+\sum\limits_{x\in W_n}h_{s(x),\sigma(x),x}\right\},
	\end{equation}
	where $\beta=1/T$, $T>0$-temperature,  $Z_n^{-1}$ is the normalizing factor,
\begin{equation}\label{hx}	\{h_x=(h_{-1,1,x},\dots, h_{-1,q,x}; h_{1,1,x},\dots, h_{1,q,x} )\in \mathbb R^{2q}, x\in V\}
\end{equation}
	 is a collection of vectors and
	$$H_n(s_n, \sigma_n)=-J\sum\limits_{\langle x,y\rangle\in L_n}
		s(x)s(y)\delta_{\sigma(x)\sigma(y)}.$$
	
	We say that the probability distributions (\ref{p*}) are compatible if for all
	$n\geq 1$ and  $s_{n-1}\in I^{V_{n-1}}$,
	$\sigma_{n-1}\in \Phi^{V_{n-1}}$:
	\begin{equation}\label{p**}
		\sum\limits_{(w_n, \omega_n)\in I^{W_n}\times \Phi^{W_n}}\mu_n(s_{n-1}\vee w_n, \sigma_{n-1}\vee \omega_n)=\mu_{n-1}(s_{n-1},\sigma_{n-1}).
	\end{equation}
	Here symbol $\vee$ denotes the concatenation (union) of the configurations.
	
	In this case, by Kolmogorov's theorem there exists a unique measure $\mu$ on $I^V\times\Phi^V$ such that,
	for all $n$ and  $s_n\in I^{V_n}$, $\sigma_n\in \Phi^{V_n}$,
	$$\mu(\{(s, \sigma)|_{V_n}=(s_n,\sigma_n)\})=\mu_n(s_n, \sigma_n).$$
	Such a measure is called a {\it splitting Gibbs measure} (SGM) corresponding to the Hamiltonian (\ref{ph}) and vector-valued functions (\ref{hx}).

	The following theorem describes conditions on $h_x$ guaranteeing compatibility of $\mu_n(s_n,\sigma_n)$.
	
	\begin{thm}\label{ep} Probability distributions
		$\mu_n(s_n,\sigma_n)$, $n=1,2,\ldots$, in
		(\ref{p*}) are compatible iff for any $x\in V$
		the following equation holds:
		\begin{equation}\label{p***}
			z_{\epsilon, i, x}=\prod_{y\in S(x)}
			{(1-\theta^{-\epsilon})\left(\theta^\epsilon z_{\epsilon,i,y}-z_{-\epsilon,i,y}\right)+\sum\limits_{j=1}^{q-1}\left(z_{-1,j,y}+z_{1,j,y} \right)+z_{1,q,y}+1\over \theta+\theta^{-1}z_{1,q,y}+\sum\limits_{j=1}^{q-1}\left(z_{-1,j,y}+z_{1,j,y} \right)},
		\end{equation}
	\begin{equation}\label{ou}	z_{1, q, x}=\prod_{y\in S(x)}
	{(1-\theta^{-1})\left(\theta z_{1,q,y}-1\right)+\sum\limits_{j=1}^{q-1}\left(z_{-1,j,y}+z_{1,j,y} \right)+z_{1,q,y}+1\over \theta+\theta^{-1}z_{1,q,y}+\sum\limits_{j=1}^{q-1}\left(z_{-1,j,y}+z_{1,j,y} \right)},\end{equation}
		where
\begin{equation}\label{zt}		
	\epsilon=-1,1, \ \ z_{\epsilon, i, x}=\exp(h_{\epsilon, i, x}-h_{-1, q, x}), \, i=1,\dots,q-1, \ \
		\theta=\exp(J\beta),\end{equation}
	and $S(x)$ is the set of direct successors of $x$.
	\end{thm}
	\begin{proof}  We use the following equalities:
		$$V_n=V_{n-1}\cup W_n, \ \ W_n=\bigcup_{x\in W_{n-1}}S(x).$$
		
		{\sl Necessity.}  Assume that (\ref{p**}) holds, we shall prove (\ref{p***}). Substituting (\ref{p*}) into
		(\ref{p**}), obtain that  for any configurations $(s_n,\sigma_{n-1})$: $x\in V_{n-1}\mapsto (s_n(x), \sigma_{n-1}(x))\in I\times \Phi $:
	$$			\frac{Z_{n-1}}{Z_n}\sum\limits_{(w_n, \omega_n)\in I^{W_n}\times \Phi^{W_n}}
			\exp\left(\sum\limits_{x\in W_{n-1}}\sum\limits_{y\in S(x)} (J\beta s_{n-1}(x)w_n(y)\delta_{\sigma_{n-1}(x)\omega_n(y)}+
				h_{s_n(y),\omega_n(y),y})\right)$$
		\begin{equation}\label{puu}			
		=		\exp\left(\sum\limits_{x\in
			W_{n-1}}h_{s_{n-1}(x),\sigma_{n-1}(x),x}\right).
	\end{equation}
			
	From (\ref{puu}) we get:
		$${Z_{n-1}\over Z_n}\sum\limits_{(w_n, \omega_n)\in I^{W_n}\times \Phi^{W_n}}
		\prod_{x\in W_{n-1}}\prod_{y\in S(x)} \exp\,(J\beta s_{n-1}(x)w_n(y)\delta_{\sigma_{n-1}(x)\omega_n(y)}+
		h_{s_n(y),\omega_n(y),y})$$ $$ = \prod_{x\in W_{n-1}} \exp\,(h_{s_{n-1}(x),\sigma_{n-1}(x),x}). $$
		Fix $x\in W_{n-1}$ and rewrite the last equality for
		$(s_{n-1}(x), \sigma_{n-1}(x))=(\epsilon, i)$, $\epsilon=-1,1$, $i=1,\dots,q$ and $(s_{n-1}(x), \sigma_{n-1}(x))=(-1, q)$, then dividing each of them to the last one we get
		\begin{equation}\label{als}\prod_{y\in S(x)}\frac{\sum\limits_{(j, u)\in I\times \Phi}\exp\,(J\beta \epsilon j\delta_{iu}+h_{j,u,y})}{\sum\limits_{(j, u)\in I\times \Phi}\exp\,(-J\beta j\delta_{qu}+h_{j,u,y})}= \exp\,(h_{\epsilon,i,x}-h_{-1,q,x}),\end{equation}
		where $\epsilon=-1,1, \, i=1,\dots, q.$
				
		Now by using notations (\ref{zt}), from (\ref{als}) we get  (\ref{p***}).
		Note that $z_{-1,q,x}\equiv 1$. Therefore, $z_{1,q,x}$ satisfies (\ref{ou}).
		
		{\sl Sufficiency.} Suppose that (\ref{p***}) holds. It is equivalent to
		the representations
		\begin{equation}\label{pru}
			\prod_{y\in S(x)}\sum\limits_{(j, u)\in I\times \Phi}\exp\,(J\beta \epsilon j\delta_{iu}+h_{j,u,y})= a(x)\exp\,(h_{\epsilon,i,x}),, \ \ i=1,\dots,q-1
		\end{equation} for some function $a(x)>0, x\in V.$
		We have
		$$
		{\rm LHS \ of \ (\ref{p**})}=\frac{1}{Z_n}\exp(-\beta H(s_{n-1},\sigma_{n-1}))\times
		$$
		\begin{equation}\label{pru1}
			\prod_{x\in W_{n-1}} \prod_{y\in S(x)}\sum\limits_{(\epsilon,u)\in I\times \Phi}\exp\,(J\beta s_{n-1}(x)\epsilon\delta_{\sigma_{n-1}(x)u}+h_{\epsilon,u,y}).\end{equation}
		
		Substituting (\ref{pru}) into (\ref{pru1}) and denoting $A_n=\prod_{x \in W_{n-1}} a(x)$,
		we get
		\begin{equation}\label{pru2}
			{\rm RHS\ \  of\ \  (\ref{pru1}) }=
			\frac{A_{n-1}}{Z_n}\exp(-\beta H(s_{n-1},\sigma_{n-1}))\prod_{x\in W_{n-1}}
			\exp(h_{s_{n-1}(x),\sigma_{n-1}(x),x}).\end{equation}
		
		Since $\mu_n$, $n \geq 1$ is a probability, we should have
		$$ \sum\limits_{(s_{n-1}, \sigma_{n-1})\in I^{V_{n-1}}\times \Phi^{V_{n-1}}} \sum\limits_{(w_n, \omega_n)\in I^{W_n}\times \Phi^{W_n}}\mu_n(s_{n-1}\vee w_n, \sigma_{n-1}\vee \omega_n) = 1. $$
		
		Hence from (\ref{pru2})  we get $Z_{n-1}A_{n-1}=Z_n$, and (\ref{p**}) holds.
		
	\end{proof}
	Hence, from Theorem \ref{ep} it follows that for any $h=\{h_x,\ \ x\in V\}$
satisfying the functional equation (\ref{p***}) there exists a unique splitting Gibbs measure $\mu$ and vice versa.

	\begin{rk} (see \cite[Remark 2.6 and Remark 2.7]{Rbp} )\label{rm:r2.1}
			Denote by $\mathcal G(H)$ the set of all Gibbs measures corresponding to the Hamiltonian $H$ (given by (\ref{ph})). This is a non-empty
			convex compact subset in the space of all
			probability measures.
		We note that splitting Gibbs measures (SGMs) are very important, because:
	
	\begin{itemize}
		\item \ \  \emph{any extreme Gibbs measure
			$\mu\in{\rm ex}\mathcal G(H)$ is SGM}; therefore, the question of
		uniqueness of the Gibbs measure is reduced to that in the SGM class. Moreover, for each given
		temperature, the description
		of the set $\mathcal G(H)$ is equivalent to the full description of the set of all extreme SGMs.
		
		\item \ \ \emph{Any SGM corresponds} to a solution of (\ref{p***}) given in Theorem \ref{ep}. Thus our main problem is reduced
		to solving the functional equation (\ref{p***}) and to check when a SGM corresponding to a solution is extreme.
		
		\end{itemize}
\end{rk}
	Solving the functional equation (\ref{p***}) is extremely challenging due to its non-linear nature, its multidimensional structure, and the fact that the unknown functions are defined on a tree. Even determining all unknown constant functions (which are independent of the vertices of the tree) is a difficult task. In this paper, we present a class of such constant solutions and examine the extremality of the corresponding SGMs.
	
	\section{Translation-invariant Gibbs measures}
	In this section, we consider Gibbs measures which are translation-invariant, which correspond to solutions of the form:
	$$z_{\epsilon, i, x}\equiv z_{\epsilon, i}, \ \mbox{for all} \ \ x\in V.$$ Then from equation (\ref{p***}) we get
	\begin{equation}\label{pt}
				z_{\epsilon, i}=\left(
		{(1-\theta^{-\epsilon})\left(\theta^\epsilon z_{\epsilon,i}-z_{-\epsilon,i}\right)+\sum\limits_{j=1}^{q-1}\left(z_{-1,j}+z_{1,j} \right)+z_{1,q}+1\over \theta+\theta^{-1}z_{1,q}+\sum\limits_{j=1}^{q-1}\left(z_{-1,j}+z_{1,j} \right)}\right)^k, \, \epsilon=-1,1,
	\end{equation}
\begin{equation}\label{ptq}	z_{1, q}=\left({(1-\theta^{-1})\left(\theta z_{1,q}-1\right)+\sum\limits_{j=1}^{q-1}\left(z_{-1,j}+z_{1,j} \right)+z_{1,q}+1\over \theta+\theta^{-1}z_{1,q}+\sum\limits_{j=1}^{q-1}\left(z_{-1,j}+z_{1,j} \right)}\right)^k.
	\end{equation}
	Denoting $u_i=z_{-1,i}$, $v_i=z_{1,i}$, $i=1,\dots,q$, we get from (\ref{pt}) and (\ref{ptq}) (recall that $u_q=1$) a fixed point equation
	\begin{equation}\label{fp}
		(u,v)=F(u,v),
	\end{equation} where $u=(u_1, \dots, u_{q-1})$, $v=(v_1, \dots, v_q)$ and the mapping $F: (u, v)\in\mathbb R_+^{2q-1}\to (u', v')\in \mathbb R^{2q-1}_+$ is defined as
	\begin{equation}\label{pt-}
		\begin{array}{lll}
	u'_i=\left(
	{(1-\theta^{-1})\left(\theta u_i-v_i\right)+\sum\limits_{j=1}^{q-1}\left(u_j+v_j \right)+v_q+1\over \theta+\theta^{-1}v_q+\sum\limits_{j=1}^{q-1}\left(u_j+v_j \right)}\right)^k,\ \ i=1,\dots,q-1;\\[3mm]
	v'_i=\left(
		{(1-\theta^{-1})\left(\theta v_i-u_i\right)+\sum\limits_{j=1}^{q-1}\left(u_j+v_j \right)+v_q+1\over \theta+\theta^{-1}v_q+\sum\limits_{j=1}^{q-1}\left(u_j+v_j \right)}\right)^k,\ \ i=1,\dots,q-1;\\[3mm]
		v'_q=\left(
		{(1-\theta^{-1})\left(\theta v_q-1\right)+\sum\limits_{j=1}^{q-1}\left(u_j+v_j \right)+v_q+1\over \theta+\theta^{-1}v_q+\sum\limits_{j=1}^{q-1}\left(u_j+v_j \right)}\right)^k.
		\end{array}
	\end{equation}
Recall that a set $A$ is invariant for the operator $F$, if $F(A)\subset A$.

For a subset $\emptyset\ne M\subset \{1, \dots, q-1\}$ introduce
$$\mathcal I_M=\{(u,v)\in \mathbb R^{2q-1}_+: u_i=1, v_i=v_q, \forall i\in M\}.$$
$$\mathcal J_M=\{(u,v)\in \mathbb R^{2q-1}_+: u_i=u_j, v_i=v_j, \ \ \forall i,j\in M\}.$$
%

Note that all these sets are invariant with respect to $F$ for any $M\subset \{1, \dots, q-1\}$.

Thus we know that $u_i=1$, $v_i=v_q$ satisfies $i$th equation (for $u$ and $v$) of the system (\ref{pt-}) for each $i=1,2,\dots,q-1$. Without loss of generality we can take a non-negative integer $0\leq m\leq q-1$, and assume
\begin{equation}\label{uvq}
	u_j=1, \, v_j=v_q, \, j=m+1,\dots, q.
	\end{equation}

\subsection{Some general results for arbitrary $k\geq 2$}
In this subsection we give some solutions of (\ref{fp}) for arbitrary $k\geq 2$. In the next subsection we consider the case $k=2$ and give complete analysis of (\ref{fp}) under assumption (\ref{uvq}).

{\bf Case $m=0$}. In this case a solution of (\ref{fp}) has form
$(1,\dots,1, v_q, \dots, v_q)\in \mathbb R^{2q-1}_+$ where $v_q$ (denoted by $w$) is a solution of
\begin{equation}\label{kap}
	w=\left(  {\Theta w+1\over w+\Theta} \right)^k, \ \ \mbox{where} \ \ \Theta={\theta+q-1\over \theta^{-1}+q-1}.
\end{equation}
Note that $\Theta(\theta)>0$ for any $\theta>0$.
\begin{lemma}\label{m0} Let $k\geq 2$. The equation (\ref{kap}) has unique solution $w=1$ if $0<\Theta\leq {k+1\over k-1}$ and it has three solutions if $\Theta> {k+1\over k-1}$.
\end{lemma}
\begin{proof}
	Note that $w=1$ is a solution of (\ref{kap}) independently on parameters. Define function
	$$f(x)=\left(  {\Theta x+1\over x+\Theta} \right)^k.$$

{\it Case: $\Theta<1$}. In this case (\ref{kap}) has unique solution, because $f(x)$ is decreasing in $(0, +\infty)$.

{\it Case: } $\Theta>1$. Denote $u=\sqrt[k]{w}$.
Rewrite (\ref{kap}) as
\begin{equation}\label{fd} u^{k+1}-\Theta u^k+\Theta u-1=0.
	\end{equation}
It is well known (see \cite{Pra}, p.28) that the number of positive roots of a polynomial
does not exceed the number of sign changes of its coefficients\footnote{This is known as Descartes rule, which states that if the nonzero terms of a single-variable polynomial with real coefficients are ordered by descending variable exponent, then the number of positive roots of the polynomial is either equal to the number of sign changes between consecutive (nonzero) coefficients, or is less than it by an even number. A root of multiplicity $n$ is counted as $n$ roots. 	In particular, if the number of sign changes is zero or one, the number of positive roots equals the number of sign changes.}.
By this rule, the equation (\ref{fd}) has up to 3 positive roots. We show now that it has exactly three roots.

 Note that the function $f(x)$ for $x>0$ is increasing and bounded. We have
	\begin{equation}\label{sr4.5}
		{d\over dx}f(1)=f'(1)=k{\Theta-1\over \Theta+1}>0
	\end{equation}
	and
	$$f'(1)=1, \ \ \mbox{gives} \ \ \Theta_c=
		{k+1\over k-1}.$$
If	$0<\Theta\leq \Theta_c$ then $f'(1)<1$, the solution $x=1$	is a stable fixed point of the map $f(x)$, and $\lim_{n\to\infty}f^{(n)}(x) = 1$, for any $x > 0$.
	Here, $f^{(n)}$ is the $n$th iterate
	of the above map $f(x)$. Therefore, $1$ is the unique positive solution.
	On the other hand, under  $\Theta> \Theta_c$, the fixed point $1$ is unstable.
	Iterates $f^{(n)}(x)$ remain for $x>1$, monotonically increasing  and hence converge to a limit, $z^*\geq 1$ which solves
	(\ref{kap}). However, $z^*>1$ as $1$ is unstable. Similarly, for  $x<1$ one gets a solution  $0<z_*<1$.
	This completes the proof.
\end{proof}
Since $\Theta(\theta)$ is a monotone increasing function of $\theta>0$, from $\Theta(\theta)=\Theta_c$ we find our first critical value:
\begin{equation}\label{theta1}
	\theta_{{\rm c},0}\equiv \theta_{{\rm c},0}(k,q):={q-1+\sqrt{k^2+q(q-2)}\over k-1}.
	\end{equation}
Let us formulate the result:

{\bf Result 1.} {\it If $k\geq 2$ and  $\theta> \theta_{{\rm c},0}$ then equation (\ref{fp}) has  three solutions:}
$$x_1=(1, 1, \dots , 1), \ \ x_2=(1,\dots,1; z_*, \dots, z_*), \ \ x_3=(1,\dots,1; z^*, \dots, z^*).$$

Both solutions $z_*$ and $z^*$ are functions of $k\geq 2$, $\theta\geq 	\theta_{{\rm c},0}$ and $q\geq 2$. In Fig. \ref{z*}, we give graphs of these functions for $k=2$ and $q=5$.
\begin{figure}[h]
	\begin{center}
		\includegraphics[width=7.3cm]{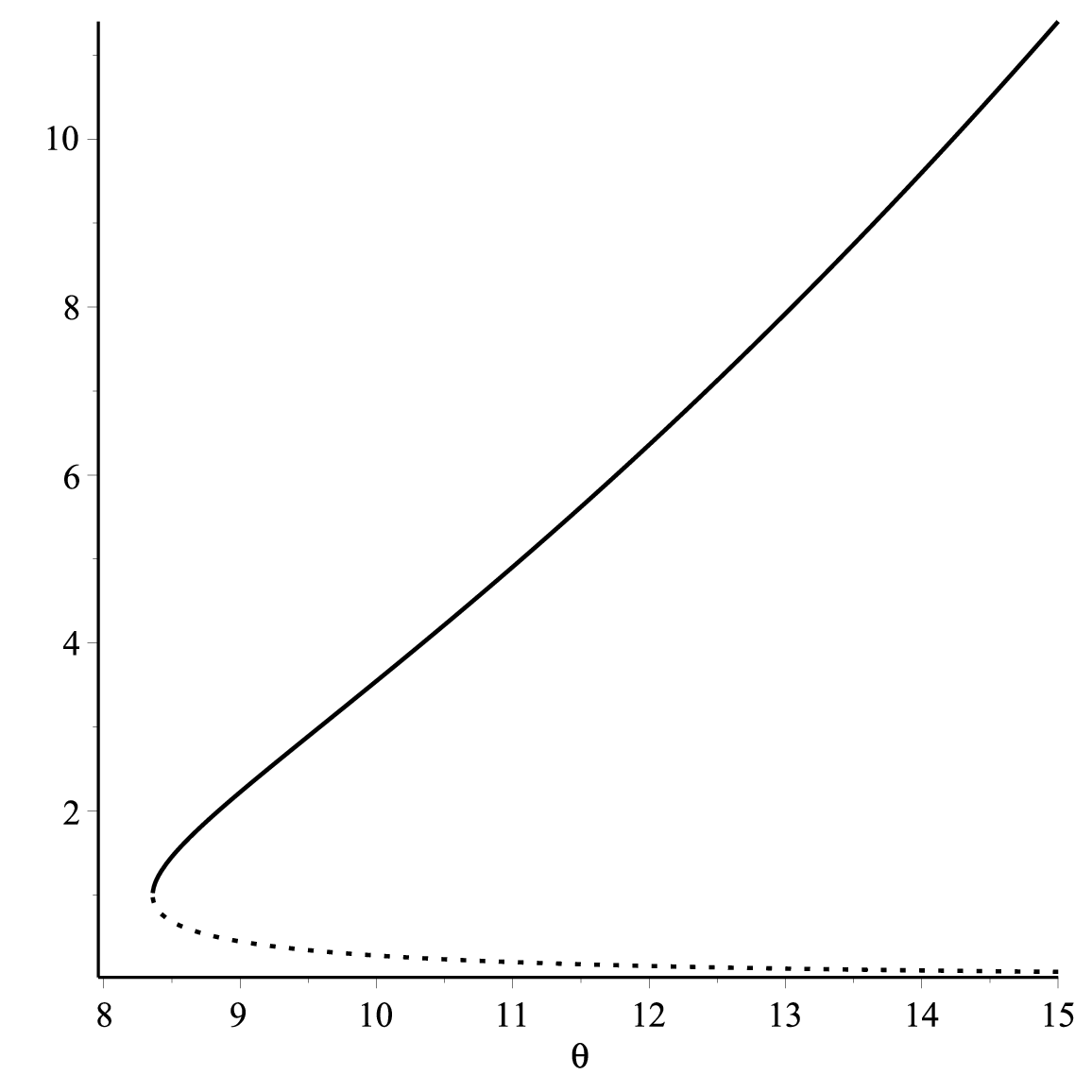}
	\end{center}
	\caption{The graphs of $z_*$ (doted curve) and $z^*$ (solid curve) for $q=5$, $\theta>	\theta_{{\rm c},0}$.}\label{z*}
\end{figure}

{\bf Case $1\leq m\leq q-1$.} In this case from (\ref{fp}), by (\ref{uvq}) we get (as fixed point equation for the operator $F$ given by (\ref{pt-})):
	\begin{equation}\label{te}
	\begin{array}{lll}
		u_i=\left(
		{(1-\theta^{-1})\left(\theta u_i-v_i\right)+\sum\limits_{j=1}^{m}\left(u_j+v_j \right)+(q-m)(v_q+1)\over \theta+q-m-1+(q-m-1+\theta^{-1})v_q+\sum\limits_{j=1}^{m}\left(u_j+v_j \right)}\right)^k,\ \ i=1,\dots,m;\\[3mm]
		v_i=\left(
		{(1-\theta^{-1})\left(\theta v_i-u_i\right)+\sum\limits_{j=1}^{m}\left(u_j+v_j \right)+(q-m)(v_q+1)\over \theta+q-m-1+(q-m-1+\theta^{-1})v_q+\sum\limits_{j=1}^{m}\left(u_j+v_j \right)}\right)^k,\ \ i=1,\dots,m;\\[3mm]
		v_q=\left(
		{(1-\theta^{-1})\left(\theta v_q-1\right)+\sum\limits_{j=1}^{m}\left(u_j+v_j \right)+(q-m)(v_q+1)\over \theta+q-m-1+(q-m-1+\theta^{-1})v_q+\sum\limits_{j=1}^{m}\left(u_j+v_j \right)}\right)^k.
	\end{array}
\end{equation}
But analysis of this system is very complicated.
Let us consider it on invariant sets
for  $M=\{1,\dots,m\}$

Consider $\mathcal J_M$, on this set we have
\begin{equation}\label{sol}
	u_1=u_2=\dots=u_m=u, \ \ v_1=v_2=\dots=v_m=v, \ \ v_q=w.
\end{equation}
The system (\ref{te}) is reduced to
	\begin{equation}\label{tej}
	\begin{array}{lll}
		u=\left(
		{(\theta+m-1)u+(\theta^{-1}+m-1)v+(q-m)w+(q-m)\over mu+mv+(\theta^{-1}+q-m-1)w+ \theta+q-m-1}\right)^k,\\[3mm]
		v=\left(
		{(\theta^{-1}+m-1)u+(\theta+m-1)v+(q-m)w+(q-m)\over mu+mv+(\theta^{-1}+q-m-1)w+ \theta+q-m-1}\right)^k,\\[3mm]
		w=\left(
		{mu+mv+(\theta+q-m-1)w+(\theta^{-1}+q-m-1) \over mu+mv+(\theta^{-1}+q-m-1)w+ \theta+q-m-1}\right)^k.
	\end{array}
\end{equation}
Let $(u,v,w)$ be a solution to (\ref{tej}). Clearly it depends on all parameters $k\geq 2$, $q\geq 2$, $\theta>0$ and $m=0,1,\dots, q-1$.

Denote by $(u(m), v(m), w(m))$ a solution of (\ref{tej}) corresponding to $m$.

\begin{pro}\label{po} If $(u(m_1), v(m_1), w(m_1))$  is a solution to (\ref{tej}) with $m=m_1$ then $({1\over u(m_1)}, {w(m_1)\over u(m_1)}, {v(m_1)\over u(m_1)})$ is a
	solution to (\ref{tej}) with $m=q-m_1$.
\end{pro}

\begin{proof} For solution $(u(m_1), v(m_1), w(m_1))$ we denote
\begin{equation}\label{tilde}
	\tilde u(m_1)={1\over u(m_1)}, \tilde v(m_1)={w(m_1)\over u(m_1)}, \tilde w(m_1)={v(m_1)\over u(m_1)}
\end{equation}
 and show that
these numbers satisfy (\ref{tej}) for $m=q-m_1$. Indeed, in case $m=q-m_1$ the first equation of (\ref{tej}) is
\begin{equation}\label{bir} \tilde u=\left(
{(\theta+q-m_1-1)\tilde u+(\theta^{-1}+q-m_1-1)\tilde v+m_1\tilde w+m_1\over (q-m_1)\tilde u+(q-m_1)\tilde v+(\theta^{-1}+m_1-1)\tilde w+ \theta+m_1-1}\right)^k.
\end{equation}
Replacing $\tilde u$, $\tilde v$ and $\tilde w$ by the values given in (\ref{tilde}) one can see that (\ref{bir}) becomes the first equation of (\ref{tej}) with $m=m_1$. The other equations can be checked similarly.
\end{proof}

Let $M\subset \{1,\dots, q-1\}$, with $|M|=m$. Then we denote the corresponding solution of (\ref{tej})
by
$$r(M):=(u(M), v(M), w(M)).$$
It is clear that the solution  only depends on the cardinality of $M$.

Denote by $\mu_{r(M)}$ the Gibbs measure which corresponds (by Theorem \ref{ep}) to solution $r(M)=(u(M), v(M), w(M))$.

The following proposition gives relation between Gibbs measures corresponding to solutions for $M$ and  $M^c=\{1,\dots,q\}\setminus M$.

\begin{pro}\label{tp} For any finite $\Lambda\subset V$ and any $s_\Lambda\in \{-1,1\}^\Lambda$, $\sigma_\Lambda\in \{1,\dots,q\}^\Lambda$ we have
	\begin{equation}\label{mu}
		\mu_{r(M)}(s_\Lambda,\sigma_\Lambda)=\mu_{r(M^c)}(s_\Lambda,\sigma_\Lambda),
	\end{equation}
	where
	\begin{equation}\label{rM}
		r(M^c)=(u(M^c), v(M^c), w(M^c))=\left({1\over u(M)}, {w(M)\over u(M)}, {v(M)\over u(M)}\right).
		\end{equation}
\end{pro}

\begin{proof} Let cardinality of $M$ be $m$, i.e., $|M|=m$.
	By Proposition \ref{po} we know that if $r(M)$ is a solution to (\ref{tej}) then $r(M^c)$ is a solution to the system for $M^c$, with cardinality $|M^c|=q-m$.
	Let ${\mathbb I}$ be the indicator function.
	
Denote $\partial \Lambda=\{x\in V\setminus \Lambda: \exists y\in \Lambda,  \langle x,y\rangle\}$, then by definition of measure we have
	$$  \mu_{r(M)}(s_\Lambda, \sigma_\Lambda)={\exp\left(-\beta H(s_\Lambda, \sigma_\Lambda)\right)\over Z_\Lambda(r(M))}(u(M))^{\sum_{x\in \partial \Lambda}{\mathbb I}(s(x)=-1,\sigma(x)\in M)}\times $$
	\begin{equation}\label{mum}	(v(M))^{\sum_{x\in \partial \Lambda}{\mathbb I}(s(x)=1,\sigma(x)\in M)}(w(M))^{\sum_{x\in \partial \Lambda}{\mathbb I}(s(x)=1,\sigma(x)\in M^c)},
		\end{equation}
			where
		$$Z_\Lambda(r(M))=\sum_{\rho_\Lambda,\varphi_\Lambda}\exp\left(-\beta H(\rho_\Lambda,\varphi_\Lambda)\right)(u(M))^{\sum_{x\in \partial \Lambda}{\mathbb I}(\rho(x)=-1,\varphi(x)\in M)}\times $$
		$$(v(M))^{\sum_{x\in \partial \Lambda}{\mathbb I}(\rho(x)=1,\varphi(x)\in M)}(w(M))^{\sum_{x\in \partial \Lambda}{\mathbb I}(\rho(x)=1,\varphi(x)\in M^c)}.$$
	Now we use the following formula:
	$$\sum_{x\in \partial \Lambda}{\mathbb I}(s(x)=\epsilon,\sigma(x)\in M)=|\partial \Lambda|-\sum_{x\in \partial \Lambda}{\mathbb I}(s(x)=\epsilon, \sigma(x)\in M^c).$$
	Then after dividing to constants depending on $|\partial \Lambda|$, the RHS of (\ref{mum}) can be written as
	$${\exp\left(-\beta H(s_\Lambda, \sigma_\Lambda)\right)\over Z_\Lambda(r(M^c))}\left({1\over u(M)}\right)^{\sum_{x\in \partial \Lambda}{\mathbb I}(s(x)=-1,\sigma(x)\in M^c)}\times $$
$$	\left({w(M)\over u(M)}\right)^{\sum_{x\in \partial \Lambda}{\mathbb I}(s(x)=1,\sigma(x)\in M^c)}\left({v(M)\over u(M)}\right)^{\sum_{x\in \partial \Lambda}{\mathbb I}(s(x)=1,\sigma(x)\in M)}
=\mu_{r(M^c)}(s_\Lambda, \sigma_\Lambda).$$
\end{proof}
It is convenient to write a solution to (\ref{fp}) belonging  in $\mathcal J_M$ and satisfying (\ref{sol}) as two-row matrix:
\begin{equation}\label{mat}\left(\begin{array}{cccccccc}
	u & u & \dots & u & 1 & 1 & \dots & 1\\
	v & v & \dots & v & w & w & \dots & w\end{array}\right),\end{equation}
where $u$ (and $v$) is written $m$ times, and 1 (and $w$)  $q-m$ times.

The following is a corollary of Propositions \ref{po} and \ref{tp}.
\begin{cor}\label{c1} TISGM corresponds to a solution of (\ref{tej}) with some $m\leq [q/2]$ coincides with a TISGM corresponding to a solution of (\ref{tej}) with $m$ replaces by $q-m$. (Therefore, we only consider the case $m\leq [q/2]$.)
	Moreover, for a given $m\leq [q/2]$,
	a fixed solution $(u, v, w)$ to (\ref{tej}) generates ${q\choose m}$ solutions to (\ref{fp}) by permuting columns of (\ref{mat}) and, in case $(u-v)^2+(w-1)^2\ne 0$, exchanging rows generates  $2{q\choose m}$ TISGMs.
\end{cor}

Denote by $\mathcal S_m$ the set of all solutions of system (\ref{tej}) for fixed $m$.

\begin{pro} A vector $(u,v,w)$ is a common solution of the system (\ref{tej}) for fixed $k$, $q$ and $\theta$, but distinct values of $m$, if and only if
	$$(u,v,w)\in \{(1, 1, 1), (1, z_*, z_*), (z_*, 1, z_*), (1, z^*, z^*), (z^*, 1, z^*)\},$$
	where $z_*, z^*$ are solutions to (\ref{kap}).
	 That is
	$$\bigcap_{m=0}^{[q/2]} \mathcal S_m=\mathcal S_0.$$
\end{pro}
\begin{proof} Denote by $(u(m), v(m), w(m))$ a solution of (\ref{tej}) corresponding to $m$.
		
	Rewrite (\ref{tej}) as
		\begin{equation}\label{jm}
		\begin{array}{lll}
			u=\left(
			{(u+v-w-1)m+(\theta-1)u+(\theta^{-1}-1)v+q(w+1)\over (u+v-w-1)m+(\theta^{-1}+q-1)w+ \theta+q-1}\right)^k=\left({Am+\tilde B\over Am+C}\right)^k=\phi(m),\\[3mm]
			v=\left(
			{(u+v-w-1)m+(\theta^{-1}-1)u+(\theta-1)v+q(w+1)\over (u+v-w-1)m+(\theta^{-1}+q-1)w+ \theta+q-1}\right)^k=\left({Am+\hat B\over Am+C}\right)^k=\psi(m),\\[3mm]
			w=\left(
			{(u+v-w-1)m+(\theta+q-1)w+\theta^{-1}+q-1\over (u+v-w-1)m+(\theta^{-1}+q-1)w+ \theta+q-1}\right)^k=\left({Am+B\over Am+C}\right)^k=\varphi(m).
		\end{array}
	\end{equation}

From (\ref{jm}) it is clear that if $u+v=w+1$ (i.e., $A=0$) then system does not depend on $m$. Moreover, the system coincides with the case $m=0$. Therefore, all solutions $(u,v,w)$ satisfying $u+v=w+1$ in $\mathcal S_m$ for any $m=0,1, \dots, [q/2].$ Now we describe all such solutions. That is we find solutions of (\ref{jm}) which satisfy $u+v=w+1$. Then (\ref{jm}) becomes:
\begin{equation}\label{js}
	\begin{array}{llll}
		u=\left(
		{(\theta-1)u+(\theta^{-1}-1)v+q(w+1)\over (\theta^{-1}+q-1)w+ \theta+q-1}\right)^k,\\[3mm]
		v=\left(
		{(\theta^{-1}-1)u+(\theta-1)v+q(w+1)\over (\theta^{-1}+q-1)w+ \theta+q-1}\right)^k,\\[3mm]
		w=\left(
		{(\theta+q-1)w+\theta^{-1}+q-1\over (\theta^{-1}+q-1)w+ \theta+q-1}\right)^k,\\[3mm]
		v=w-u+1.
	\end{array}
\end{equation}
The third equation is (\ref{kap}), which is already discussed in Lemma \ref{m0}.
It has unique solution $w=1$ if $\theta\leq \theta_{{\rm c},0}$ and three solutions $1$, $z_*$ and $z^*$ if $\theta>\theta_{{\rm c},0}$. For a solutions $w^*\in\{1, z_*, z^*\}$, using $v=w^*-u+1$ from the first equation of (\ref{js}) we get
\begin{equation}\label{um} u=\left(
{(\theta-\theta^{-1})u+(\theta^{-1}+q-1)(w^*+1)\over (\theta^{-1}+q-1)w^*+ \theta+q-1}\right)^k.
\end{equation}
Rewrite the last equation as
\begin{equation}\label{mm} az^k-z+b=0,
\end{equation}
where
$$z=u^{1/k},  \ \ a={\theta-\theta^{-1}\over (\theta^{-1}+q-1)w^*+ \theta+q-1}, \ \ b={(\theta^{-1}+q-1)(w^*+1)\over (\theta^{-1}+q-1)w^*+ \theta+q-1}.$$
Since $b>0$, the equation (\ref{mm}) has up to two solutions, one of which is always $z=u=1$. In case when (\ref{mm}) has a solution $z\ne 1$ then it is $z=\sqrt[k]{w^*}$. To see this, one can just put $z=\sqrt[k]{w^*}$ in RHS of (\ref{mm}) and check that it satisfies the equation. This can be done equivalently with equation (\ref{um}): when we put $u=w^*$ the equation (\ref{um}) becomes (\ref{kap}).

To check that a solution of other kind can not be in different classes $\mathcal S_m$, we use
$$\varphi'(m)=k\left({Am+B\over Am+C}\right)^{k-1}{A(C-B)\over (Am+C)^2}.$$
Therefore, if $A(C-B)\ne 0$ then $\varphi(m)$ is a monotone function with respect to $m$, hence $w(m_1)\ne w(m_2)$ for $m_1\ne m_2.$ In case $A(C-B)=0$ we find $A=0$ or $C=B$.

{\it Case $A\ne 0$ and $C=B$}. In this case we show that there is no common solution. Indeed, $C=B$ is
$$(\theta+q-1)w+\theta^{-1}+q-1=(\theta^{-1}+q-1)w+ \theta+q-1 \ \ \Leftrightarrow \ \ (\theta-\theta^{-1})(w-1)=0 \ \ \Leftrightarrow \ \ w=1.$$
Note that $\psi$ and $\phi$ are similar to $\varphi$.
For $w=1$ from first and second equations of (\ref{jm}), in order to have $u(m_1)=u(m_2)$ and $v(m_1)=v(m_2)$, for some $m_1\ne m_2$, we get (for $A\ne 0$) that
$$\tilde B=\hat B=C \ \ \Leftrightarrow \ \ (\theta-\theta^{-1})(u-v)=0 \ \ \Leftrightarrow \ \ u=v.$$
Consequently, by $u=v$, $w=1$ from the first equation (since $A=2(u-1)\ne 0$) we get
$$(\theta+\theta^{-1}-2)u+2q=(\theta+\theta^{-1}-2)+2q \ \ \Leftrightarrow \ \ u=1.$$
But $u=1$ makes $A=0$. This completes the proof.
\end{proof}

One can see that $u=v$ and $w=1$ satisfies (invariant) for (\ref{tej}).
Are there some solutions of system with $u\ne v$ or/and $w\ne 1$? To  answer this question we denote $\xi=u-v$ and $\eta=w-1$ then subtracting from the first equation of  (\ref{tej}) the second one, and subtracting 1 from both sides of the third equation we get:
\begin{equation}\label{xi}
	\begin{array}{ll}
	\xi=(\theta-\theta^{-1})\xi \left({A^{k-1}+A^{k-2}B+\dots +B^{k-1}\over D^{k}}\right),\\[3mm]
	\eta=(\theta-\theta^{-1})\eta \left({C^{k-1}+C^{k-2}+\dots +1\over D^{k}}\right),
	\end{array}
	\end{equation}
where
$$A=(\theta+m-1)u+(\theta^{-1}+m-1)v+(q-m)w+(q-m),$$
$$B=(\theta^{-1}+m-1)u+(\theta+m-1)v+(q-m)w+(q-m),$$
$$C=mu+mv+(\theta+q-m-1)w+(\theta^{-1}+q-m-1),$$
$$D=mu+mv+(\theta^{-1}+q-m-1)w+ \theta+q-m-1.$$
From (\ref{xi}) we find relations for $u, v, w$ to make $\xi^2+\eta^2\ne 0$.
\begin{itemize}
\item In case $\xi=0$ and $\eta\ne 0$, from (\ref{xi}) we get condition
\begin{equation}\label{eta}
	(\theta-\theta^{-1})\left(C^{k-1}+C^{k-2}+\dots +1\right)- D^{k}=0.
	\end{equation}
Here $C$ and $D$ are as above, but with $u=v$.
\item In case $\xi\ne 0$ and $\eta=0$ we get
\begin{equation}\label{eta0}
	(\theta-\theta^{-1})\left(A^{k-1}+A^{k-2}B+\dots +B^{k-1}\right)- D^{k}=0.
	\end{equation}
Here $A$, $B$ and $D$ are as above, but with $w=1$.
\item Finally, for $\xi\eta\ne 0$ we have
\begin{equation}\label{exi}
	C^{k-1}+C^{k-2}+\dots +1=A^{k-1}+A^{k-2}B+\dots +B^{k-1}=(\theta-\theta^{-1})^{-1}D^k.
\end{equation}
\end{itemize}
\begin{rk}
	Since $A, B, C, D>0$ the equalities (\ref{eta}), (\ref{eta0}) and (\ref{exi}) may be satisfied only for $\theta>1$.
\end{rk}
Let us now solve (\ref{tej}):

{\bf Case 1:} In case $u=v$ and $w=1$ we have
$$u=\left(
{(\theta+\theta^{-1}+2m-2)u+2(q-m)\over 2mu+\theta+\theta^{-1}+2q-2m-2}\right)^k.$$
Denote $z=\sqrt[k]{u}$, then the last equation can be written as
\begin{equation}\label{zk}
	2mz^{k+1}-(\theta+\theta^{-1}+2(m-1))z^k+
	(\theta+\theta^{-1}+2(q-m-1))z-2(q-m)=0.
	\end{equation}
Independently on parameters the equation (\ref{zk}) has solution $z=1$. Divide the LHS of the equation to $z-1$ and obtain
\begin{equation}\label{zkk}
	mz^k-(\tau-1)\sum_{j=1}^{k-1}z^j+q-m=0,
\end{equation}
where
\begin{equation}\label{tau}
	\tau={1\over 2}(\theta+\theta^{-1}).
	\end{equation}

Since $\theta>0$ we have $\theta+\theta^{-1}>2$, i.e., $\tau>1$. Therefore, (\ref{zkk}) has at most two positive roots.
The equation (\ref{zkk}) coincides with the equation (3.14) of \cite{KRK} where one has to change $\theta$ to $\tau$.
From results of \cite{KRK} one can deduce the following
\begin{lemma}\label{le2} For each $k\geq 2$, $m=1, 2, \dots, [q/2]$ there exists unique $\tau_{{\rm c},m}=\tau_{{\rm c},m}(k)$ such that
	\begin{itemize}
		\item[1.]  The equation (\ref{zk}) has unique solution $z=1$ if $\tau<\tau_{{\rm c},m}$, two positive solutions at $\tau=\tau_{{\rm c},m}$ and three positive solutions for $\tau>\tau_{{\rm c},m}$.
		\item[2.] If $x>0$ is a common
		solution to equation (\ref{zk}) for different values of $m$ then $x=1$.		Moreover,  $x=1$ is a solution to (\ref{zk}) iff $\tau=\tau_c={k+q-1\over k-1}$. For this critical value the second solution of (\ref{zk}) will be 1 iff $q=2m$.
		\item[3.] Critical values satisfy
	\begin{equation}\label{tt}
		 \tau_{{\rm c},1}<\tau_{{\rm c},2}<\dots<\tau_{{\rm c}, [{q\over 2}]-1}<\tau_{{\rm c},[{q\over 2}]}\leq \tau_c={k+q-1\over k-1}.
	\end{equation}
	\end{itemize}
\end{lemma}
Recall $\theta=\exp(J\beta)$ (see (\ref{zt})). So $\theta<1$ iff $J<0$. From
	$$2\tau_{{\rm c},m}=\theta_{{\rm c},m}+{1\over \theta_{{\rm c},m}}$$ we get
	$$\theta_{{\rm c},m}=	\tau_{{\rm c},m}+\sqrt{\tau_{{\rm c},m}^2-1}, \ \ \mbox{if} \ \ J>0$$ and
	$$\tilde \theta_{{\rm c},m}={1\over  \theta_{{\rm c},m}}=	\tau_{{\rm c},m}-\sqrt{\tau_{{\rm c},m}^2-1}, \ \ \mbox{if} \ \ J<0.$$
	
	Therefore we define critical temperatures
\begin{equation}\label{Tcm}
	T_{c,m}={|J|\over \ln(\theta_{{\rm c},m})}, \ \ m=0,\dots, [{q\over 2}].\end{equation}
 Below we give explicit formula of these critical values for $k=2$. But using formula (\ref{theta1}), we can give explicit formula of
 $$T_{c,0}={J\over \ln\left({q-1+\sqrt{k^2+q(q-2)}\over k-1}\right)}, \ \ J>0. $$
 Another critical temperature is found from $\tau=\tau_c={k+q-1\over k-1}$:
 $$T_{cr}={|J|\over \ln\left({k+q-1+\sqrt{q^2+2q(k-1)}\over k-1}\right)}. $$
Since function $\tau$ (defined in (\ref{tau})) is decreasing for $\theta<1$ and increasing for $\theta>1$ one can deduce from (\ref{tt}) that critical values $T_{c,m}$ are monotone with respect to $m$.

  Now using Result 1, Lemma \ref{le2} and Corollary \ref{c1} we summarize obtained results in following theorems
\begin{thm} \label{Thm2}
	For the {\bf anti-ferromagnetic} ($J<0$) coupled Ising-Potts model (\ref{ph}) on the Cayley tree of
	order $k\geq 2$ there are critical temperatures $T_{c,m}\equiv T_{c,m}(k,q)$, $m=1,\dots,[q/2]$
	such that the following statements hold
	
	\begin{itemize}
		\item[1.] There are $[q/2]+1$ critical values such that
		$$T_{c,1}>T_{c,2}>\dots>T_{c,[{q\over 2}]-1}>T_{c,[{q\over 2}]}\geq T_{cr};$$
		\item[2.]
		If $T>T_{c,1}$ then there exists a lest one TISGM;
		
		\item[3.]
		If $T_{c,m+1}<T<T_{c,m}$ for some $m=1,\dots,[{q\over 2}]-1$ then there are at least
		$1+2\sum_{s=1}^m{q\choose s}$ TISGMs.
		
		\item[4.] If $T_{cr}\ne T<T_{c,[{q\over 2}]}$ then there are at least $2^{q}-1$ TISGMs.
		
		\item[5.] If $T=T_{cr}$ then
		the number of TISGMs is at least
		$$\left\{\begin{array}{ll}
			2^{q-1}, \ \ \mbox{if} \ \ q -odd\\[2mm]
			2^{q-1}-{q-1\choose q/2}, \ \ \mbox{if} \ \ q -even.
		\end{array}\right.;$$
		
		\item[6.] If $T=T_{c,m}$, $m=1,\dots,[{q\over 2}]$, \, ($T_{c,[q/2]}\ne T_{cr}$) then  there are at least
		$1+{q\choose m}+2\sum_{s=1}^{m-1}{q\choose s}$ TISGMs.
	\end{itemize}
	
\end{thm}	
\begin{rk} We note that the anti-ferromagnetic Ising model (see \cite[p.26]{Ro}) and the anti-ferromagnetic Potts model (see \cite[Theorem 8.2]{Rbp}) each have a unique TISGM. However, Theorem \ref{Thm2} demonstrates that this uniqueness does not hold (if $T<T_{c,1}$) for the anti-ferromagnetic coupled Ising-Potts model.
\end{rk}
For $J>0$ it is easy to see that $T_{cr}<T_{c,0}$. To compare $T_{c,0}$ with $T_{c,m}$, for $m=1, \dots, [q/2]$, one can
find $m$ from the third equation of (\ref{tej}) and see that $m$ is monotone decreasing function of $\theta$, therefore, $\theta_{{\rm c},m}>\theta_{{\rm c},0}$, i.e.,
$T_{c,0}>T_{c,m}$.

The following theorem summarizes results for the ferromagnetic case.
\begin{thm} \label{Thm3}
	For the {\bf ferromagnetic} ($J>0$) coupled Ising-Potts model (\ref{ph}) on the Cayley tree of
	order $k\geq 2$ there are critical temperatures\footnote{One critical temperature $T_{c,0}$ is more than the anti-ferromagnetic case.} $T_{c,m}\equiv T_{c,m}(k,q)$, $m=0,1,\dots,[q/2]$
	such that the following statements hold
	
	\begin{itemize}
		\item[1.] There are $[q/2]+2$ critical values such that $$T_{c,0}>T_{c,1}>T_{c,2}>\dots>T_{c,[{q\over 2}]-1}>T_{c,[{q\over 2}]}\geq T_{cr};$$
			\item[2.]
		If $T>T_{c,0}$ then there exists a lest one TISGM;
		
		\item[3.]
		If $T_{c,m+1}<T<T_{c,m}$ for some $m=0, 1,\dots,[{q\over 2}]-1$ then there are at least
		$3+2\sum_{s=1}^m{q\choose s}$ TISGMs.
		
		\item[4.] If $T_{cr}\ne T<T_{c,[{q\over 2}]}$ then there are at least $2^{q}+1$ TISGMs.
		
		\item[5.] If $T=T_{cr}$ then
		the number of TISGMs is at least
		$$\left\{\begin{array}{ll}
			2^{q-1}+2, \ \ \mbox{if} \ \ q -odd\\[2mm]
			2^{q-1}-{q-1\choose q/2}+2, \ \ \mbox{if} \ \ q -even.
		\end{array}\right.;$$
		
		\item[6.] If $T=T_{c,m}$, $m=1,\dots,[{q\over 2}]$, \, ($T_{c,[q/2]}\ne T_{cr}$) then  there are at least
		$3+{q\choose m}+2\sum_{s=1}^{m-1}{q\choose s}$ TISGMs.
	\end{itemize}
	\end{thm}	

\subsection{Detailed analysis for $k=2$.}		
In this subsection we are going to give detailed analysis of system (\ref{tej}) for $k=2$, and illustrate our results mainly for $q=5$. We explicitly obtain equations depending on $q$ and $m$, therefore, one can do similar analysis for more simple cases:  $q=2, 3, 4$.
		
	{\bf Result 2.} {\it For $k=2$ explicit solutions to (\ref{zk}) are
	$$z_1=1, \ \ \hat z_{2,3}(m)={(\theta-1)^2\mp \sqrt{D}\over 4m \theta},$$
	where}
	$$D=\theta^4-4\theta^3+(16m^2-16mq+6)\theta^2-4\theta+1.$$
	
Solutions $z_{2,3}$ exist if $D\geq 0$.
It is easy to rewrite $D$, with $\tau$ defined in (\ref{tau}), as
$$D=4(\tau-1-2\sqrt{m(q-m)})(\tau-1+2\sqrt{m(q-m)}).$$
	Since $\tau>1$, the condition $D\geq 0$ is reduced to
$$	\tau\geq 1+2\sqrt{m(q-m)}.$$ Thus critical values of $\theta$ can be found from $\tau= 1+2\sqrt{m(q-m)}$, i.e.,
$$\theta_{{\rm c}, m}=\left\{\begin{array}{ll}
	1+2\sqrt{m(q-m)}-2\sqrt{m(q-m)+\sqrt{m(q-m)}}, \ \ J<0\\[2mm]
	 1+2\sqrt{m(q-m)}+2\sqrt{m(q-m)+\sqrt{m(q-m)}}, \ \ J>0
	 \end{array}\right.$$

 In Fig.\ref{z1234} we give graphs of four functions $\hat z_{2,3}(m)$  for $q=5$ and $m=1$ and $m=2$.
\begin{figure}[h]
	\begin{center}
		\includegraphics[width=8cm]{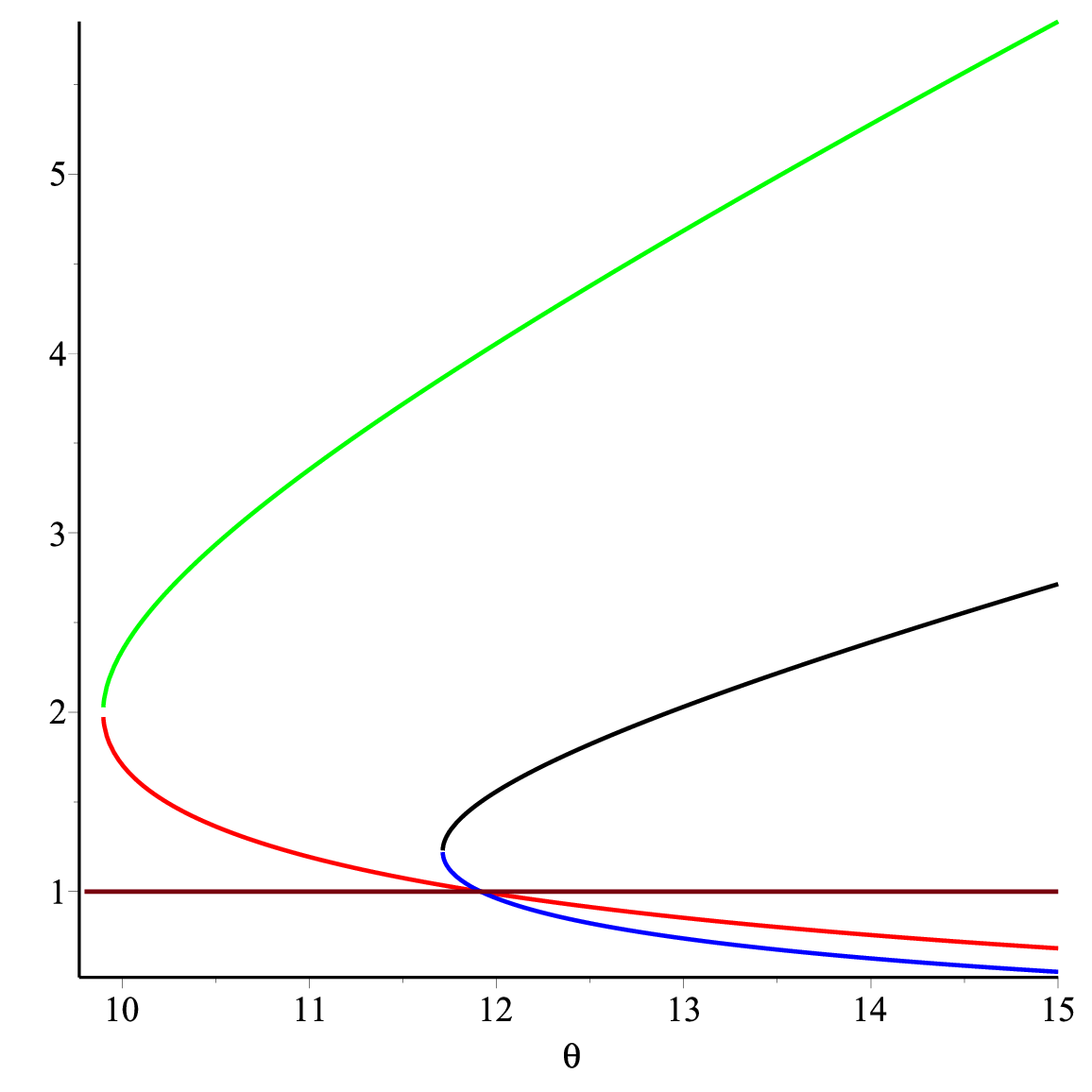}
	\end{center}
	\caption{Case $q=5$. From the left, started at $\theta=\theta_{{\rm c},1}\approx 9.8989$, the graph of $\hat z_{2}(1)$ (red curve), graph of $\hat z_3(1)$ (green curve); and started at $\theta=\theta_{{\rm c},2}\approx 11.7125$, the graph of $\hat z_{2}(2)$ (blue), graph of $\hat z_3(2)$ (black).}\label{z1234}
\end{figure}

{\bf Case 2:} Assume now $u=v$ but $w\ne 1$. Then  from (\ref{tej})
we get	
\begin{equation}\label{ej}
	\begin{array}{lll}
		u=\left(
		{(\theta+\theta^{-1}+2(m-1))u+(q-m)w+(q-m)\over 2mu+(\theta^{-1}+q-m-1)w+ \theta+q-m-1}\right)^k,\\[3mm]
		w=\left(
		{2mu+(\theta+q-m-1)w+(\theta^{-1}+q-m-1) \over 2mu+(\theta^{-1}+q-m-1)w+ \theta+q-m-1}\right)^k.
	\end{array}
\end{equation}
Assume $k=2$ and
denoting $z=\sqrt{u}$, $t=\sqrt{w}$ we rewrite
the last system as:
\begin{equation}\label{2j}
	\begin{array}{lll}
		z={(\theta+\theta^{-1}+2(m-1))z^2+(q-m)t^2+(q-m)\over 2mz^2+(\theta^{-1}+q-m-1)t^2+ \theta+q-m-1},\\[3mm]
		t=	{2mz^2+(\theta+q-m-1)t^2+(\theta^{-1}+q-m-1) \over 2mz^2+(\theta^{-1}+q-m-1)t^2+ \theta+q-m-1}.
	\end{array}
\end{equation}
 Solving the first equation of the (\ref{2j}) with respect to $t^2$ we get
 \begin{equation}\label{t2}
 	t^2={2mz^3-(\theta+\theta^{-1}+2(m-1))z^2 + (\theta+q-m-1)z-(q-m)\over q-m-(\theta^{-1}+q-m-1)z}.
 	\end{equation}
 Now from the second equation of (\ref{2j}) we find $t^2$ (assuming $t\ne 1$, since $w\ne 1$)  as
 \begin{equation}\label{t22}
 	t^2={1\over \theta^{-1}+q-m-1}((\theta-\theta^{-1})t-(2mz^2+\theta^{-1}+q-m-1)).\end{equation}
 From RHS of (\ref{t22}) it is clear that we may have positive solution if $\theta> 1$. Therefore, we consider the case $\theta>1$.

 Substitute $t^2$ from (\ref{t22}) in the LHS of (\ref{t2}) then solving obtained equation with respect to $t$ gives:
 \begin{equation}\label{ttt}
 	t=
 	\frac{[\{(m-q+1)\theta^2+ (q+m-2)\theta+1\}z +(q- m-1)\theta^2 + (q-m)\theta + 1]z}{(\theta + 1)[((m-q+1)\theta-1)z + (q-m)\theta]}.
 	 	\end{equation}
 Substituting this $t$ in the first equation of (\ref{2j})
 we get a quartic  (fourth-degree with respect to $z$) equation:
 \begin{equation}\label{z4}
 	P_4^{(4)}z^4+	P_5^{(3)}z^3+	P_5^{(2)}z^2 +	P_5^{(1)}z+P_4^{(0)}=0,
 	\end{equation}
where $P_n^{(j)}\equiv P_n^{(j)}(\theta,q,m)$ denotes a two-parametric (parameters $q$, $m$) polynomial of degree $n$ in $\theta$ as coefficient of $z^j$.
We have explicitly computed all $P_n^{(j)}$ using a computer; however, we do not present them here due to their cumbersome form. Additionally, solving equation (\ref{z4}) with such coefficients is challenging. Of course, one can use Ferrari's method\footnote{https://en.wikipedia.org/wiki/Quartic$_-$equation} and using computer give explicit solutions. Due to the complexity of the equation's coefficients, the solutions of (\ref{z4}) will be extremely lengthy and cumbersome. Therefore, here we give some solutions of the equation for concrete values of parameters $\theta$, $q$, $m$ with $q>m$.

{\it Case: $q=5$, $m=1$.} Consider the case $q=5$ and $m=1$ then equation (\ref{z4}) has
the form\\
\begin{equation}\label{4e}
	\begin{array}{lll}
(15\theta^4-10\theta^3+20\theta^2+10\theta+1)z^4-(3\theta^5+33\theta^4+24\theta^3-16\theta^2-11\theta-1)z^3\\[3mm]
+(7\theta^5+33\theta^4+58\theta^3+46\theta^2+15\theta+1)z^2-4\theta(\theta^4+8\theta^3+14\theta^2+8\theta+1)z\\[3mm]
+16\theta^2(\theta+1)^2=0.
\end{array}
\end{equation}

Since $\theta>1$ this equation has 4 sign changes
of coefficients, therefore, by Descartes' rule of signs, the equation has up to 4 positive solutions. Moreover, using implicit plot (see Fig. \ref{tort}) one can see the following \\

{\bf Result 3.} {\it There are two critical values  $\theta_c^{(1)}\approx 8.3589$ and $\theta_c^{(2)}\approx 10.3633$ such that
\begin{itemize}
	\item If $\theta<\theta_c^{(1)}$ then the equation (\ref{4e}) does not have any positive solution;
	\item If $\theta_c^{(1)}\leq \theta<\theta_c^{(2)}$ then the equation (\ref{4e}) has two positive solutions;
	\item If $\theta\geq\theta_c^{(2)}$ then the equation (\ref{4e}) has four positive solutions.
\end{itemize}}

Now, for $q=5$, $m=1$, we have to check which solution $z$ of (\ref{4e}), by (\ref{ttt}), defines positive $t$. Then for $\theta\geq\theta_c^{(2)}$, (by (\ref{ttt})) we should check
 \begin{equation}\label{t51}
	t=
	\frac{[(-3\theta^2+ 4\theta+1)z +3\theta^2 + 4\theta + 1]z}{(\theta + 1)[(-3\theta-1)z + 4\theta]}>0.
\end{equation}
Denote
$$\phi(\theta)={3\theta^2+4\theta+1\over 3\theta^2-4\theta-1}, \ \ \psi(\theta)={4\theta\over 3\theta+1}.$$
Solving the inequality  (\ref{t51}), with respect to $z$, we get
\begin{equation}\label{min}
	z\in (0, \min\{\phi(\theta), \psi(\theta)\})\cup (\max\{\phi(\theta), \psi(\theta)\}, +\infty)
\end{equation} 	
By Fig.\ref{tort} one can see that each function $z(\theta)$ (black curves)  satisfy condition (\ref{min}).
\begin{figure}[h]
	\begin{center}
		\includegraphics[width=9cm]{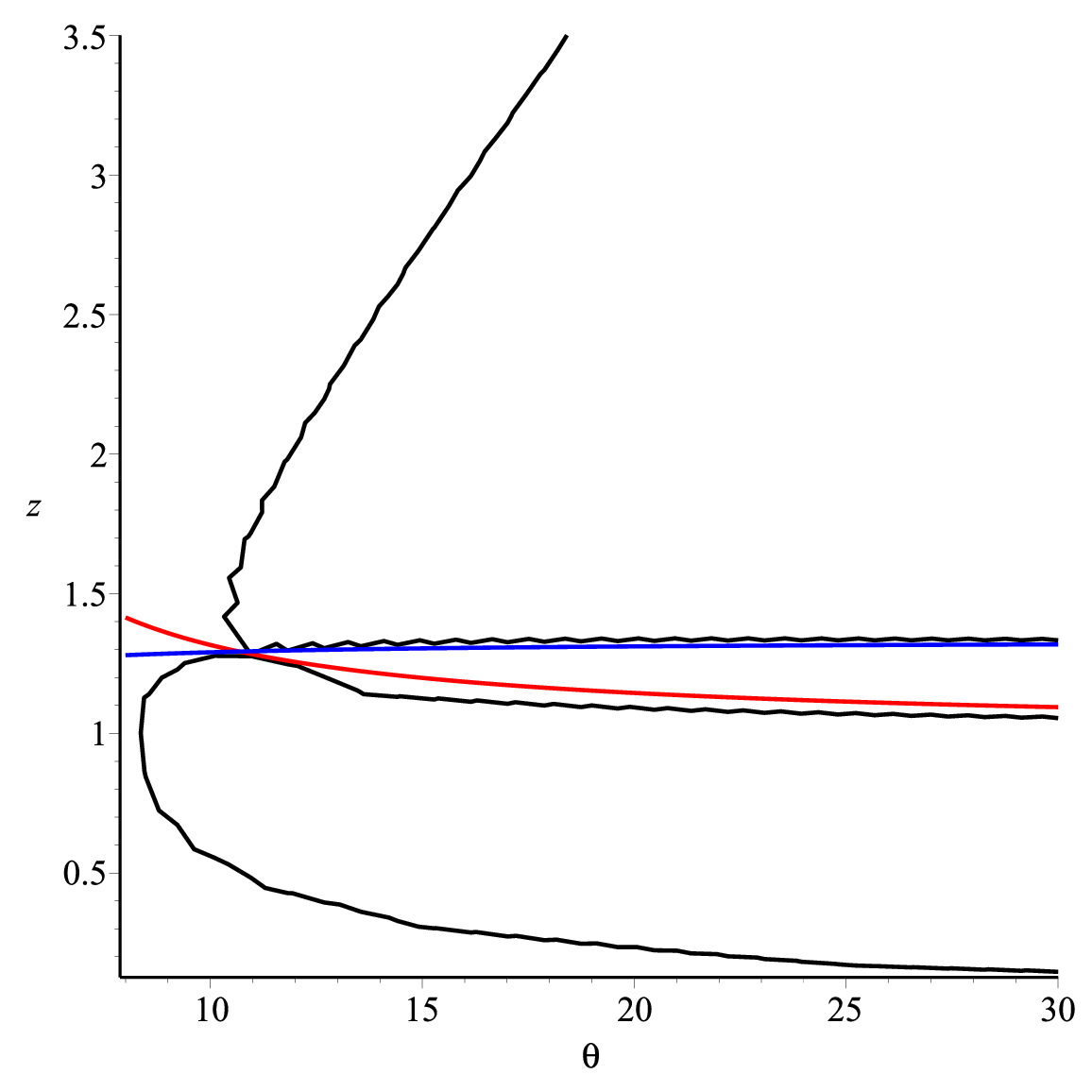}
	\end{center}
	\caption{The implicit plot of~\eqref{4e} the black curves. There graph started at $\theta\approx 8.3589$ and the upper branch started at $\theta\approx 10.3633$. Graph of $\phi$ is red, the graph of $\psi$ is blue.}\label{tort}
\end{figure}

{\it Case: $q=5$, $m=2$.}
Consider the case $q=5$ and $m=2$ then equation (\ref{z4}) has
the form\\
	\begin{equation}\label{4e2}
	\begin{array}{lll}
 \left(12 \theta^4 + 37 \theta^2 + 14 \theta + 1 \right) z^4
	- \left(\theta^5 +29 \theta^4 + 32 \theta^3 - 4 \theta^2 - 10 \theta - 1 \right) z^3\\[2mm]
+ \left(5 \theta^5 + 25 \theta^4 + 44 \theta^3 + 34 \theta^2 + 11 \theta + 1 \right) z^2\\[2mm]
-3\theta\left(\theta^4+6\theta^3 +10\theta^2+6\theta+1\right) z
	+ 9 \theta^2(\theta + 1)^2=0.
\end{array}
\end{equation}

Since $\theta>1$ this equation has 4 sign changes
of coefficients, therefore it has up to 4 positive solutions. Moreover, using implicit plot (see Fig. \ref{rt}) one can see the following \\

{\bf Result 4.} {\it There are two critical values  $\theta_c^{(3)}\approx 8.3376$ and $\theta_c^{(4)}\approx 11.9871$ such that
	\begin{itemize}
		\item If $\theta<\theta_c^{(3)}$ then the equation (\ref{4e}) does not have any positive solution;
		\item If $\theta_c^{(3)}\leq \theta<\theta_c^{(4)}$ then the equation (\ref{4e2}) has two positive solutions;
		\item If $\theta\geq\theta_c^{(4)}$ then the equation (\ref{4e2}) has four positive solutions.
\end{itemize}}

Now, for $q=5$, $m=2$, we have to check which solution $z$ of (\ref{4e}), by (\ref{ttt}), defines positive $t$. Then for $\theta\geq\theta_c^{(3)}$, we should check
\begin{equation}\label{t52}
	t=
	\frac{[(-2\theta^2+ 5\theta+1)z +2\theta^2 + 3\theta + 1]z}{(\theta + 1)[(-2\theta-1)z + 3\theta]}>0.
\end{equation}
Denote
$$a(\theta)={2\theta^2+3\theta+1\over 2\theta^2-5\theta-1}, \ \ b(\theta)={3\theta\over 2\theta+1}.$$
Solving the inequality (\ref{t52}),  with respect to $z$, we get
\begin{equation}\label{max}
	z\in (0, \min\{a(\theta), b(\theta)\})\cup (\max\{a(\theta), b(\theta)\}, +\infty)
\end{equation} 	
By Fig.\ref{rt} one can see that each function $z(\theta)$ (black curves)  satisfy condition (\ref{max}).

	\begin{figure}[h]
		\begin{center}
			\includegraphics[width=9cm]{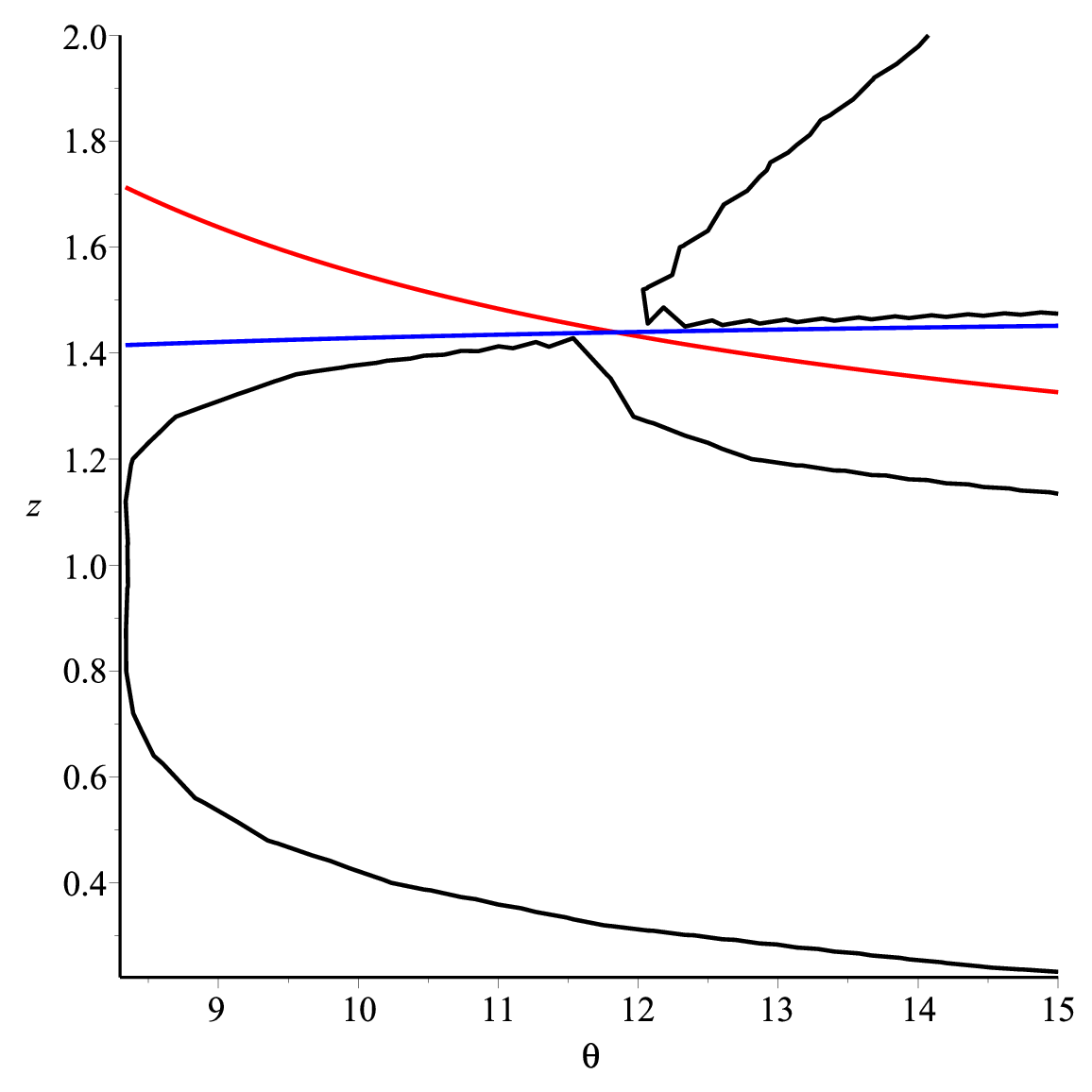}
		\end{center}
		\caption{The implicit plot of~\eqref{4e2} the black curves. The graph is started at $\theta\approx 8.3376$ and the upper branch started at $\theta\approx 11.9871$. Graph of function $a(\theta)$ (red) and $b(\theta)$ (blue).}\label{rt}
	\end{figure}
{\bf Case 3:} Assume now $u\ne v$ but $w=1$. Again we consider the case $k=2$ and denoting $z=\sqrt{u}$, $s=\sqrt{v}$ from (\ref{tej}) we get	
	\begin{equation}\label{tuv}
	\begin{array}{lll}
		z=	{(\theta+m-1)z^2+(\theta^{-1}+m-1)s^2+2(q-m)\over mz^2+ms^2+\theta^{-1}+\theta+2(q-m-1)},\\[3mm]
		s=	{(\theta^{-1}+m-1)z^2+(\theta+m-1)s^2+2(q-m)\over mz^2+ms^2+\theta^{-1}+\theta+2(q-m-1)}.
	\end{array}
\end{equation}
Subtracting from the first equation the second one (since $z\ne s$) we get
\begin{equation}\label{zs}
	(\theta-\theta^{-1})(z+s)=m(z^2+s^2)+\theta^{-1}+\theta+2(q-m-1).
\end{equation}
Now add equations of (\ref{tuv}):
\begin{equation}\label{zs2}
	z+s={(\theta+\theta^{-1}+2(m-1))(z^2+s^2)+4(q-m)\over m(z^2+s^2)+\theta^{-1}+\theta+2(q-m-1)}.
\end{equation}
Denoting
\begin{equation}\label{gz}
	g=z+s, \ \ h=z^2+s^2.
	\end{equation}
from (\ref{zs}) and (\ref{zs2}) we obtain (recall that $\theta>1$):
\begin{equation}\label{gh}
\left\{	\begin{array}{ll}
	g=(\theta-\theta^{-1})^{-1}\left(mh+\theta^{-1}+\theta+2(q-m-1)\right)\\[2mm]
	g={(\theta+\theta^{-1}+2(m-1))h+4(q-m)\over mh+\theta^{-1}+\theta+2(q-m-1)}.
\end{array}\right.
\end{equation}
The system yields a quadratic equation in $h$ (also in $ g $), where the coefficients depend on three parameters: $\theta$, $q$, and $m$. One can provide both solutions to the equation, along with the conditions on these parameters that ensure the solutions are positive. However, in the general case, the resulting expressions are quite lengthy.

Therefore, we will perform these computations specifically for $q = 5$.

{\it Case: $q=5$, $m=1$.} Then for $h$ we have to solutions:
\begin{equation}\label{h5}
	h_{1,2}(\theta)={\theta^4 - 2\theta^3 - 12\theta^2 - 2\theta - 1 \pm (\theta-1)\sqrt{(\theta+1)\mathcal D_1}\over 2\theta^2}
		\end{equation}
	where
	$$\mathcal D_1=\theta^5 - 3\theta^4 - 26\theta^3 + 30\theta^2 +5\theta + 1.$$
	Since $\theta>1$ by plot $\mathcal D_1$ and $h_{1,2}$ one can see that $h_{1,2}$ are positive  iff  $$\theta>\theta_c^{(5)}\approx 6.336 .$$
	Corresponding positive $g_{1,2}$ has the form
	\begin{equation}\label{h53}
		g_{1,2}(\theta)={\theta^3+\theta^2+\theta+1 \pm \sqrt{(\theta+1)\mathcal D_1}\over 2(\theta+1)\theta}.
	\end{equation}
We note that if positive $g$ and $h$ exist then,
for $\theta$ satisfying $2h-g^2\geq 0$, the
corresponding to them  positive solutions $z$ and $s$ to  (\ref{gz}) are
$$	z_{1,2}={1\over 2}(g\pm \sqrt{2h-g^2}), \ \ s_{1,2}={1\over 2}(g\mp \sqrt{2h-g^2}).$$

By computer computations one can see that $2h_1(\theta)-g_1^2(\theta)>  0$  iff
$$\theta_c^{(5)}< \theta <\hat\theta_c^{(5)}\approx 6.4174,$$
and  $2h_2(\theta)-g_2^2(\theta)>0 $ iff
$\theta>\theta_c^{(5)}$.
Thus we have the following result

{\bf Result 5.} {\it From (\ref{gz}) we get the following  explicit solutions:}
\begin{itemize}
	\item If $\theta_c^{(5)}< \theta <\hat\theta_c^{(5)}$ then there are 4 solutions:
\begin{equation}\label{zs12}
	\begin{array}{ll}
	z_{1,2}={1\over 2}(g_1\pm \sqrt{2h_1-g_1^2}), \ \ s_{1,2}={1\over 2}(g_1\mp \sqrt{2h_1-g_1^2}),\\[2mm]
		z_{3,4}={1\over 2}(g_2\pm \sqrt{2h_2-g_2^2}), \ \ s_{3,4}={1\over 2}(g_2\mp \sqrt{2h_2-g_2^2}).
	\end{array}
	\end{equation}
	\item If $\hat\theta_c^{(5)}< \theta$ the there are 2 solutions:
$$	z_{3,4}={1\over 2}(g_2\pm \sqrt{2h_2-g_2^2}), \ \ s_{3,4}={1\over 2}(g_2\mp \sqrt{2h_2-g_2^2}).$$
	\end{itemize}

{\it Case: $q=5$, $m=2$}. Then for $h$ we have two solutions:
\begin{equation}\label{h6}
	\hat h_{1,2}={\theta^4 - 2\theta^3 - 16\theta^2 - 6\theta - 1 \pm (\theta-1)\sqrt{(\theta+1)(\theta^5 - 3\theta^4 - 46\theta^3 + 66\theta^2 +13\theta + 1)}\over 8\theta^2}.
\end{equation}
Since $\theta>1$ positive solutions exist iff $$\theta>\theta_c^{(6)}\approx 7.7897 .$$
Corresponding $g_{1,2}$ has the form
\begin{equation}\label{h53}
	\hat g_{1,2}={(\theta+1)^3 \pm \sqrt{(\theta+1)(\theta^5 - 3\theta^4 - 46\theta^3 + 66\theta^2 +13\theta + 1)}\over 4(\theta+1)\theta}.
\end{equation}

By computer computations one can see that $2\hat h_1(\theta)-\hat g_1^2(\theta)>  0$  iff
$$\theta_c^{(6)}< \theta <\hat\theta_c^{(6)}\approx 8.3589,$$
and  $2\hat h_2(\theta)-\hat g_2^2(\theta)>0 $ iff
$\theta>\theta_c^{(6)}$.
Thus we have the following result

{\bf Result 6.} {\it From (\ref{gz}) we get the following results:}
\begin{itemize}
	\item If $\theta_c^{(6)}< \theta <\hat\theta_c^{(6)}$ then there are 4 solutions:
\begin{equation}\label{zs12}
	\begin{array}{ll}
		\hat z_{1,2}={1\over 2}(\hat g_1\pm \sqrt{2\hat h_1-\hat g_1^2}), \ \ \hat s_{1,2}={1\over 2}(\hat g_1\mp \sqrt{2\hat h_1-\hat g_1^2}),\\[2mm]
		\hat z_{3,4}={1\over 2}(\hat g_2\pm \sqrt{2\hat h_2-\hat g_2^2}), \ \ \hat s_{3,4}={1\over 2}(\hat g_2\mp \sqrt{2\hat h_2-\hat g_2^2}).
	\end{array}
\end{equation}
	\item If $\hat\theta_c^{(6)}< \theta$ then there are 2 solutions:
\begin{equation}\label{zs122}
			\hat z_{3,4}={1\over 2}(\hat g_2\pm \sqrt{2\hat h_2-\hat g_2^2}), \ \ \hat s_{3,4}={1\over 2}(\hat g_2\mp \sqrt{2\hat h_2-\hat g_2^2}).
\end{equation}
\end{itemize}

\begin{figure}[h]
	\begin{center}
		\includegraphics[width=11cm]{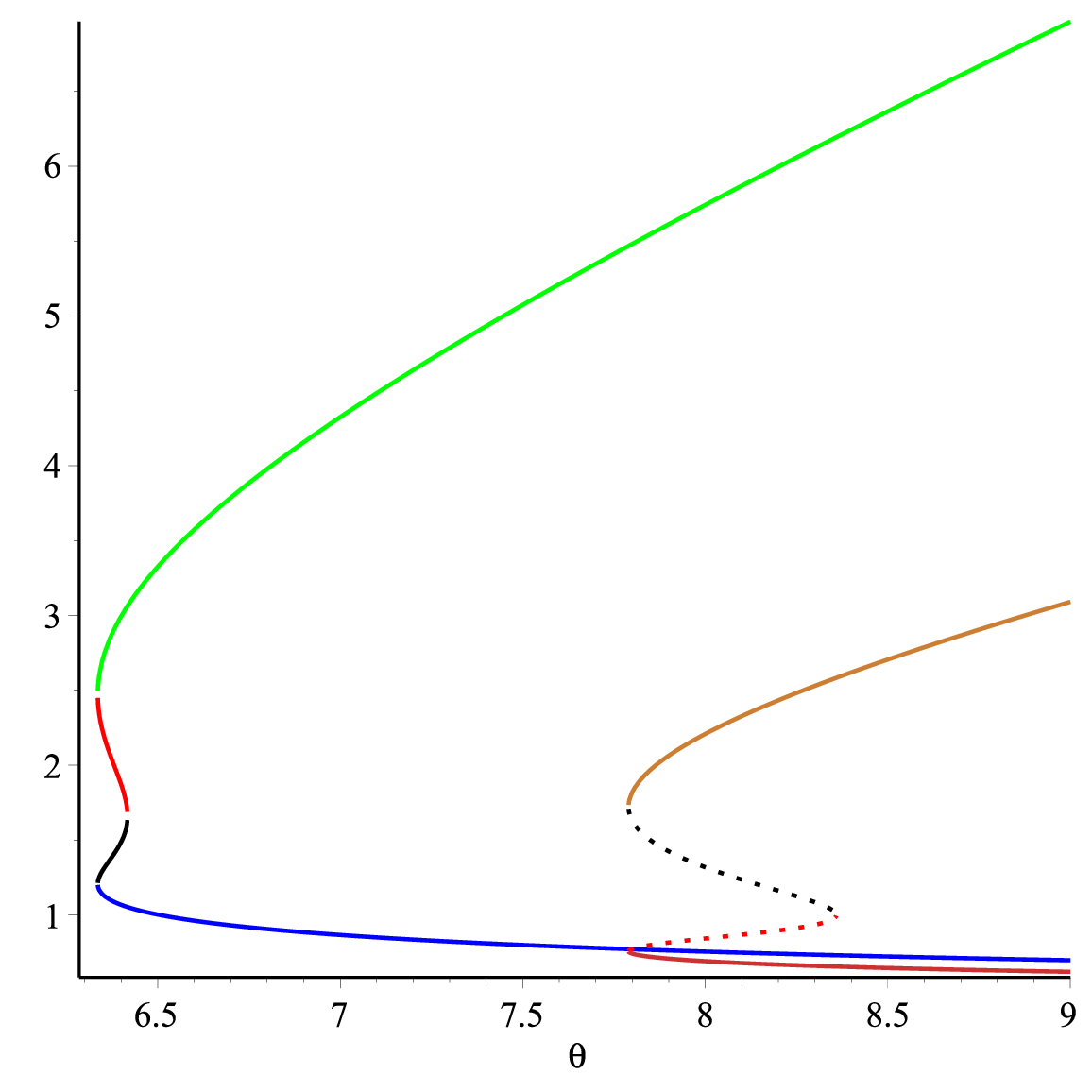}
	\end{center}
	\caption{The graphs of functions (see Result 5 and 6): for $\theta>6.336$ the function $z_1$ (red), $z_2$ (black), $z_3$ (green), $z_4$ (blue); for $\theta>7.7889$ the function $\hat z_1$ (black dots), $\hat z_2$ (red dots), $\hat z_3$ (gold) and $\hat z_4$ (orange). }\label{z1-8}
\end{figure}
	\begin{rk} In case of Ising and Potts model if a Gibbs measure exists for some $T_0$ then it exists for all $T<T_0$.  The Results 5 and 6 show one of interesting  property for coupled Ising-Potts model: some Gibbs measures may only exist in a bounded interval of temperatures (they are Gibbs measures corresponding to functions $z_1$, $z_2$, $\hat z_1$, $\hat z_2$ see Fig. \ref{z1-8}).
		\end{rk}
{\bf Case 4:} Assume now $u\ne v$ and $w\ne1$.
Then from (\ref{exi}) for $k=2$ we get $C+1=A+B=(\theta-\theta^{-1})^{-1}D^2$. The equation $C+1=A+B$ is
\begin{equation}\label{uvw}
	m(u+v)+(\theta+q-m-1)w+\theta^{-1}+q-m
=(\theta+\theta^{-1}+2(m-1))(u+v)+2(q-m)(w+1).
\end{equation}
Solve equation (\ref{uvw}) with respect to $u+v$:
$$u+v={(\theta-(q-m+1))\theta w-(q-m)\theta+1\over (\theta-1)^2+m\theta}.$$
Substitute this $u+v$ in the third equation of (\ref{tej}) for $k=2$, and since $w\ne 1$ we obtain a quadratic equation which has two solutions:
$$w_{1,2}=
\frac{1}{2} \cdot \frac{P \mp (\theta+1)((\theta-1)^2+\theta m)\sqrt{\mathcal D}}{(\theta - 1)^2((q-1)\theta + 1)^2},
$$
where
$$P=\theta^6+2(m-1)\theta^5+(m^2+2mq-2q^2-2m+4q-3)\theta^4$$
$$+2(m^2-mq+2q^2-6q+6)\theta^3+(m^2-2mq-2q^2+12q-13)\theta^2-2(2q-3)\theta - 1$$
and
$$\mathcal D=\theta^6+2(m-1)\theta^5+(m^2+4mq-4q^2-4m+8q-5)\theta^4$$ $$+2(m^2-2mq+4q^2+2m-12q+10)\theta^3+(m^2-4mq-4q^2+24q-25)\theta^2-2(4q+m-7)\theta-3.$$
Thus positive solutions exists for parameters satisfying \begin{equation}\label{dp}
\mathcal D\geq 0, \ \ P> (\theta+1)((\theta-1)^2+\theta m)\sqrt{\mathcal D}.
\end{equation}
Using these solutions
from the first equation of (\ref{tej}) we get the following quadratic equation for $u$:
\begin{equation}\label{U}
	u=\left({u+\mathcal W\over
		1+\theta^{-1}+\mathcal W}\right)^2,
\end{equation}
where $\mathcal W$ is one of
$$\mathcal W_{1,2}={(\theta^{-1}+q-1)w_{1,2}+q-m\over \theta-\theta^{-1}}.$$
The equation (\ref{U}) is
$$u^2-[(\mathcal W+{1\over \theta})^2+{2\over \theta}+1]u+\mathcal W^2=0.$$
If $w_{1,2}$ exists than the last equation has two positive solutions, because its discriminant is
$$
\left((\mathcal W-1)^2+{1\over \theta^2}(2\theta(\mathcal W+1)+1)\right)\left((\mathcal W+{1\over \theta})^2+{2\over \theta}+1+2\mathcal W\right)>0$$
and by Vieta's formulas applied to quadratic polynomial we see that both solutions are positive. Denote by
$u_{1}$, $u_2$ (resp. $u_3$, $u_4$) the positive solutions corresponding to $w_1$ (resp. $w_2$). Since $u$ and $v$ are symmetric in (\ref{tej}), similar argument shows that there are
positive solutions $v_{i}$, $i=1,2,3,4$ corresponding to $w_1$ and $w_2$ too.
Summarize this to

{\bf Result 7.} If parameters satisfy conditions (\ref{dp}) then
the system of equations (\ref{tej}) has four positive solutions:
$$(u_i,v_i,w_1), i=1,2; \ \ (u_i,v_i,w_2), i=3,4.\\$$

 To see simpler form of condition (\ref{dp}) we consider $q=5$.

 {\it Case: $q=5$, $m=1$}. Then
 $$P=\theta^6-24\theta^4+44\theta^3-12\theta^2-14\theta-1.$$
 $$\mathcal D=\theta^6-48\theta^4+86\theta^3
 -24\theta^2-28\theta-3.$$
 By plotting of graph of $\mathcal D$ for $\theta>1$ one can see that there is $\theta_c^{(7)}\approx 5.8416$ such that $\mathcal D>0$ iff $\theta>\theta_c^{(7)}$. Moreover, for such values of $\theta$, $w_1$ and $w_2$ are positive (see Fig. \ref{w12}). For graphs of $u_i$, $i=1,2,3,4$ see Fig. \ref{u1234}.
 \begin{figure}[h]
 	\begin{center}
 		\includegraphics[width=8cm]{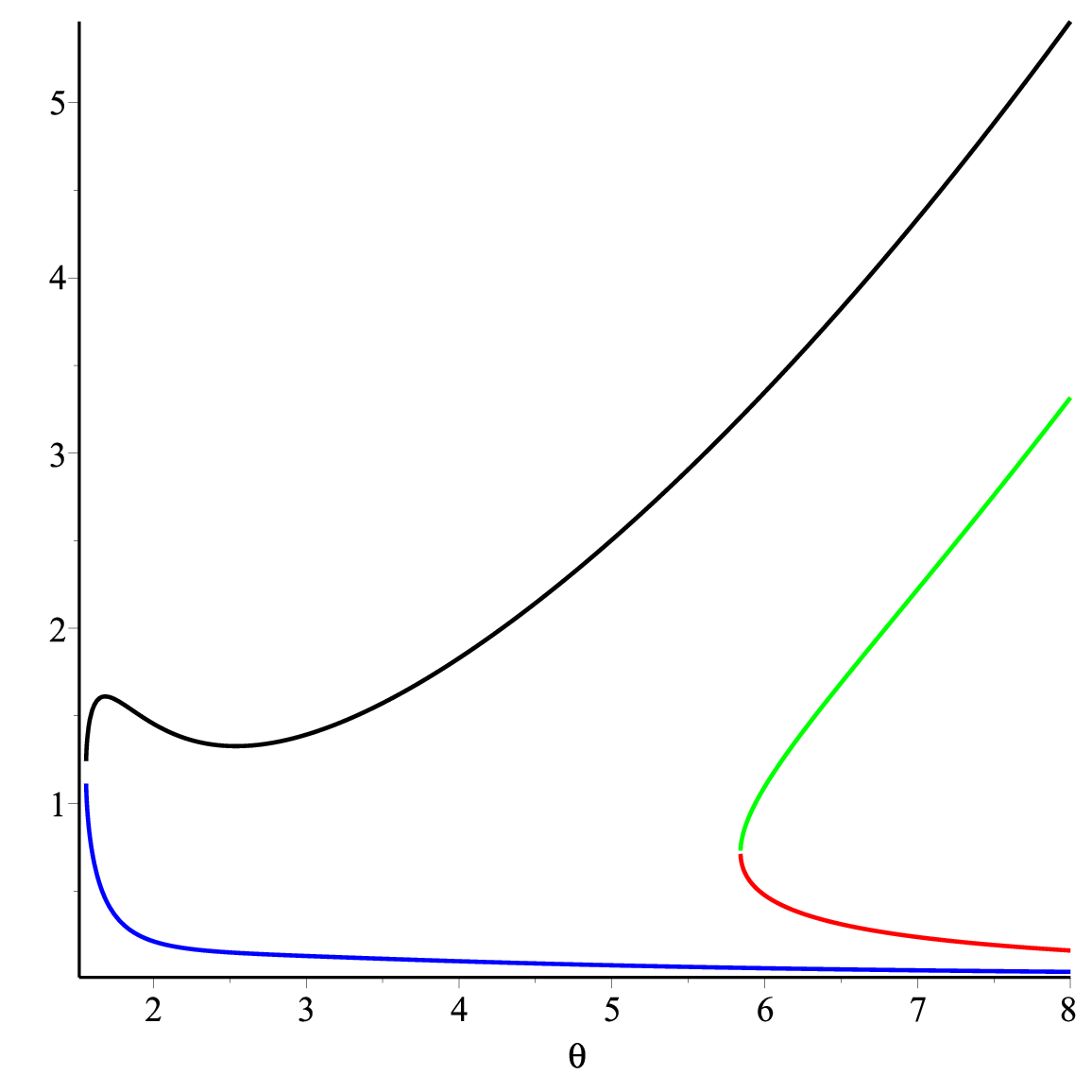}
 	\end{center}
 	\caption{The graphs of $w_i$, $i=1,2$ for $q=5$. {\it Case} $m=1$: $w_1$ (red), $w_2$ (green). {\it Case $m=2$}: $w_1$ (blue) and $w_2$ (black).}\label{w12}
 \end{figure}
\begin{figure}[h]
	\begin{center}
		\includegraphics[width=8cm]{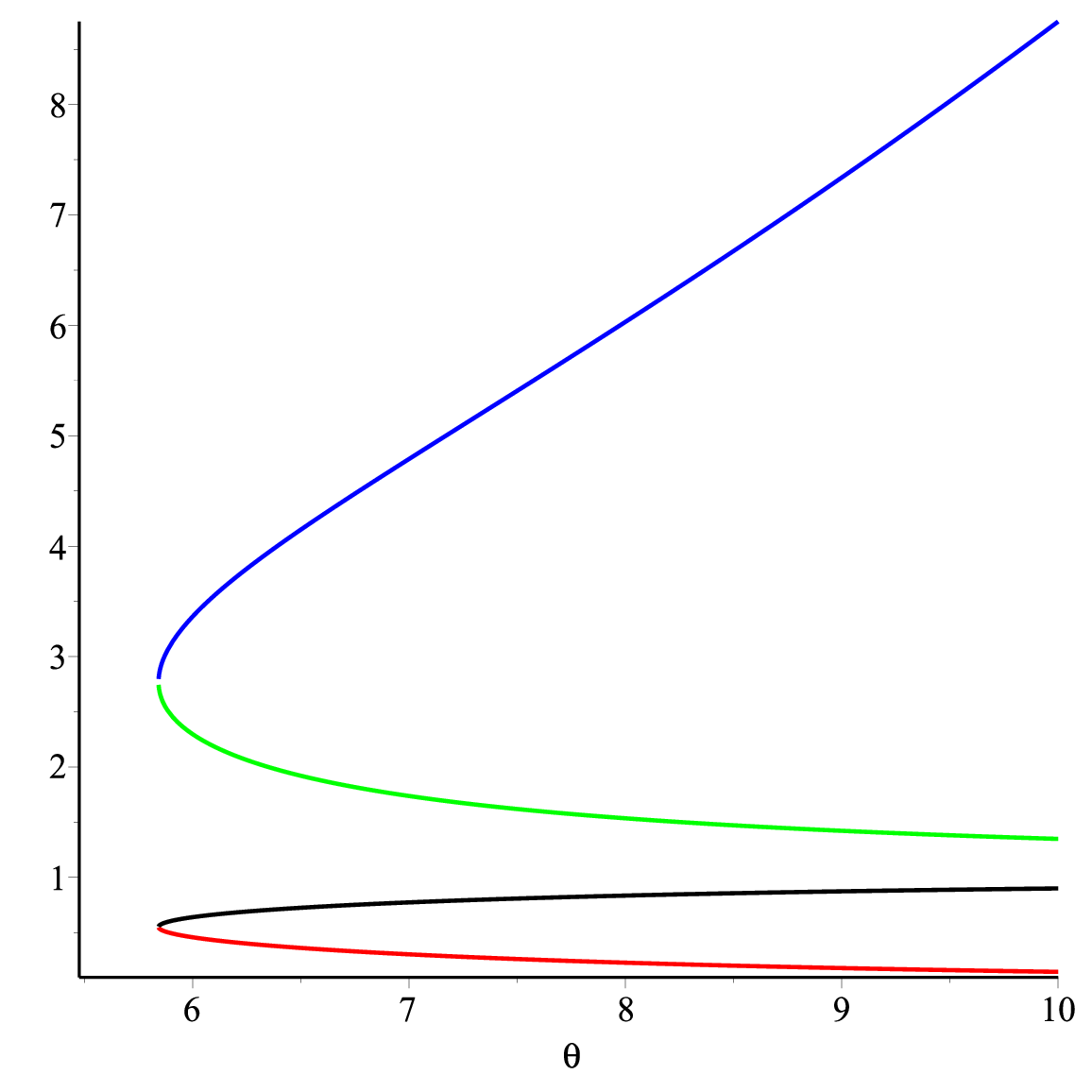}
	\end{center}
	\caption{The graphs of $u_1$ (red), $u_2$ (green), $u_3$ (black) and $u_4$ (blue) for $q=5$, $m=1$, $\theta>\theta_c^{(7)}$.}\label{u1234}
\end{figure}

{\it Case: $q=5$, $m=2$}. Then
$$P=\theta^6+2\theta^5-13\theta^4+40\theta^3-19\theta^2-14\theta-1.$$
$$\mathcal D=\theta^6+2\theta^5-29\theta^4+76\theta^3
-41\theta^2-30\theta-3.$$
By plotting of graph of $\mathcal D$ for $\theta>1$ one can see that there is $\theta_c^{(8)}\approx 1.5587$ such that $\mathcal D>0$ iff $\theta>\theta_c^{(8)}$. Moreover, for such values of $\theta$, $w_1$ and $w_2$ are positive (see Fig. \ref{w12}). For graphs of $u_i$, $i=1,2,3,4$ see Fig. \ref{u1234-2}.
\begin{figure}[h!]
	\begin{center}
		\includegraphics[width=7cm]{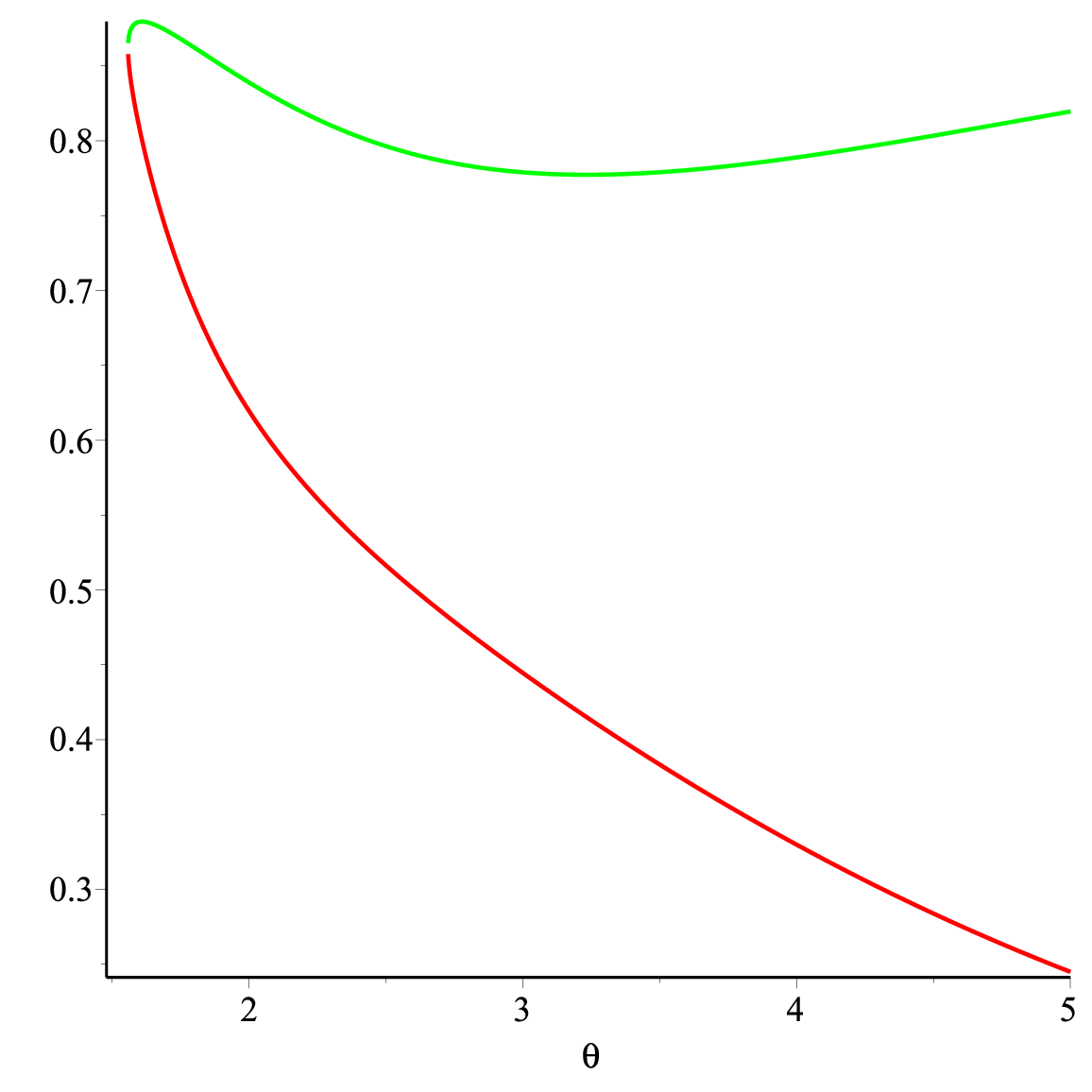}
			\includegraphics[width=7cm]{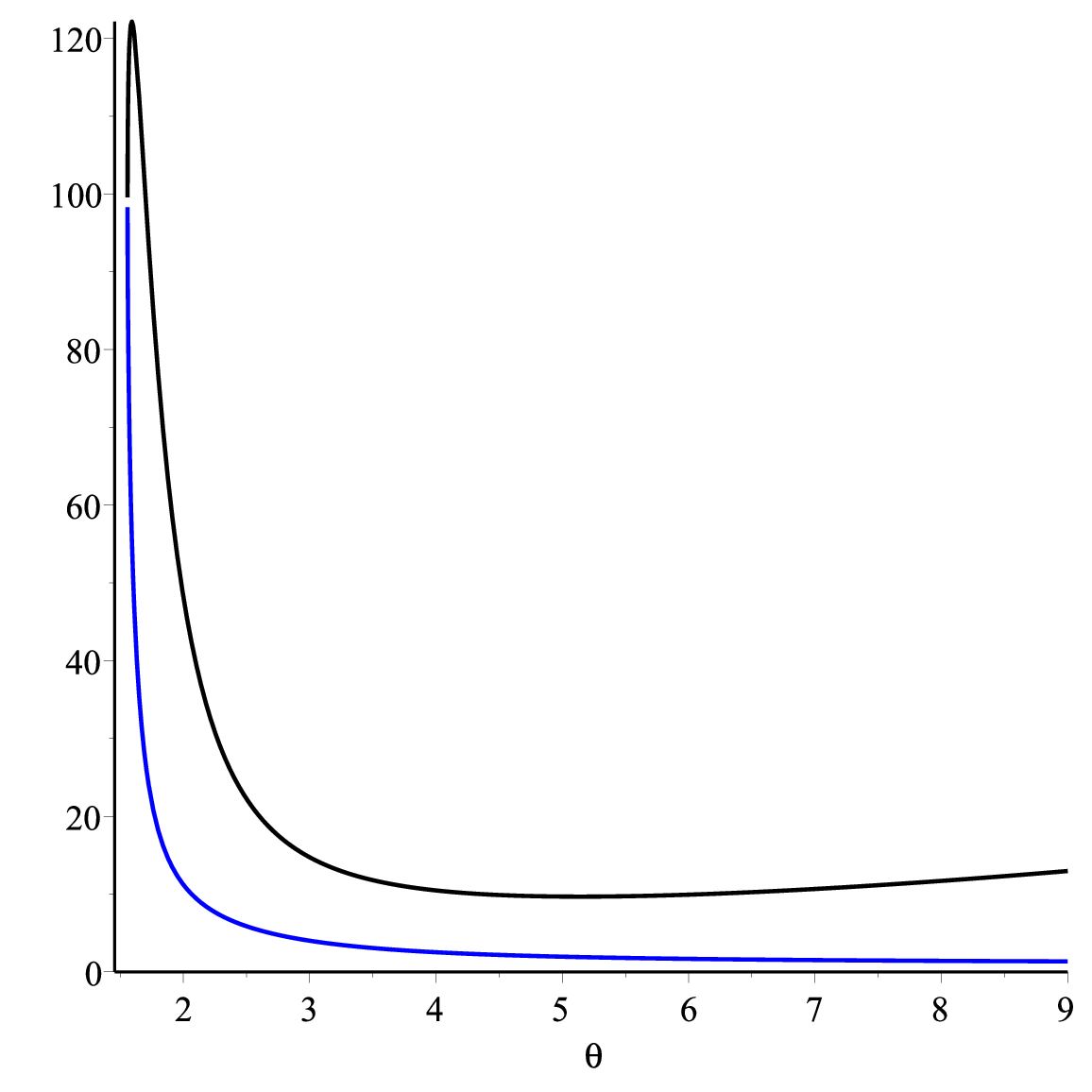}
	\end{center}
	\caption{The graphs of $u_1$ (red), $u_2$ (blue), $u_3$ (green), $u_4$ (black) for $q=5$, $m=2$, $\theta>\theta_c^{(8)}$.}\label{u1234-2}
\end{figure}
\begin{figure}[h]
	\begin{center}
		\includegraphics[width=16cm]{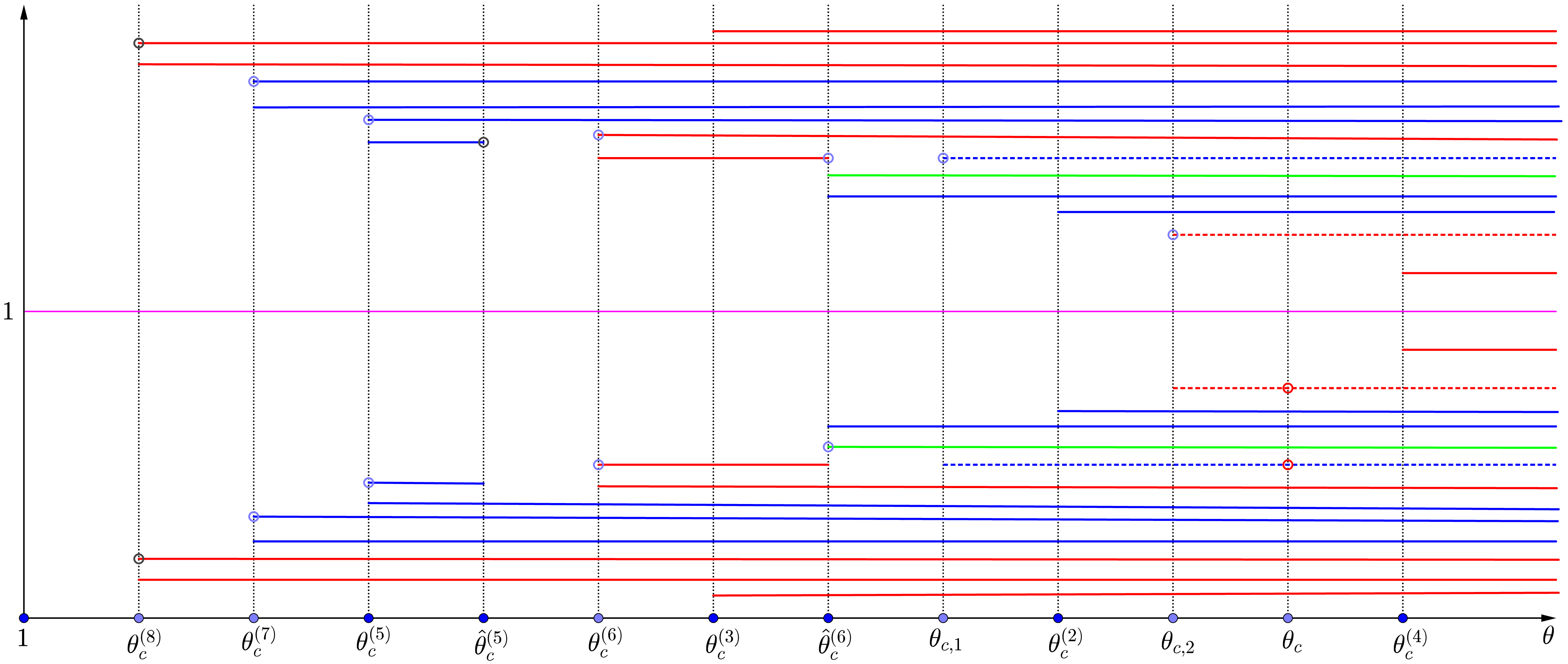}
	\end{center}
	\caption{The distribution of solutions to (\ref{tej}) for  $k=2$  and  $q=5$  is illustrated in the figure. Each line represents a distinct solution plotted over the region of $\theta$  where it exists. The colors denote different types of solutions: pink corresponds to the solution 1 (which always exists), green represents the solutions  $z_*$  and  $z^*$, blue indicates the solutions for  $m=1$, (dashed blue is the case $u=v$, $w=1$)  and red signifies the solutions for  $m=2$ (dashed red is the case $m=2$ with $u=v$, $w=1$). The presence of circular dots indicates that a solution is absent at that point. This figure is useful for counting the number of solutions between critical values of  $\theta$; simply count how many lines appear in the region. }\label{rang}
\end{figure}

Summarized above obtained Results 1-7, in case $k=2$ and $q=5$, we plot the Fig.\ref{rang}. This figure is useful to count number
 of solutions having the form (\ref{mat}) (i.e., $u$ and $v$ occupied the first $m$ coordinates). As one can see in the figure that such solutions as up to 27.

  Denote by
$n_r(\theta)$ the number of color $r$ in Fig. \ref{rang} at $\theta$.

Let $\mathcal N_{k,q}(\theta)$ be the number of solutions to (\ref{fp}), for given $k\geq 2$, $q\geq 2$ and $\theta>0$.
To count number $\mathcal N_{2,5}(\theta)$ as function of $\theta$ we use Fig. \ref{rang} and the  Corollary \ref{c1}.
The following formula is obvious:
$$
\mathcal N_{2,5}(\theta)=1+2n_{\rm green}(\theta)+\left(n_{\rm dash \, blue}(\theta)+2n_{\rm blue}(\theta)\right){5\choose 1}+\left(n_{\rm dash \, red}(\theta)+2n_{\rm red}(\theta)\right){5\choose 2}$$
\begin{equation}\label{nq}
	=1+2n_{\rm green}(\theta)+5n_{\rm dash \, blue}(\theta)+10n_{\rm blue}(\theta)+10n_{\rm dash \, red}(\theta)+20n_{\rm red}(\theta).
\end{equation}
In the Table we give exact values of $\mathcal N_{2,5}(\theta)$.

\begin{rk} We note that for the ferromagnetic Ising model (see \cite[Theorem 2.2]{Ro}) there are up to (exactly) three TISGMs and for  the ferromagnetic $q$-state Potts model (see \cite[Theorem 8.3]{Rbp}) there are up to (exactly) $2^q-1$ TISGMs. However, Theorem \ref{Thm3} says that for the ferromagnetic coupled Ising-Potts model there are at lest $2^q+1$ TISGMs, this is not exact upper bound, because for case $k=2$, and $q=5$, as we showed (see Table) that the number of TISGMs is up to exact bound 335, i.e., very big than $2^5+1=33$.
\end{rk}

In Fig. \ref{rang}, we were unable to consider the few values of $\theta$ at which two distinct solutions coincide. Therefore, in the following theorem, we carefully summarize our results using the phrase "at most."

\begin{thm} \label{Thm4}
	For the ferromagnetic ($J>0$, i.e. $\theta>1$) coupled Ising-Potts model (\ref{ph}) on the Cayley tree of
	order two, with $q=5$ there are 12 critical (temperatures) values of $\theta$. For each $\theta>1$ there are at most $\mathcal N_{2,5}(\theta)$  (see Table)  TISGMs.	
\end{thm}
\begin{rk} It is known that the 5-state Potts model has 3 critical temperatures and up to 31 TISGMs (see \cite[Theorem 1]{KRK}). As stated in Theorem \ref{Thm4}, the coupled Ising-Potts model has four times as many critical temperatures as the Potts model and approximately 11 times more TISGMs. Therefore, our model alters the number of phases more rapidly and features a significantly richer class of splitting Gibbs measures.
\end{rk}

$\begin{array}{|l|l|}
\hline
\textbf{Regions of } \theta &  \mathcal{N}_{2,5}(\theta) \\[1.4mm]
\hline
\theta \in (1, \theta_c^{(8)}) & 1 \\[1.4mm]
\hline
\theta = \theta_c^{(8)} \approx 1.5587 & 41 \\[1.4mm]
\hline
\theta \in (\theta_c^{(8)}, \theta_c^{(7)}) & 81 \\[1.4mm]
\hline
\theta = \theta_c^{(7)} \approx 5.8416 & 101 \\[1.4mm]
\hline
\theta \in (\theta_c^{(7)}, \theta_c^{(5)}) & 121 \\[1.4mm]
\hline
\theta = \theta_c^{(5)} \approx 6.3360 & 141 \\[1.4mm]
\hline
\theta \in (\theta_c^{(5)}, \hat{\theta}_c^{(5)}) & 161 \\[1.4mm]
\hline
\theta = \hat{\theta}_c^{(5)} \approx 6.4174 & 151 \\[1.4mm]
\hline
\theta \in (\hat{\theta}_c^{(5)}, \theta_c^{(6)}) & 141 \\[1.4mm]
\hline
\theta = \theta_c^{(6)} \approx 7.7897 & 181 \\[1.4mm]
\hline
\theta \in (\theta_c^{(6)}, \theta_c^{(3)}), \theta_c^{(3)}\approx 8.3376 & 221 \\[1.4mm]
\hline
\end{array}\qquad \begin{array}{|l|l|}
\hline
\textbf{Regions of } \theta & \mathcal{N}_{2,5}(\theta) \\[1.7mm]
\hline
\theta \in [\theta_c^{(3)}, \hat{\theta}_c^{(6)}) & 261 \\[1.7mm]
\hline
\theta = \hat{\theta}_c^{(6)} \approx 8.3589 & 263 \\[1.7mm]
\hline
\theta \in (\hat{\theta}_c^{(6)}, \theta_{c,1}) & 245 \\[1.7mm]
\hline
\theta = \theta_{c,1} \approx 9.8989 & 250 \\[1.7mm]
\hline
\theta \in (\theta_{c,1}, \theta_c^{(2)}) & 255 \\[1.7mm]
\hline
\theta \in [\theta_c^{(2)}, \theta_{c,2}) & 275 \\[1.7mm]
\hline
\theta = \theta_{c,2} \approx 11.7125 & 285 \\[1.7mm]
\hline
\theta \in (\theta_{c,2}, \theta_c) & 295 \\[1.7mm]
\hline
\theta = \theta_c \approx 11.9160 & 280 \\[1.7mm]
\hline
\theta \in (\theta_c, \theta_c^{(4)}), \theta_c^{(4)}\approx 11.9871 & 295 \\[1.7mm]
\hline
\theta \in [\theta_c^{(4)}, +\infty) & 335 \\[1.7mm]
\hline
\end{array}.$

\section{Conditions of extremality of TISGMS}
In this section, for a given TISGM $\mu_\ell$, (where $\ell=(u,v,w)$ is a solution to (\ref{tej})),  we find conditions (on temperature)  under which $\mu$ is extreme (or non-extreme).
Below we need explicit formulas of solutions $\ell$, which are  given in the previous sections.

\begin{rk} For each fixed $m\leq [{q\over 2}]$, the symmetry of the
	Ising-Potts model allows us to study extremality of TISGMs only for solutions of the form (\ref{mat}), i.e.,extremality is invariant under permutations in coordinates of solutions.
\end{rk}
Given a solution ${\bf z}=\left(z_{\epsilon,i}: \, \epsilon\in\{-1,1\}, \, i=1, 2, \dots, q\right)$ to  (\ref{pt}), (\ref{pt}), the marginals with respect to TISGM $\mu_{{\bf z}}$, on the two-site volumes which consist of two adjacent sites $x,y$,
in the case of the Ising-Potts model is
$$\mu_{{\bf z}}((s(x)=\eta, \s(x)=i),(s(y)=\epsilon, \s(y)=j)
)= \frac{1}{Z} z_{\eta, i}  \exp(J\beta\eta\epsilon\delta_{i j}) z_{\epsilon,j}.$$
From this  the relation between a solution and the transition matrix for the associated tree-indexed Markov chain, in our case the transition matrix $\mathbb{P}$ is given by formula  of the conditional probability. Indeed, the
corresponding transition matrix giving the probability to go from a state $(\eta, i)$
at site $x$ to a state
to $(\epsilon, j)$ at the neighbor $y$
is then
\begin{equation}\label{epe2}
	P_{(\eta, i)(\epsilon, j)}={\theta^{\eta\epsilon\delta_{i j}} z_{\epsilon,j}\over \sum_{\hat\epsilon\in \{-1,1\}}\sum_{r=1}^q\theta^{\eta\hat\epsilon\delta_{i r}} z_{\hat\epsilon,r}}.
\end{equation}
In our investigated case we have
\begin{equation}\label{mat1}
	\left(\begin{array}{ll}
		z_{-1,i}\\
		z_{1,i}\end{array}\right)_{i=1}^q=
	\left(\begin{array}{cccccccc}
		u & u & \dots & u & 1 & 1 & \dots & 1\\
		v & v & \dots & v & w & w & \dots & w\end{array}\right),\end{equation}
where $u$ (and $v$) is written $m$ times, and 1 (and $w$)  $q-m$ times.

For all $j = 2, \dots, q$, and noting that the set of vectors
$\{e_1, e_2, \dots, e_q\}$ forms the canonical basis of $ \mathbb{R}^q$, we define the following expressions:
 $$\begin{aligned}
\varsigma_{11} &= \frac{1}{\mathcal{Z}_1} \left( \theta u e_1 + \sum_{i=2}^q e_i \right),
& \varsigma_{1j} &= \frac{1}{\mathcal{Z}_2} \left( u e_1 + (\theta - 1)e_j + \sum_{i=2}^q e_i \right), \\
\varsigma_{21} &= \frac{1}{\mathcal{Z}_1} \left( \theta^{-1} v e_1 + w \sum_{i=2}^q e_i \right),
& \varsigma_{2j} &= \frac{1}{\mathcal{Z}_2} \left( v e_1 + w(\theta^{-1} - 1)e_j + w \sum_{i=2}^q e_i \right), \\
\varsigma_{31} &= \frac{1}{\mathcal{Z}_3} \left( \theta^{-1} u e_1 + \sum_{i=2}^q e_i \right),
& \varsigma_{3j} &= \frac{1}{\mathcal{Z}_4} \left( u e_1 + \left(\theta^{-1} - 1\right)e_j + \sum_{i=2}^q e_i \right), \\
\varsigma_{41} &= \frac{1}{\mathcal{Z}_3} \left( \theta v e_1 + w \sum_{i=2}^q e_i \right),
& \varsigma_{4j} &= \frac{1}{\mathcal{Z}_4} \left( v e_1 + w(\theta - 1)e_j + w \sum_{i=2}^q e_i \right)
\end{aligned} $$
with
$$\begin{gathered}
\mathcal{Z}_1=\theta u+\theta^{-1} v+(q-1)(w+1), \quad \mathcal{Z}_2=\theta+u+v+\theta^{-1} w+(q-2)(w+1), \\
\mathcal{Z}_3=\theta^{-1} u+\theta v+(q-1)(w+1), \quad \mathcal{Z}_4=\theta^{-1}+u+v+\theta w+(q-2)(w+1).
\end{gathered}$$
For solutions of the form (\ref{mat1}), in the case where $m = 1$, the transition matrix $\mathbb{P}_m$, as defined by (\ref{epe2}), is expressed as follows:
\begin{equation}\label{PT}
\mathbb P_1	= \left(\begin{array}{cc}
{\mathbb P_{11}} & {\mathbb P_{12}}\\[2mm]
{\mathbb P_{21}} & {\mathbb P_{22}}
\end{array}\right),\end{equation}
where
$$\begin{array}{ll}
\mathbb{P}_{11}=\left(\varsigma_{11}, \varsigma_{12}, \ldots, \varsigma_{1 q}\right)^t, & \mathbb{P}_{12}=\left(\varsigma_{21}, \varsigma_{22}, \ldots, \varsigma_{2 q}\right)^t, \\[2mm]
\mathbb{P}_{21}=\left(\varsigma_{31}, \varsigma_{32}, \ldots, \varsigma_{3 q}\right)^t, & \mathbb{P}_{22}=\left(\varsigma_{41}, \varsigma_{42}, \ldots, \varsigma_{4 q}\right)^t.
\end{array}$$
\subsection{Conditions of non-extremality}
It is established that a sufficient condition for the non-extremality of a Gibbs measure $ \mu $ corresponding to the matrix $ \mathbb{P} $ is given by $ k\lambda_2^2 > 1 $, where $ \lambda_2 $ represents the second largest eigenvalue of $ \mathbb{P} $ in absolute value \cite{Ke}.
In the particular case where $ u = v = w = 1 $, the normalization constants are equal, $ \mathcal{Z}_1 = \mathcal{Z}_2 = \mathcal{Z}_3 = \mathcal{Z}_4 = \theta + \theta^{-1} + 2(q-1) $. The matrix $ \mathbb{P}_1 $ can then be expressed as $ \mathbb{P}_1 = \frac{1}{\mathcal{Z}_1} \cdot \mathbb{P} $, where
$$\mathbb{P} =
\begin{pmatrix}
\theta \cdot E + I & \theta^{-1} E + I \\[2mm]
\theta^{-1} E + I & \theta \cdot E + I
\end{pmatrix},
$$
with $ E $ denoting the $ q \times q $ identity matrix, and $ I = (a_{ij}) $ defined as:
$$
a_{ij} =
\begin{cases}
0, & \text{if } i = j, \\
1, & \text{if } i \neq j.
\end{cases}
$$

\begin{pro}The matrix $\mathbb{P}_1$ has three distinct eigenvalues: $1$,
	$$\lambda_1(\theta) = \frac{(\theta-1)^2}{\theta^2 + 2(q-1)\theta + 1}, \ \ \mbox{with multiplicity} \  \ q-1,$$ $$\lambda_2(\theta) = \frac{\theta^2-1}{\theta^2 + 2(q-1)\theta + 1}, \ \  \mbox{with multiplicity}  \ \ q.$$
\end{pro}
\begin{proof} Let $\mathbb{P}u = \lambda u$. Then,
$\mathbb{P}_1 u = \frac{1}{\mathcal{Z}_1} \mathbb{P}u = \frac{1}{\mathcal{Z}_1} \lambda u.$ Thus, the eigenvalue of $\mathbb{P}_1$ is $\frac{\lambda}{\mathcal{Z}_1}$. Consequently, it suffices to determine the eigenvalues of the matrix $\mathbb{P}$.
Consider the matrix $\theta \cdot E + I = A$ as a linear transformation on $\mathbb{R}^q$. It follows directly that $A u = (\theta + q-1) u,$ with
$u = (1, 1, \ldots, 1)^{t}.$
Define $V = \{v \in \mathbb{R}^q \mid (v, u) = 0\}$ and $e_i = (\alpha_{i1}, \alpha_{i2}, \ldots, \alpha_{iq})^{t}$, $i = 1, \ldots, q-1$, with
$\alpha_{ii} = 1, \alpha_{iq} = -1$ and for the other cases $\alpha_{ij} = 0.$
Clearly, $$V = \operatorname{span}\{e_i \mid i = 1, \ldots, q-1\}.$$ Furthermore, it can be verified that
$A e_i = (\theta-1) e_i \quad \text{for all } i = 1, \ldots, q-1$. Since $\dim V=q-1$ and all vectors in $V$ are eigenvectors, the linear transformation $A$ has two distinct eigenvalues: $\theta + q-1$ and $\theta - 1$, with multiplicities $1$ and $q-1$, respectively.
To find the eigenvalues of $\mathbb{P}$, consider:
$$\det(\mathbb{P} - \lambda E) = \det\left(\begin{array}{cc}
(\theta - \lambda)E + I & \theta^{-1}E + I \\[2mm]
\theta^{-1}E + I & (\theta - \lambda)E + I
\end{array}\right).$$

\textbf{Case 1:} $\lambda \notin \{\theta - 1, \theta + q - 1\}$.
In this case, $(\theta - \lambda)E + I$ is invertible. Denoting $((\theta - \lambda)E + I)B_\lambda = E$, we obtain:
$$\left(\begin{array}{cc}
E & 0 \\[2mm]
-\left(I + \theta^{-1}E\right)B_\lambda & E
\end{array}\right) \cdot \left(\begin{array}{cc}
(\theta - \lambda)E + I & \theta^{-1}E + I \\[2mm]
\theta^{-1}E + I & (\theta - \lambda)E + I
\end{array}\right)=$$
$$\left(\begin{array}{cc}
(\theta - \lambda)E + I & \theta^{-1}E + I \\[2mm]
0 & (\theta - \lambda)E + I - \left(I + \theta^{-1}E\right)B_\lambda\left(I + \theta^{-1}E\right)
\end{array}\right).$$\\[1mm]
From this equality,
$$
\det(\mathbb{P} - \lambda E) = \det((\theta - \lambda)E + I) \cdot \det\left[(\theta - \lambda)E + I - \left(I + \theta^{-1}E\right)B_\lambda\left(I + \theta^{-1}E\right)\right].
$$
Simplifying further,
$$
\det(\mathbb{P} - \lambda E) = \det\left[((\theta - \lambda)E + I)^2 - \left(I + \theta^{-1}E\right)^2\right].
$$
This simplifies to:
$$
\det(\mathbb{P} - \lambda E) = \left(\theta - \theta^{-1} - \lambda\right)^q \cdot \det\left[2I + \left(\theta + \theta^{-1} - \lambda\right)E\right].
$$

The eigenvalues of the linear transformation $2I + \left(\theta + \theta^{-1} - \lambda\right)E$ are found to be $\theta + \theta^{-1} + 2(q-1)$ and $\theta + \theta^{-1} - 2$. Thus, the matrix $\mathbb{P}$ has three eigenvalues:
$\theta + \theta^{-1} + 2(q-1)$, $\theta - \theta^{-1}$ (multiplicity $q$), and  $\theta+\theta^{-1}-2$ (multiplicity $q-1$).\medskip

\textbf{Case 2:} $\lambda \in \{\theta - 1, \theta + q - 1\}$.
Here, $(\theta - \lambda)E + I$ is not invertible. However, for any $\varepsilon$ such that $0 < \varepsilon < \alpha(\theta)$, where
$$\alpha(\theta) =
\begin{cases}
|\theta - 1| & \text{if } \theta \neq 1, \\
q & \text{if } \theta = 1,
\end{cases}$$
we have $\det[(\theta - \lambda + \varepsilon)E + I] \neq 0$. Defining \medskip
$$\mathbb{P}_\varepsilon = \left(\begin{array}{cc}
(\theta + \varepsilon)E + I & \theta^{-1}E + I \\[2mm]
\theta^{-1}E + I & \theta E + I
\end{array}\right),$$\medskip
and proceeding as in Case 1, we find
$$
\det(\mathbb{P}_\varepsilon - \lambda E) = \det\left[(\theta E + I)((\theta + \varepsilon)E + I) - \left(\theta^{-1}E + I\right)^2\right].
$$
By continuity of the determinant function, taking the limit as $\varepsilon \to 0$ yields the result.
\end{proof}
It is easy to see that $\lambda_1(\theta)\leq|\lambda_2(\theta)|$ for any $\theta>0$ and the equality only at $\theta=1$.
Therefore non-extremality of the measure $\mu_1$ corresponding to solution\footnote{Recall that the solution $u=v=w=1$ to (\ref{tej}) exists independently on all parameters.} for  $u=v=w=1$ is (for $k=2$)
\begin{equation}\label{qq}
	2{(\theta^2-1)^2\over (\theta^2+2(q-1)\theta+1)^2}>1.
	\end{equation}
For $\theta\leq 1$ (resp. $\theta>1$) the equality (\ref{qq}) can be reduced to
$$(1+\sqrt{2})\theta^2+2(q-1)\theta+1-\sqrt{2}<0$$
$$\left({\rm resp.} \ \ (\sqrt{2}-1)\theta^2-2(q-1)\theta-(\sqrt{2}+1)>0. \right)$$
Solving these inequalities we find that the inequality (\ref{qq}) is satisfied for all
\begin{equation}\label{bire} \theta\in \left(0, {1\over \theta_1(q)}\right)\cup
 \left(\theta_1(q), +\infty\right).
 \end{equation}
where
$$\theta_1(q)={q-1+\sqrt{(q-1)^2+1}\over \sqrt{2}-1}.$$
 Thus we have

 {\bf Result 8.} For each $q\geq 2$, the measure $\mu_1$  is not extreme for all values of  $\theta$  that satisfy (\ref{bire}).\\

 {\bf Case:} $u=v=1$, $w\ne 1$. In this case we have two measures $\mu_*$ and $\mu^*$ which correspond to solutions mentioned in Result 1. In this case by a computer we have explicitly found that the matrix (\ref{mat1}) has 5 eigenvalues: 1 and
 $\lambda_i(\theta, w), i=1,2,3,4.$
 Furthermore, four of these eigenvalues have a multiplicity of 1, while two of them have a multiplicity of 3. Since the eigenvalues have very long formulas, we do not present them here. But for concrete values $w=z_*$ and $w=z^*$ we give explicitly the second largest ones.

 For $q=5$, $k=2$ and $w=z_*$ from explicit formulas by computer we note that (see Fig.\ref{ex0})

 $$\max_{i=1,2,3,4} |\lambda_i(\theta,z_*)|=
 \lambda_1(\theta,z_*)={\theta^2-1\over \theta^2+4\theta z_*+4\theta+1}$$
 where for $\theta>\theta_{c,0}$ the solution is
  $$z_*=
 \frac{1}{2} \cdot \frac{\theta^4 - 34 \theta^2 - 16 \theta- 1-\sqrt{\theta^8 - 68 \theta^6 - 32 \theta^5 + 130 \theta^4 + 64 \theta^3 - 60 \theta^2 - 32 \theta - 3} }{16 \theta^2 + 8 \theta + 1}.
 $$
 \begin{figure}[h!]
 	\begin{center}
 		\includegraphics[width=8cm]{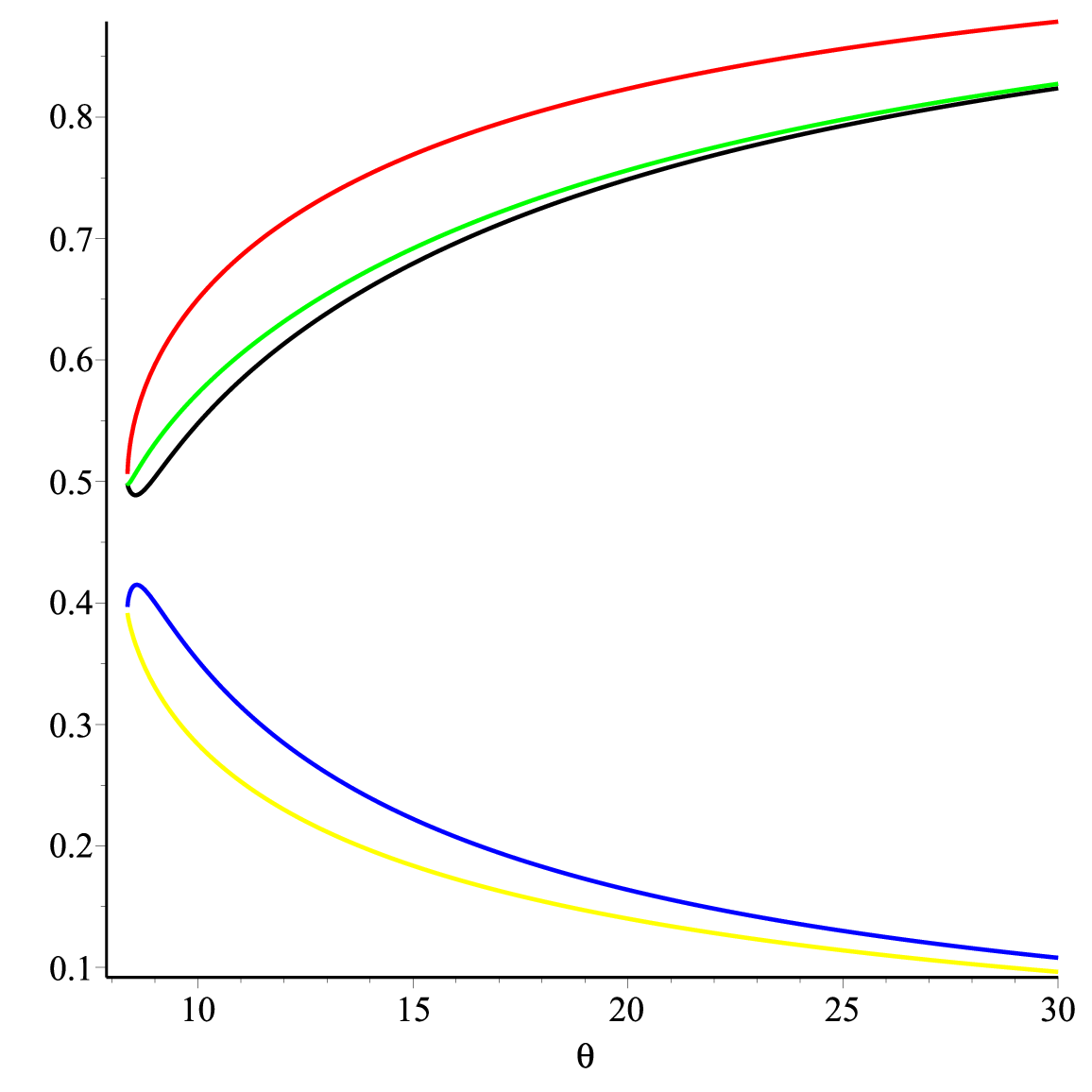}
 	 	\end{center}
 	\caption{The graph of $\lambda_1(\theta, z_*)$ (red),  $\lambda_2$ (black), $\lambda_3$ (blue), $\lambda_4$ (green, which has multiplicity 3),  $\lambda_5$ (yellow, with multiplicity 3), for $q=5$, $\theta>\theta_{c,0}$.}\label{ex0}
 \end{figure}
Thus sufficient condition of non-extremality of $\mu_*$ is
$$2 \lambda^2_1(\theta, z_*)-1>0,$$
and computer analysis shows that this condition is satisfied for
any
\begin{equation}\label{ex6}
	\theta>\theta_*\approx 11.76
	\end{equation}
{\bf Result 9.} For $q=5$ the measure $\mu_*$ is not extreme for
all $\theta$ satisfying (\ref{ex6}).\\

 For $q=5$, $k=2$ and $w=z^*$ using explicit formulas of $\lambda_i(\theta, w)$ by computer we obtain that (see Fig.\ref{exa}):

$$\max_{i=1,2,3,4} |\lambda_i(\theta, z^*)|=
\lambda_4(\theta, z^*),$$
for $\theta>\theta_{c,0}$.
Here
$$\lambda_4(\theta, w) := \frac{\theta-1}{2M} \left((3\theta+1)\theta w^2+ 2 (\theta+4)\theta^2 w + 5 \theta^2 + \theta+ \sqrt{D_1} \right),
$$
with
$$D_1:=
(3\theta+1)^2\theta^2 w^4 + 4\left(3 \theta^3 + 4 \theta^2 + 10 \theta + 3\right)\theta w^3 $$ $$-2\left( 11 \theta^4 -8 \theta^3 -59 \theta^2 -16 \theta -2 \right)w^2 -4\left(5\theta^3 -4 \theta^2 -26 \theta -5 \right)\theta w + (5\theta + 1)^2\theta^2.$$

$$M:=\left( 3 \theta^3 + 10 \theta^2 + 3 \theta \right)w^2 + \left( \theta^4 + 8 \theta^3 + 30 \theta^2 + 8 \theta + 1 \right) w + 5 \theta^3 + 26 \theta^2 + 5 \theta.
$$
and the solution is
$$z^*=
\frac{1}{2} \cdot \frac{\theta^4 - 34 \theta^2 - 16 \theta- 1+\sqrt{\theta^8 - 68 \theta^6 - 32 \theta^5 + 130 \theta^4 + 64 \theta^3 - 60 \theta^2 - 32 \theta - 3} }{16 \theta^2 + 8 \theta + 1}.
$$
\begin{figure}[h!]
	\begin{center}
		\includegraphics[width=8cm]{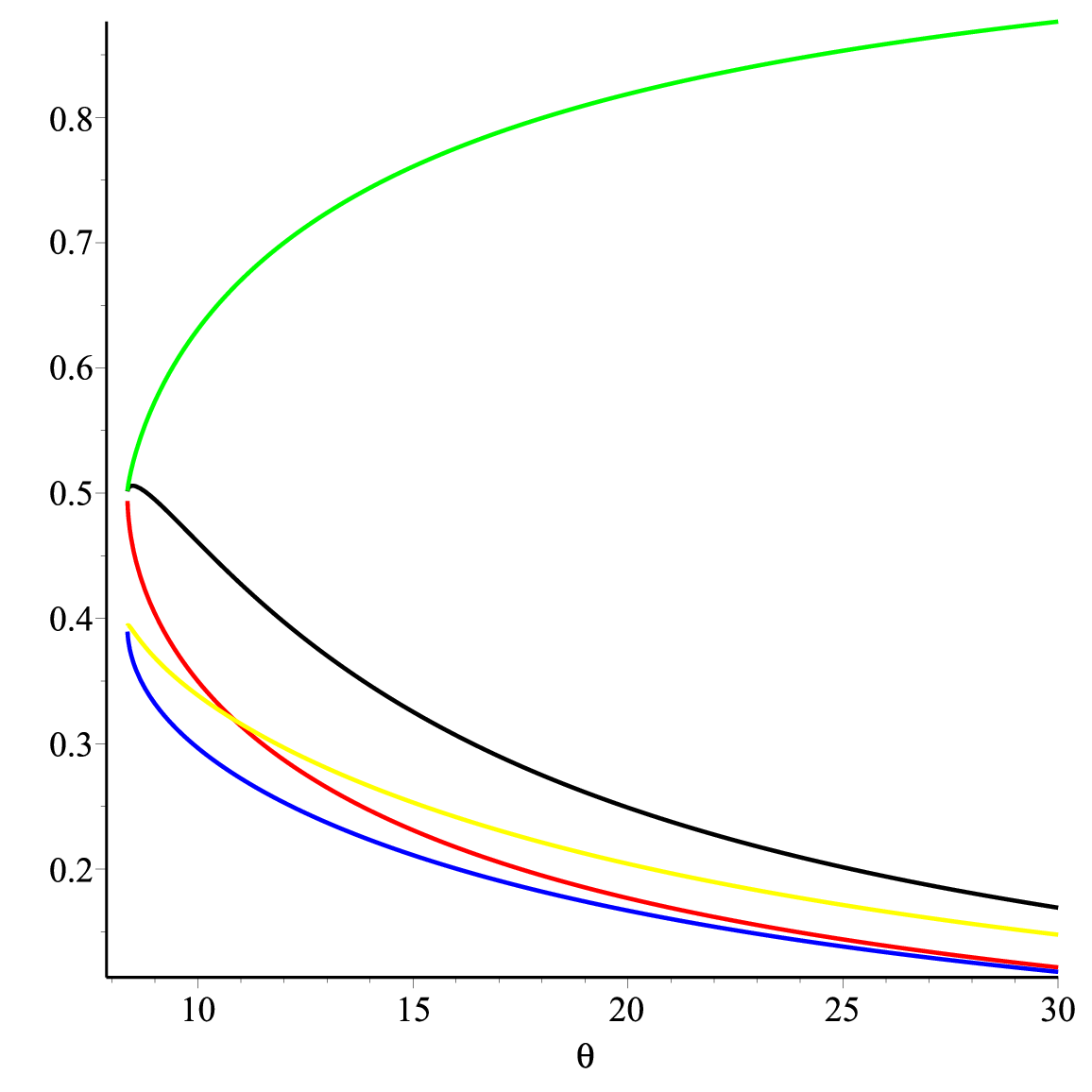}
	\end{center}
	\caption{The graph of $\lambda_1(\theta, z^*)$ (red),  $\lambda_2$ (black), $\lambda_3$ (blue), $\lambda_4$ (green, which has multiplicity 3),  $\lambda_5$ (yellow, with multiplicity 3), for $q=5$, $\theta>\theta_{c,0}$.}\label{exa}
\end{figure}

Thus sufficient condition of non-extremality of $\mu^*$ is
$$2 \lambda^2_4(\theta, z^*)-1>0,$$
and computer analysis shows that this condition is satisfied for
any
\begin{equation}\label{ex7}
	\theta>\theta^*\approx 12.28
\end{equation}
{\bf Result 10.} For $q=5$ the measure $\mu^*$ is not extreme for
all $\theta$ satisfying (\ref{ex7}).\\

\begin{rk} 	
	In the remaining cases (i.e., Case 1: $u=v\ne 1$, $w=1$; Case 2: $u=v, w\ne 1$;  Case 3: $u\ne v$, $w=1$; and Case 4: $u\ne v$, $w\ne 1$) we were unable to establish non-extremality conditions. For instance, in Case 1, our numerical analysis aimed at finding eigenvalues resulted in the message: "Length of output exceeds limit of 1,000,000."
\end{rk}

\subsection{Conditions of extremality}
    In this subsection we establish sufficient conditions for extremality of TISGMs corresponding to solutions mentioned in Result 1  for the $q=5$. That are the TISGMs: $\mu_1$, $\mu_*$ and $\mu^*$ (as denoted in the previous subsection).

To check extremality, we will use a result of~\cite{MSW} to establish a bound for reconstruction impossibility (which is equivalent to extremality) that corresponds to the matrix~\eqref{PT}) for solutions $u=v=1$ and $w\in \{1, z_*, z^*\}$.
 For such solutions the transition matrix is
   \begin{equation}\label{ET}
\mathbb P^{\prime}_1	= \left(\begin{array}{cc}
{\mathbb P^{\prime}_{11}} & {\mathbb P^{\prime}_{12}}\\[2mm]
{\mathbb P^{\prime}_{21}} & {\mathbb P^{\prime}_{22}}
\end{array}\right), \quad \begin{array}{ll}
\mathbb{P}^{\prime}_{11}=\left(\varsigma_{11}, \varsigma_{12}, \ldots, \varsigma_{1 q}\right)^t, & \mathbb{P}^{\prime}_{12}=\left(\varsigma_{21}, \varsigma_{22}, \ldots, \varsigma_{2 q}\right)^t, \\[2mm]
\mathbb{P}^{\prime}_{21}=\left(\varsigma_{31}, \varsigma_{32}, \ldots, \varsigma_{3 q}\right)^t, & \mathbb{P}^{\prime}_{22}=\left(\varsigma_{41}, \varsigma_{42}, \ldots, \varsigma_{4 q}\right)^t.
\end{array}\end{equation}
Here
    $$\begin{aligned}
\varsigma_{11} &= \frac{1}{\mathcal{Z}_1} \left( (\theta-1) e_1 + \sum_{i=1}^q e_i \right),
& \varsigma_{1j} &= \frac{1}{\mathcal{Z}_2} \left((\theta - 1)e_j + \sum_{i=1}^q e_i \right), \\
\varsigma_{21} &= \frac{1}{\mathcal{Z}_1} \left( \theta^{-1} e_1 + w \sum_{i=2}^q e_i \right),
& \varsigma_{2j} &= \frac{1}{\mathcal{Z}_2} \left(e_1 + w(\theta^{-1} - 1)e_j + w \sum_{i=2}^q e_i \right), \\
\varsigma_{31} &= \frac{1}{\mathcal{Z}_4} \left( (\theta^{-1}-1)e_1 + \sum_{i=1}^q e_i \right),
& \varsigma_{3j} &= \frac{1}{\mathcal{Z}_4} \left(\left(\theta^{-1} - 1\right)e_j + \sum_{i=1}^q e_i \right), \\
\varsigma_{41} &= \frac{1}{\mathcal{Z}_4} \left( \theta e_1 + w \sum_{i=2}^q e_i \right),
& \varsigma_{4j} &= \frac{1}{\mathcal{Z}_4} \left(e_1 + w(\theta - 1)e_j + w \sum_{i=2}^q e_i \right)
\end{aligned} $$
with
$$\begin{gathered}
\mathcal{Z}_1=\theta+\theta^{-1}+(q-1)(w+1), \ \mathcal{Z}_2=\theta+q+(\theta^{-1}+q-2)w:=\mathcal{Z}_2(\theta), \ \mathcal{Z}_4=\mathcal{Z}_2(\theta^{-1}).
\end{gathered}$$

   To enhance clarity for the reader, let us begin by some definitions from~\cite{MSW}. Considering finite complete subtrees $\mathcal T$ that are {\em initial} with respect to the Cayley tree $\Gamma^k$, i.e., share the same root. If $\mathcal T$ has depth $d$, i.e., the vertices of $\mathcal T$ are within distance $\leq d$ from the root, then it has $(k^{d+1}-1)/(k-1)$ vertices, and its boundary $\partial \mathcal T$ consists of the neighbors (in $\Gamma^k\setminus \mathcal T$) of its vertices, i.e., $|\partial \mathcal T|=k^{d+1}$.  We identify subgraphs of $\mathcal T$ with their vertex sets and write $E(A)$ for the edges within a subset $A$ and $\partial A$ for the boundary of $A$, i.e., the neighbors of $A$ in $(\mathcal T\cup \partial\mathcal T)\setminus A$.

  Let  $\{\mu^\tau_{{\mathcal T}}\}$ be Gibbs measures, where the boundary condition $\tau$
   is fixed and $\mathcal T$ ranges over all initial finite complete subtrees of $\Gamma^k$.

   Take subtree $\mathcal T$ of $\Gamma^k$ and a vertex $x\in\mathcal T$, and  write $\mathcal T_x$ for
   the maximal subtree of $\mathcal T$ rooted at $x$. If $x$ is not the root of $\mathcal T$, then let $\mu_{\mathcal T_x}^s$
   denote the (finite-volume) Gibbs measure in which the
   parent of $x$ has its spin fixed to $s$ and the configuration on the bottom boundary of ${\mathcal T}_x$
   (i.e., on $\partial {\mathcal T}_x\setminus \{\mbox{parent\ \ of}\ \ x\}$) is specified by $\tau$.

   The variational distance between the projections of two measures $\mu_1$ and $\mu_2$ onto the spin at $x$, defined as
   $$\|\mu_1-\mu_2\|_x:={1\over 2}\sum_{h\in I\times \Phi}|\mu_1(\sigma(x)=h)-\mu_2(\sigma(x)=h)|.$$
   Let $s\in I\times \Phi$ and $\eta^{x,s}$ be the
   configuration $\eta$ with  $\eta(x)=s$.

   Define (see \cite{MSW}):
   $$\kappa:= \kappa(\mu)=\sup_{x\in\Gamma^k}\max_{x,s,s'}\|\mu^s_{{\mathcal T}_x}-\mu^{s'}_{{\mathcal T}_x}\|_x\quad\text{ and }\quad\gamma:=\gamma(\mu)=\sup_{A\subset \Gamma^k}\max\|\mu^{\eta^{y,s}}_A-\mu^{\eta^{y,s'}}_A\|_x,$$
   where the maximum on the right-hand side is taken over all boundary conditions $\eta$, all sites $y\in \partial A$, all neighbors $x\in A$ of $y$, and all spins $s, s'\in I\times \Phi$. We apply~\cite[Theorem 9.3]{MSW} which goes as follows.
   \begin{thm}\label{tmw} For an arbitrary (ergodic\footnote{Ergodicity here means irreduciblity and aperiodicity. In this case, we have a unique stationary distribution $\pi=(\pi_1,\dots,\pi_q)$ with $\pi_i>0$ for all $i$.} and permissive\footnote{Permissive here means that, for arbitrary finite $A$ and boundary condition outside $A$ being $\eta$, the conditioned Gibbs measure on $A$, corresponding to the channel, is positive for at least one configuration.}) channel ${\mathbb P}=(P_{ij})_{i,j=1}^q$
   	on a tree, the reconstruction of the corresponding tree-indexed Markov chain
   	is impossible if $k\kappa\gamma<1$.
   \end{thm}

   In our case each TISGM $\mu$ corresponds to a solution $(u,v,w)$ of the system of equations~\eqref{tej}, therefore, we can write $\gamma(\mu)=\gamma(u,v,w)$ and $\kappa(\mu)=\kappa(u,v,w)$.

   It is easy to see that the channel ${\mathbb P}$ (see (\ref{PT})) corresponding to a TISGM of the Ising-Potts  model is ergodic and permissive. Thus the criterion of {\it extremality} of a TISGM is $k\kappa\gamma<1$.

   Note that $\kappa$ has the particularly simple form (see~\cite{MSW})
   \begin{equation}\label{ka}
   	\kappa={1\over 2}\max_{i,j}\sum_l|P_{il}-P_{jl}|
   \end{equation}
   and $\gamma$ is a constant that does not have a clean general formula, but it can be estimated.

  Consider the matrix ${\mathbb P}^{\prime}$, given by (\ref{ET}) (which depends on $w$), and denote by
  	$\mu=\mu(\theta,w)$ the corresponding Gibbs measure. Then, for any subset $A\subset {\mathcal T}$, (where $\mathcal T$ is complete subtree of $\Gamma^k$)
  	any boundary configuration $\eta$, any pair of spins
  	$h_1, h_2\in I\times \Phi$, any site $y\in \partial A$, and any neighbor $x\in A$ of $y$, we have
  	$$\|\mu^{\eta^{y,h_1}}_A-\mu^{\eta^{y,h_2}}_A\|_x\leq \max_h |p^{h_1}(h)-p^{h_2}(h)|,$$
  	with $p^{t}(h)=\mu^{\eta^{y,t}}_A((s(x), \sigma(x))=h)$.

  Denote by $p(h)=\mu^{\eta^{y,\,free}}_A((s(x),\sigma(x))=h)$ the Gibbs measure with boundary
  	condition $\eta$, except that the spin at $y$ is free (or equivalently, the edge connecting $x$ to $y$ is erased).
  	  	
   Let $t=(\kappa, i), h=(\varepsilon, j)\in I\times \Phi$, using definition of Gibbs measure and formula (\ref{epe2}) we get
  	\begin{equation}\label{pth} p^{t}(h)=
  			p^{(\kappa, i)}(\epsilon, j)={\theta^{\kappa\epsilon\delta_{i j}} z_{\epsilon,j}p(\epsilon,j)\over \sum_{\hat\epsilon\in \{-1,1\}}\sum_{r=1}^q\theta^{\kappa\hat\epsilon\delta_{i r}} z_{\hat\epsilon,r}p(\hat\epsilon,r)}.
  	\end{equation}
  	
  	For coordinates of vector $(p(h), h\in I\times \Phi)$ we use new notations
  	$$p_i=p(-1,i), \ \ q_i=p(1,i), \ \ b_i=\sum_{{r=1\atop r\ne i}}^q q_r.$$
  	Then
  	$$\sum_{{r=1\atop r\ne i}}^q p_r=1-p_i-q_i-b_i.$$
  	Using these notations, for the case $z_{-1,i}=1$, $i=1,\dots,q$,  $z_{1,1}=1$, $z_{1, j}=w$, $j=2, \dots, q$, from (\ref{pth}) we get
  		\begin{equation}\label{px} p^{t}(h)=
  	p^{(\kappa, i)}(\epsilon, j)=\left\{\begin{array}{llll}
  		{\theta^{-\kappa\delta_{ij}} p_j\over 1+(\theta^{-\kappa}-1) p_1+(\theta^{\kappa}-1)q_1+(w-1)b_1}, \ \  \epsilon=-1, i=1, \\[2mm]
  		{\theta^{-\kappa\delta_{ij}} p_j\over 1+(\theta^{-\kappa}-1) p_i+(\theta^{\kappa}w-1)q_i+(w-1)(b_i-q_1)}, \ \  \epsilon=-1, i\ne 1, \\[2mm]
  		{\theta^{\kappa\delta_{ij}} q_j\over 1+(\theta^{\kappa}-1) p_1+(\theta^{-\kappa}-1)q_1+(w-1)b_1}, \ \  \epsilon=1, i=1, \\[2mm]
  	{\theta^{\kappa\delta_{ij}}w q_j\over 1+(\theta^{\kappa}-1) p_i+(\theta^{-\kappa}w-1)q_i+(w-1)(b_i-q_1)}, \ \  \epsilon=1, i\ne 1.
    		\end{array}\right.
  	\end{equation}
 To estimate $\gamma$, it is necessary to evaluate the expression
  $|p^{(\kappa, i)}(\epsilon, j)-p^{(\kappa^{\prime}, \ell)}(\epsilon, j)|$, where $\kappa, \kappa^{\prime}\in\{-1, 1\}$, $i,j,\ell\in\{1,\dots,q\}$.

  {\bf Extremality of $\mu_1$.}
 In case $w=1$ we estimate $\gamma$. To do this we consider
  $ g(x, y, z, t, a, b) $  defined as:
$$g(x, y, z, t, a, b) := \left| \frac{a x}{z + t + a x + a^{-1} y} - \frac{b x}{x + y + b z + b^{-1} t} \right|,$$
where $ x, y, z, t \geq 0 $ and $ a = b^i > 0 $ with $ i \in \{-1, 1\} $.\\
We seek to study the supremum (norm) of the function $ g(x, y, z, t, a, b) $ with respect to the variables $ x, y, z, t $, denoted by:
$$\| g \| := \sup_{x, y, z, t \geq 0} g(x, y, z, t, a, b).$$
The following lemma provides an analysis of the norm of $ g(x, y, z, t, a, b) $ for different ranges of the parameter $ a $.

Define
\begin{equation}\label{pc}
	P(c) := -c^6 - 2c^5 - c^4 + 2c^2 + 2c + 1.
\end{equation}

By Descartes' rule of signs for polynomials, $ P(c) $ has at most one positive solution. Given that $ P(0) > 0 $ and $ P(2) < 0 $, we conclude that there is exactly one positive solution, denoted by $ a_{cr} $.

The next lemma states the following:
\begin{lemma}\label{newlemma} Under the assumptions stated above, the following results hold: \begin{enumerate}
\item For $0 < a \leq a_{cr}$, the norm of the function $ g(x, y, z, t, a, b) $ is
$$\| g \| = \frac{|a \sqrt{a} - 1| \cdot (a + \sqrt{a} + 1)}{a (a + 1)(\sqrt{a} + 1)}.$$
\item For $ a > a_{cr} $, the norm of the function $ g(x, y, z, t, a, b) $ is $\| g \| = a - 1.$
\end{enumerate}
\end{lemma}
\begin{proof}
We focus on the case where $x > 0$. To facilitate analysis, let us introduce the normalized variables
 $\frac{y}{x}=y_1, \ \frac{z}{x}=z_1, \ \frac{t}{x}=t_1.$
Substituting these into $g(x, y, z, t, a, b)$, we obtain the expression
\begin{equation}\label{eq0}
 \left|\frac{a}{a+a^{-1} y_1+z_1+t_1}-\frac{b}{1+y_1+b z_1+b^{-1} t_1}\right|.
\end{equation}
This reformulated expression, denoted as $\eqref{eq0}$, provides a more compact representation for further analysis under the assumption $x > 0$.\\
\textbf{Case $a=b$.} When $a = b$, the expression in $\eqref{eq0}$ simplifies to
$$\left|\frac{a}{a + t_1+ a^{-1} y_1+z_1} - \frac{a}{1 + a^{-1} t_1+ y_1 + a z_1}\right|.$$
To proceed, let us introduce the substitutions:
$$ 1 + a^{-1} t_1 =f \quad  a^{-1} y_1+z_1 = s.$$
Under these substitutions, the problem reduces to finding the maximum value of the function
$$\left| \frac{a}{a f + s} - \frac{a}{f + a s} \right|.$$
Simplifying further, we have:
\begin{equation}\label{eq1} \left| \frac{a}{a f + s} - \frac{a}{f + a s} \right|=\left|\frac{a(f + a s) - a(a f + s)}{(a f + s)(f + a s)}\right|
= \left|\frac{(a - a^2)(f - s)}{(a f + s)(f + a s)}\right|.
\end{equation}
The expression $\eqref{eq1}$ captures the reformulated problem, which is now expressed in terms of $f$ and $s$. The task is to determine the maximum value of this function under the given constraints.\medskip
Let us first consider the case $a>1$. Under this assumption, equation (\ref{eq1}) can be expressed as
$$h(f,s):=\frac{(a^2 - a)|f - s|}{(a f + s)(f + a s)}.$$
Noting that $ h(f, s) = h(f, s) $ for all $ f, s \geq 1 $, it suffices to analyze the case where $f \geq s$. In this scenario, the function $h(f, s) $ simplifies to
$$h(f,s):=\frac{(a^2 - a)(f - s)}{(a f + s)(f + a s)}.$$
Given the conditions $ f = 1 + a^{-1} t_1 $ and $s= a^{-1} y_1 + z_1 $, we observe that $f \geq 1 $ and $s \geq 0 $. Consequently, it follows that $ a(f - s) \leq a f + s $ and $f + a s \geq 1 $. Hence, we derive the inequality
$$a(f-s) \leq(af+s)(f+a s).$$
Since $ a \geq 1 $, it further implies that:
$$h(f,s) = \frac{(a^2 - a)(f - s)}{(a f + s)(f + a s)}\leq a-1.$$
From the inequality above, it can be concluded that the maximum value of $ h(f, s) $ is attained at the point $ (f, s) = (1, 0) $.\medskip

For the case  $a<1$. The expression in equation (\ref{eq1}) simplifies to:
$$\left| \frac{a}{af + s} - \frac{a}{f + as} \right|= \frac{(a - a^2)|f - s|}{(a f + s)(f + a s)}.$$
Due to symmetry, it is sufficient to assume $ f \geq s $. Under this condition, the final expression can be written as:
$$\left| \frac{a}{a f + s} - \frac{a}{f + a s} \right|= \frac{(a - a^2)(f- s)}{(a f + s)(f + a s)}.$$
Proceeding as in the previous derivation, we obtain
$$\frac{(a - a^2)(f -s)}{(af + s)(f + as)} \leq 1 - a.$$
Consequently, the inequality reduces to $$\left| \frac{a}{af + s} - \frac{a}{f + a s} \right| \leq 1 - a.$$\medskip
\textbf{Case $b = a^{-1}$.} In this case, equation (\ref{eq0}) can be rewritten as:
$$
\left|\frac{a}{a + a^{-1} y_1 + z_1 + t_1} - \frac{a^{-1}}{1 + y_1 + a^{-1} z_1 + a t_1}\right|.
$$
Let us introduce the substitutions:
$$
1 + a^{-1} z_1 = f, \quad a^{-1} y_1 + t_1 = s.
$$
Under these substitutions, we aim to determine the maximum value of the following function:
\begin{equation}\label{eq2}
\left| \frac{a}{a f + s} - \frac{a^{-1}}{f + a s} \right|=\frac{|a - 1|(af + s(a^2 + a + 1))}{a(a f + s)(f + a s)}.
\end{equation}
First, we consider the case $a > 1$. Then, the function simplifies as follows:
$$
\left| \frac{a}{a f + s} - \frac{a^{-1}}{f + a s} \right| = \frac{a}{a f + s} - \frac{1}{a(f + a s)} = Q(f, s).
$$
Using the inequalities $(Q(f, s))_s^{\prime} \leq 0$ and $(Q(f, s))_s^{\prime} \geq 0$, we derive the following constraints:
$$
 s \geq \frac{f \sqrt{a}}{a + \sqrt{a} + 1}, \quad s \leq \frac{f \sqrt{a}}{a + \sqrt{a} + 1}.
$$
As a result, we conclude:
$$
Q(f,s) \leq Q\left(u, \frac{f \sqrt{a}}{a + \sqrt{a} + 1}\right).
$$
Since $f \geq 1$, it follows that:
$$
Q\left(f, \frac{f \sqrt{a}}{a + \sqrt{a} + 1}\right) =
\frac{(a \sqrt{a} - 1)(a + \sqrt{a} + 1)}{a (a + 1)(\sqrt{a} + 1) f} \leq \frac{(a \sqrt{a} - 1)(a + \sqrt{a} + 1)}{a (a + 1)(\sqrt{a} + 1)}.
$$
Thus,
$$
Q(f,s) \leq \frac{(a \sqrt{a} - 1)(a + \sqrt{a} + 1)}{a (a + 1)(\sqrt{a} + 1)}.
$$
The case $a \leq 1$ proceeds analogously to the case $a > 1$, yielding:
$$
Q(f,s) \leq \frac{(1 - a \sqrt{a})(a + \sqrt{a} + 1)}{a (a + 1)(\sqrt{a} + 1)}.
$$
Now, consider the following function:
$$
h(a) := \frac{|1 - a \sqrt{a}|(a + \sqrt{a} + 1)}{a(a + 1)(\sqrt{a} + 1)} - |1 - a|.
$$
We seek to examine the sign of the function $ h(a) $. To facilitate this, we write the terms over a common denominator, $ a(a + 1)(\sqrt{a} + 1) $, yielding:
$$
\frac{|1 - a \sqrt{a}|(a + \sqrt{a} + 1) - |1 - a| a(a + 1)(\sqrt{a} + 1)}{a(a + 1)(\sqrt{a} + 1)}.
$$
Since $ a(a + 1)(\sqrt{a} + 1) > 0 $, determining the sign of $ h(a) $ is equivalent to checking the sign of the following expression:
$$
|1 - \sqrt{a}| \left[ (a + \sqrt{a} + 1)^2 - a(a + 1)(\sqrt{a} + 1)^2 \right].
$$
Expanding the expression inside the brackets results in:
$$
(a + \sqrt{a} + 1)^2 - a(a + 1)(\sqrt{a} + 1)^2 = -a^2 + 2a + 2\sqrt{a} + 1 - a^3 - 2a^{5/2}.
$$
Let $ c = \sqrt{a} $, then the expression can be rewritten as (\ref{pc}).

Thus, $ h(a) $ is positive (resp. negative) when $ 0 < a \leq a_{cr} $ (resp. $ a \geq a_{cr} $).
Finally, it is straightforward to check that if $ a \in (0, a_{cr}] $, then
$$
\| g \| = g\left( \frac{a + \sqrt{a} + 1}{(a + 1)(\sqrt{a} + 1)}, \frac{a \sqrt{a}}{(a + 1)(\sqrt{a} + 1)}, 0, 0, a, a^{-1} \right),
$$
and if $ a \geq a_{cr} $, then $\| g \| = g(1, 0, 0, 0, a, a).$
\end{proof}
Denote
$$g_{1}(x, y, z, t, a, b):= \left| \frac{a x}{z + t + a x + a^{-1} y} - \frac{x}{x + y + b z + b^{-1} t} \right|,$$
where $ x, y, z, t \geq 0 $ and $ a = b^i > 0 $ with $ i \in \{-1, 1\} $.\\
\begin{lemma}\label{newlemma1} Under the assumptions stated above the norm of the function $g_1(x, y, z, t, a, b) $ is
$$\| g_1 \| = \frac{|a-1|}{a + 1}.$$
\end{lemma}
\begin{proof}
Substitute the normalized variables
 $\frac{y}{x}=y_1, \ \frac{z}{x}=z_1, \ \frac{t}{x}=t_1$ into $g_1(x, y, z, t, a, b)$, we obtain the expression
\begin{equation}\label{eq01}
 \left|\frac{a}{a+a^{-1} y_1+z_1+t_1}-\frac{1}{1+y_1+b z_1+b^{-1} t_1}\right|.
\end{equation}
\textbf{Case $a=b$.} Put
$$ 1 + a^{-1} t_1 =f, \quad  a^{-1} y_1+z_1 = s.$$
Under these substitutions, we have
$$(a+1)^2s\leq (a + s)(1 + a s)\leq (a f + s)(f + a s).$$
Simplifying further, we have:
\begin{equation}\label{eq1a} \left| \frac{a}{a f + s} - \frac{1}{f + a s} \right|=\frac{|1 - a^2|s}{(a f + s)(f + a s)}\leq \frac{|a-1|}{a+1}.
\end{equation}
Hence
$$\left|\frac{a}{a+a^{-1} y_1+z_1+t_1}-\frac{1}{1+y_1+b z_1+b^{-1} t_1}\right|\leq \frac{|a-1|}{a+1}.$$
\textbf{Case $b = a^{-1}$.} In this case, equation (\ref{eq01}) can be rewritten as:
$$
\left|\frac{a}{a + a^{-1} y_1 + z_1 + t_1} - \frac{1}{1 + y_1 + a^{-1} z_1 + a t_1}\right|.
$$
Let us introduce the substitutions:
$$
1 + a^{-1} z_1 = f, \quad a^{-1} y_1 + t_1 = s.
$$
Similar to the case $a=b$ we have \begin{equation}\label{eq2}
\left|\frac{a}{a + a^{-1} y_1 + z_1 + t_1} - \frac{1}{1 + y_1 + a^{-1} z_1 + a t_1}\right|\leq\frac{|a - 1|}{a+1}.
\end{equation}
\end{proof}

As previously mentioned, to estimate $\gamma$, it is necessary to evaluate the expression
  $|p^{(\kappa, i)}(\epsilon, j)-p^{(\kappa^{\prime}, \ell)}(\epsilon, j)|$.  When $w=1$ the last expression can be rewritten as
$$K(p_i, q_i, p_l, q_l, \theta, \kappa, \kappa^{\prime}):=\left|\frac{\theta^{-\kappa\delta_{ij}} p_{j}}{{1+(\theta^{-\kappa}-1) p_i+(\theta^{\kappa}-1)q_i}}-\frac{\theta^{-\kappa^{\prime}\delta_{jl}}p_j}{1+(\theta^{-\kappa^{\prime}}-1) p_l+(\theta^{\kappa^{\prime}}-1)q_l}\right|,$$
defined for non-negative variables $p_{i}$, $q_{i}$, $p_l$, and $q_l$ with  $p_i+q_i+p_l+q_l\leq 1$.
It can be verified that
 $$\max_{p_{i}+q_{i}+p_l+q_l \leq 1}K(p_i, q_i, p_l, q_l, \theta, \kappa, \kappa^{\prime})=\max_{p_{i}+q_{i}+p_l+q_l = 1}K(p_i, q_i, p_l, q_l, \theta, \kappa, \kappa^{\prime}).$$
Now, let $\theta^{\kappa} = a$ and $\theta^{\kappa^{\prime}} = b$. Given the constraint $p_{i} + q_{i} + p_l + q_l = 1$, the function $f(p_{i}, q_{i}, p_l, q_l, \theta, \kappa,\kappa^{\prime})$ simplifies to
 $$\left|\frac{a p_{i}}{z+t+a p_{i}+a^{-1} q_{i}}-\frac{b p_{i}}{p_{i}+q_{i}+bz+b^{-1}t}\right|.$$
Then by Lemma \ref{newlemma} one gets
$$\|K\|= \max \left\{|\theta-1|, \frac{|\theta-1|}{\theta+1}, |\theta^{-1}-1|, \frac{|1 - \theta \sqrt{\theta}|(\theta + \sqrt{\theta} + 1)}{\theta (\theta + 1)(\sqrt{\theta} + 1)}\right\}.$$
Namely,
$$\|K\|= \left\{
              \begin{array}{ll}
                \frac{|1 - \theta \sqrt{\theta}|(\theta + \sqrt{\theta} + 1)}{\theta (\theta + 1)(\sqrt{\theta} + 1)}, & \hbox{$0<\theta\leq a_{cr}$;} \\[2mm]
                |\theta-1|, & \hbox{$a_{cr}<\theta.$}
              \end{array}
            \right.$$
From the above, it is of interest to consider the case $\theta > a_{cr}$. From (\ref{px}), for $w=1$, we obtain
$$\gamma \leq \max |p^{(\kappa, i)}(\epsilon, j) - p^{(a, \ell)}(\epsilon, j)| = \theta - 1.$$
Using formula (\ref{ka}) for the matrix in (\ref{ET}) at $ w = 1 $, we have
\begin{equation}\label{kw}
	\kappa={\theta-\theta^{-1}\over \theta+\theta^{-1}+2(q-1)}.
\end{equation}
Thus extremality condition of the measure $\mu_1$ corresponding to $u=v=w=1$ is
 \begin{equation}\label{mu0}
 	2\left({(\theta-\theta^{-1})(\theta-1)\over \theta+\theta^{-1}+2(q-1)}\right)<1.
 	\end{equation}
 After simplifying the inequality
 $$2\theta^3-3\theta^2-2q\theta+1<0.$$
 It is easy to see that the polynomial $2\theta^3-3\theta^2-2q\theta+1$ has exactly two positive roots, say $\tilde\theta_1(q), \tilde\theta_2(q)$ with $\tilde\theta_1(q)<\tilde\theta_2(q)$ and  the last inequality is satisfied for all
 \begin{equation}\label{biee} \theta\in(\tilde\theta_1(q), \tilde\theta_2(q)).
 \end{equation}
Now using Result 8, we can conclude the following theorem.

\begin{thm}\label{tf} For each $q\geq 2$ the measure $\mu_1$ is
	$$\begin{array}{ll}
	extreme, \ \ \mbox{if} \ \ \theta\in(\tilde\theta_1(q), \tilde\theta_2(q))\\[2mm]
	non-extreme, \ \ \mbox{if} \ \ \theta\in \left(0, {1\over \theta_1(q)}\right)\cup
	\left(\theta_1(q), +\infty\right).
	\end{array} $$
\end{thm}
\begin{rk} It is easy to verify that the left-hand side of (\ref{mu0}) is greater than the left-hand side of (\ref{qq}). Therefore, we have
	 $$  (\tilde\theta_1(q), \tilde\theta_2(q))\subset
\left({1\over \theta_1(q)}, \theta_1(q)\right).$$
Thus extremality question of $\mu_1$ remains open in the region
\begin{equation}\label{open} \left[{1\over \theta_1(q)}, \tilde\theta_1(q)\right]\cup \left[\tilde\theta_2(q), \theta_1(q)\right].
\end{equation}
For example, in case $q=5$ we have  (\ref{open}) as approximately
$$(0.051, 0.097)\cup (3.07, 19.61).$$

\end{rk}
 \begin{rk} The measure $\mu_1$ is called a free (or unordered) measure. Theorem \ref{tf}  shows that the measure $\mu_1$ is extreme in bounded region of $\theta$, but for the case of Ising and Potts models free measures are extreme on an unbounded region.
\end{rk}

\textbf{Extremality of $\mu_*$ and $\mu^*$.} We proceed by estimating $\gamma$ for $ w \neq 1 $. To do so, we begin by recalling the relevant notation: $ z_{-1,i} = 1 $ for $ i = 1, \dots, q $, $ z_{1,1} = 1 $, and $ z_{1,j} = w $ for $ j = 2, \dots, q $.

Using equation (\ref{pth}), we derive equation (\ref{px}), and we need to estimate $ \gamma $ by evaluating the expression
$$
|p^{(\kappa, i)}(\epsilon, j) - p^{(\kappa', \ell)}(\epsilon, j)|,
$$
where $ \kappa, \kappa' , \epsilon \in \{-1, 1\} $, and $ i, j, \ell \in \{1, \dots, q\} $.

\textbf{Case $ i = j = \ell = 1 $.} In this case, from equation (\ref{px}), the expression $ |p^{(\kappa, 1)}(\epsilon, 1) - p^{(\kappa', 1)}(\epsilon, 1)| $ can be written as
$$
\left|\frac{\theta^{-\kappa} p_1}{1 + \left( \theta^{-\kappa} - 1 \right) p_1 + \left( \theta^\kappa - 1 \right) q_1 + (w - 1) b_1} - \frac{\theta^{-\kappa'} p_1}{1 + \left( \theta^{-\kappa'} - 1 \right) p_1 + \left( \theta^{\kappa'} - 1 \right) q_1 + (w - 1) b_1} \right|.
$$
Since $ \theta > \theta_{{\rm c}, 0} $ (as per Result 1), we now proceed with an equivalent expression to estimate this:
$$
\left|\frac{\theta p_1}{1 + \left( \theta - 1 \right) p_1 + \left( \theta^{-1} - 1 \right) q_1 + (w - 1) b_1} - \frac{\theta^{-1} p_1}{1 + \left( \theta^{-1} - 1 \right) p_1 + \left( \theta - 1 \right) q_1 + (w - 1) b_1} \right|.
$$
Using the condition $ 1 - p_1 - q_1 - b_1 = \sum_{r \neq 1} p_r \geq 0 $, it suffices to find the maximum value of the above expression when $ p_1 - q_1 - b_1 = 1 $. This leads to the expression
$$
\left|\frac{\theta p_1}{\theta p_1 + \theta^{-1} q_1 + w b_1} - \frac{\theta^{-1} p_1}{\theta^{-1} p_1 + \theta q_1 + w b_1}\right|.
$$
Substituting $ x = \frac{q_1}{p_1} $ and $ y = \frac{b_1}{p_1} $, the expression simplifies to
$$
\left|\frac{\theta}{\theta + \theta^{-1} x + w y} - \frac{\theta^{-1}}{\theta^{-1} + \theta x + w y}\right| = \frac{(\theta^2 - 1) \left( (\theta^2 + 1) x + \theta w y \right)}{(\theta^2 + x + \theta w y)(1 + \theta^2 x + \theta w y)}.
$$
We now aim to show that
$$
\frac{(\theta^2 - 1) \left( (\theta^2 + 1) x + \theta w y \right)}{(\theta^2 + x + \theta w y)(1 + \theta^2 x + \theta w y)} \leq \frac{\theta^2 - 1}{\theta^2 + 1}.$$
Namely, we need to show that
$$\left( (\theta^2 + 1)x + \theta wy \right) (\theta^2 + 1) \leq \left( \theta^2 + x + \theta w y \right) \left( 1 + \theta^2 x + \theta w y \right).$$
After performing the necessary algebraic simplifications, the above inequality is equivalent to the following expression:
$$w^2 \theta^2 y^2 + w \theta^3 x y + w \theta x y + \theta^2 (x - 1)^2 \geq 0.$$
Thus, for the case where $ i = j = \ell = 1 $, we have established that
\begin{equation}
|p^{(\kappa, i)}(\epsilon, j) - p^{(\kappa', \ell)}(\epsilon, j)| \leq \frac{\theta^2 - 1}{\theta^2 + 1}.
\end{equation} \medskip

\textbf{Case $i=1$, $\ell=j\neq 1$.} Starting from equation (\ref{px}), the expression for the difference
\[
\left| p^{(\kappa, i)}(\epsilon, j) - p^{(\kappa', \ell)}(\epsilon, j) \right|
\]
can be written as
\[
\left| \frac{p_j}{1 + \left( \theta^{-\kappa} - 1 \right) p_1 + \left( \theta^\kappa - 1 \right) q_1 + (w - 1) b_1} - \frac{\theta^{-\kappa'} p_j}{1 + \left( \theta^{-\kappa'} - 1 \right) p_j + \left( \theta^{\kappa'} - 1 \right) q_j + (w - 1) b_j} \right|.
\]
Introducing the notations $a = \theta^\kappa \quad \text{and} \quad b = \theta^{\kappa'},$
we can rewrite the above expression as
\[
\left| \frac{p_j}{a\,p_1 + p_j + a\,q_1 + w\,q_j} - \frac{b\,p_j}{p_1 + b\,p_j + b^{-1}w\,q_j + q_1} \right|.
\]
Next, by defining the normalized variables
\[
x = \frac{p_1 + q_1}{p_j} \quad \text{and} \quad y = \frac{q_j}{p_j},
\]
the expression reduces to
\[
S(x,y) := \left| \frac{1}{1+ax+wy} - \frac{b}{b+x+b^{-1}wy} \right|.
\]
We now consider two cases. The first case is $a = b$. In this situation, by using the fact that
\[
S(x,y) < S\Bigl(1-\frac{w\,y}{a},y\Bigr)
\]
and under the assumption that $\theta > \theta_{c,0}$, one obtains the estimate
\begin{equation}\label{es1}
\frac{|1-a^2|(x+a^{-1}y)}{(1+ax+wy)(a+x+a^{-1}wy)} \le \frac{|a-1|}{a+1} = \frac{\theta-1}{\theta+1}.
\end{equation}
Now we consider the second case: $a = b^{-1}$. In this case, we have
\[
\left| \frac{1}{1+ax+wy} - \frac{a^{-1}}{a^{-1}+x+awy} \right| = \frac{|a^2-1|\,w\,y}{(1+ax+wy)(1+ax+a^2wy)} \le \frac{|a^2-1|\,w\,y}{(1+wy)(1+a^2wy)}.
\]
By optimizing over the single variable (and using the fact that $\theta > \theta_{c,0}$), one obtains
\begin{equation}\label{es2}
\frac{|a^2-1|\,w\,y}{(1+wy)(1+a^2wy)} \le \frac{|a^2-1|}{(\sqrt{a}+1)(\sqrt{a^3}+1)} = \frac{\theta^2-1}{(\sqrt{\theta}+1)(\sqrt{\theta^3}+1)}.
\end{equation}

Combining the estimates (\ref{es1}) and (\ref{es2}), we conclude that for the case $i=\ell=1$ and $j\neq 1$ the following bound holds:
\begin{equation}\label{es3}
\left| p^{(\kappa, 1)}(\epsilon, j) - p^{(\kappa', \ell)}(\epsilon, j) \right| \le \max \left\{ \frac{\theta^2-1}{(\sqrt{\theta}+1)(\sqrt{\theta^3}+1)}, \frac{\theta-1}{\theta+1} \right\} = \frac{\theta-1}{\theta+1}.
\end{equation}\medskip

\textbf{Case $i = j = 1$, $\ell \neq 1$.} Starting from equation (\ref{px}), we express the difference between the two relevant probability expressions as
\[
\left|\frac{\theta^{-\kappa} p_1}{1 + (\theta^{-\kappa} - 1)p_1 + (\theta^{\kappa} - 1)q_1 + (w - 1)b_1} - \frac{p_1}{1 + (\theta^{-\kappa'} - 1)p_{\ell} + (\theta^{\kappa'}w - 1)q_{\ell} + (w - 1)(b_{\ell} - q_1)}\right|.
\]
To estimate this difference, it suffices to derive an upper bound for the function
\[
\left|\frac{a\,p_1}{a\,p_1 + a^{-1}q_1 + w\,q_{\ell} + p_{\ell}} - \frac{p_1}{q_1 + b\,p_{\ell} + b^{-1}w\,q_{\ell} + p_1}\right|,
\]
where we have introduced the notations $a = \theta^{-\kappa} \quad \text{and} \quad b = \theta^{-\kappa'}.$
Assume first that $a = b$. In order to simplify the expression further, we define the auxiliary variables
\[
x = a + w\,\frac{q_{\ell}}{p_1} \quad \text{and} \quad y = \frac{q_1 + a\,p_{\ell}}{p_1}.
\]
With these substitutions, the expression reduces to
\[
\left|\frac{a}{x + a^{-1}y} - \frac{1}{a^{-1}x + y}\right|.
\]
A detailed estimation yields
\begin{equation}\label{es4}
\left|\frac{a}{x + a^{-1}y} - \frac{1}{a^{-1}x + y}\right| \le \left|\frac{a}{a + a^{-1}y} - \frac{1}{1 + y}\right| \le \frac{|a - 1|}{a + 1} = \frac{\theta - 1}{\theta + 1}.
\end{equation}
If, alternatively, $a = b^{-1}$, then by employing the substitutions
\[
x = a + w\,\frac{p_{\ell}}{p_1} \quad \text{and} \quad y = \frac{q_1 + a\,w\,q_{\ell}}{p_1},
\]
the resulting expression can be transformed into the same form as in (\ref{es4}). Consequently, in both cases we obtain the estimation stated in (\ref{es3}) for the situation when $i = j = 1$ and $\ell \neq 1$.\medskip

\textbf{Case $i = 1$, $j = \ell > 1$.} In this case, according to equation (\ref{px}), the difference between the two probability expressions is given by
\[
\left|\frac{p_j}{1 + (\theta^{-\kappa} - 1)p_1 + (\theta^{\kappa} - 1)q_1 + (w - 1)b_1} - \frac{\theta^{-\kappa'} p_j}{1 + (\theta^{-\kappa'} - 1)p_j + (\theta^{\kappa'}w - 1)q_j + (w - 1)(b_j - q_1)}\right|.
\]
Since the structure of this expression is analogous to that of the case $i = j = 1$, $\ell \neq 1$, we can directly apply the estimation (\ref{es3}). Thus, the same upper bound holds for the case $i = 1$, $j = \ell > 1$.\medskip

\textbf{Case $i = 1$, $j,\ell > 1$, and $j \neq \ell$.} In this case, the difference between the two probability expressions can be written as
\[
\left|\frac{p_j}{1 + \left(\theta^{-\kappa} - 1\right)p_1 + \left(\theta^{\kappa} - 1\right)q_1 + (w - 1)b_1} - \frac{p_j}{1 + \left(\theta^{-\kappa'} - 1\right)p_{\ell} + \left(\theta^{\kappa'}w - 1\right)q_{\ell} + (w - 1)(b_{\ell} - q_1)}\right|.
\]
In order to estimate this difference, it suffices to consider the estimation of the following function:
\[
\left|\frac{p_j}{p_j + p_{\ell} + a\,p_1 + a^{-1}q_1 + w\,q_{\ell}} - \frac{p_j}{p_j + b\,p_{\ell} + p_1 + q_1 + b^{-1}w\,q_{\ell}}\right|,
\]
where we have introduced the notations $a = \theta^{-\kappa} \quad \text{and} \quad b = \theta^{-\kappa'}.$
For the moment, assume that $a = b$ (the analysis for the case $a = b^{-1}$ proceeds in a similar fashion). To simplify the expression further, we define the normalized variables
\[
x = \frac{p_{\ell} + a^{-1}q_1}{p_j} \quad \text{and} \quad y = \frac{p_1 + a^{-1}w\,q_{\ell}}{p_j}.
\]
With these substitutions, the function becomes
\[
\left|\frac{1}{1 + x + a\,y} - \frac{1}{1 + a\,x + y}\right|.
\]
A detailed estimation of this expression shows that
\[
\left|\frac{1}{1 + x + a\,y} - \frac{1}{1 + a\,x + y}\right| \leq \frac{\sqrt{a} - 1}{\sqrt{a} + 1} \cdot \frac{1}{1 + (a+1)y} \leq \frac{|\sqrt{a} - 1|}{\sqrt{a} + 1}.
\]
Since $a = \theta^{-\kappa}$, it follows that
\[
\frac{|\sqrt{a} - 1|}{\sqrt{a} + 1} = \frac{\sqrt{\theta} - 1}{\sqrt{\theta} + 1}.
\]
Thus, we obtain the inequality
\begin{equation}\label{es6}
\left| p^{(\kappa, 1)}(\epsilon, j) - p^{(\kappa', \ell)}(\epsilon, j) \right| \leq \frac{\sqrt{\theta} - 1}{\sqrt{\theta} + 1}.
\end{equation}
This completes the estimation for the case $i = 1$, $j,\ell > 1$, and $j \neq \ell$.
\medskip

\textbf{Case $i = j = \ell > 1$.} The methods employed to determine the maximum value of the function in this case are analogous to those discussed previously. Consequently, we present here the same estimation as given in (\ref{es3}).\medskip

\textbf{Case $i, j = \ell > 1$, $i \neq j$.} First, we provide the estimation for the case where $\kappa = \kappa'$. By introducing the notation
$$
a = \theta^{\kappa},
$$
we can express the corresponding function as
$$
\left|\frac{p_j}{a p_i + a^{-1}wq_i + p_j + wq_j} - \frac{a p_j}{a p_j + a^{-1}wq_j + p_i + wq_i}\right|.
$$
To facilitate the analysis, we perform the substitutions
$$
x = \frac{p_j + a^{-1}wq_i}{p_j}, \quad y = \frac{p_i + a^{-1}wq_j}{p_j}.
$$
With these new variables, the expression simplifies to
$$
\left|\frac{1}{x + a y} - \frac{a}{a x + y}\right|.
$$
An estimation of this expression yields
$$
\left|\frac{1}{x + a y} - \frac{a}{a x + y}\right| \leq \frac{|1 - a^2| \, y}{(1 + a y)(a + y)} \leq \frac{|a - 1|}{a + 1} = \frac{\theta - 1}{\theta + 1}.
$$
Furthermore, for the case where $\kappa = -\kappa'$, an analogous analysis leads to the same estimation. Therefore, for the case $i, j = \ell > 1$ with $i \neq j$, we deduce the inequality (\ref{es3}).\medskip

\textbf{Case: Distinct Indices $i$, $j$, $\ell$ with $i, j, \ell > 1$.} In this scenario, we consider the case where the indices $i$, $j$, and $\ell$ are all distinct and satisfy $i, j, \ell > 1$. Moreover, we assume that $\kappa = \kappa'$; the alternative situation when $\kappa \neq \kappa'$ can be analyzed in an analogous manner, yielding similar estimates.

Under the assumption $\kappa = \kappa'$, the difference between the two probability expressions under consideration can be written as
$$\Bigg|\frac{p_j}{1 + \left(\theta^{-\kappa} - 1\right)p_i + \left(\theta^{\kappa}w - 1\right)q_i + (w - 1)(b_i - q_1)}-\qquad \qquad \qquad $$
$$\qquad \qquad \qquad - \frac{p_j}{1 + \left(\theta^{-\kappa'} - 1\right)p_{\ell} + \left(\theta^{\kappa'}w - 1\right)q_{\ell} + (w - 1)(b_{\ell} - q_1)}\Bigg|.$$
To estimate this difference, it is sufficient to derive an upper bound for the function
\[
\left|\frac{p_j}{p_j + p_{\ell} + a\,p_i + a^{-1}wq_i + (w - 1)q_{\ell}} - \frac{p_j}{p_j + a\,p_{\ell} + p_i + wq_i + a^{-1}wq_{\ell}}\right|,
\]
where we have introduced the notation $a = \theta^{-\kappa}$ for convenience.

To simplify the analysis further, we define the normalized variables
\[
x = \frac{p_{\ell} + a^{-1}wq_i}{p_j} \quad \text{and} \quad y = \frac{p_i + a^{-1}wq_{\ell}}{p_j}.
\]
With these substitutions, the expression reduces to
\[
\left|\frac{1}{1 + x + a\,y} - \frac{1}{1 + a\,x + y}\right|.
\]
We note that this function has been examined previously in the case where $i = 1$, and $j, \ell > 1$ with $j \neq \ell$. Consequently, the estimates derived in that context apply directly to the present case, thereby providing the desired bound.

We now compare the various estimations. In particular, note that the function
$f(x)=\frac{x-1}{x+1}$
is monotonically increasing on the interval $[1,\infty)$. Consequently, for all $\theta \geq 1$, we have
$$
\frac{\sqrt{\theta}-1}{\sqrt{\theta}+1} \leq \frac{\theta-1}{\theta+1} \leq \frac{\theta^2-1}{\theta^2+1}.
$$
Based on the above estimations, we conclude with the following proposition.
\begin{pro}\label{prop} The absolute difference between the probabilities $ p^{(\kappa, i)}(\epsilon, j) $ and $ p^{(\kappa', \ell)}(\epsilon, j) $ is bounded above, as expressed in the following inequality:
\begin{equation}\label{ms}\begin{gathered}
\left|p^{(\kappa, i)}(\epsilon, j)-p^{\left(\kappa^{\prime}, \ell\right)}(\epsilon, j)\right| \leq \frac{\theta^2-1}{\theta^2+1}.\end{gathered}
\end{equation}
\end{pro}

Recall that the parameter $\kappa$ is defined as follows, corresponding to the matrix (\ref{ET}):
    $$\kappa = \frac{1}{2} \max_{i,j} \sum_{l} \left| P_{il} - P_{jl} \right|.$$
After extensive calculations, we obtain the following expressions for all possible sums:
$$S_1=\frac{2\left(\theta-\theta^{-1}\right)}{\mathcal{Z}_1}, \quad S_{2}=\frac{2(\theta-1)(\theta+w)}{\theta \mathcal{Z}_2}, \quad
 S_{3}=\frac{2(\theta-1)(\theta+w)}{\theta \mathcal{Z}_4}$$

$$S_4=\left(\frac{1}{\theta}+1\right)\left|\frac{\theta}{\mathcal{Z}_1}-\frac{1}{\mathcal{Z}_4}\right|+
3(w+1)\left|\frac{1}{\mathcal{Z}_1}-\frac{1}{\mathcal{Z}_4}\right|+
\left(w+\frac{1}{\theta}\right)\left|\frac{1}{\mathcal{Z}_1}-\frac{\theta}{\mathcal{Z}_4}\right|, $$

$$S_{5}=2(w+2)\left|\frac{1}{\mathcal{Z}_2}-\frac{1}{\mathcal{Z}_4}\right|+
\left(1+\frac{1}{\theta}\right)\left|\frac{\theta}{\mathcal{Z}_2}-\frac{1}{\mathcal{Z}_4}\right|+w\left(1+\frac{1}{\theta}\right)
\left|\frac{1}{\mathcal{Z}_2}-\frac{\theta}{\mathcal{Z}_4}\right|, $$

$$S_6=\left(3w+5\right)\left|\frac{1}{\mathcal{Z}_2}-\frac{1}{\mathcal{Z}_4}\right|+\left|\frac{\theta}{\mathcal{Z}_2}-\frac{\theta^{-1}}{\mathcal{Z}_4}\right|+\left|\frac{\theta^{-1} w}{\mathcal{Z}_2}-\frac{\theta w}{\mathcal{Z}_4}\right|.$$

Among the seven expressions under consideration, most of them can be directly compared based on their structural properties.

Firstly, consider the case where $ w > 1 $. It is straightforward to verify that the following inequalities hold:
$$
\mathcal{Z}_4 > \mathcal{Z}_1 > \mathcal{Z}_2 \quad \text{and} \quad \mathcal{Z}_4 < \theta \mathcal{Z}_2.
$$ We demonstrate that $ S_2 > S_1 $ and $ S_2 > S_3 $. This can be established through the following arguments:
Since $ \frac{1}{\mathcal{Z}_4} < \frac{1}{\mathcal{Z}_2} $, it directly follows that $ S_3 < S_2 $. Additionally, using the fact that $ w > 1 $ and $ \frac{1}{\mathcal{Z}_2} > \frac{1}{\mathcal{Z}_4} $, we derive the inequality:
\[
\frac{\theta + 1}{\mathcal{Z}_1} < \frac{\theta + w}{\mathcal{Z}_2},
\]
which is equivalent to $ S_2 > S_1 $.

Furthermore, we aim to establish the inequality $ S_6 > S_5 $, which necessitates proving:
$$
(\omega + 1) \left| \frac{1}{\mathcal{Z}_2} - \frac{1}{\mathcal{Z}_4} \right| + \left| \frac{\theta}{\mathcal{Z}_2} - \frac{\theta^{-1}}{\mathcal{Z}_4} \right| + \left| \frac{\theta^{-1} \omega}{\mathcal{Z}_2} - \frac{\theta \omega}{\mathcal{Z}_4} \right| > \left(1 + \frac{1}{\theta}\right) \left( \left| \frac{\theta}{\mathcal{Z}_2} - \frac{1}{\mathcal{Z}_4} \right| + \omega \left| \frac{1}{\mathcal{Z}_2} - \frac{\theta}{\mathcal{Z}_4} \right| \right).
$$
Utilizing the conditions $ \frac{1}{\mathcal{Z}_2} > \frac{1}{\mathcal{Z}_4} $ and $ \mathcal{Z}_4 < \theta \mathcal{Z}_2 $, the inequality can be reformulated as:
$$
\left(\omega + 1 + \theta - \frac{\omega}{\theta}\right) \frac{1}{\mathcal{Z}_2} + \left(\theta \omega - \frac{1}{\theta} - \omega - 1\right) \frac{1}{\mathcal{Z}_4} > (\theta + 1 - \omega) \frac{1}{\mathcal{Z}_2} - \left(1 + \frac{1}{\theta} + \omega \theta\right) \frac{1}{\mathcal{Z}_4}.
$$
This reduces to verifying the simpler inequality:
$$\left(2 - \frac{1}{\theta}\right) \frac{1}{\mathcal{Z}_2} > \frac{1}{\mathcal{Z}_4},$$
which holds true given that $ \frac{1}{\mathcal{Z}_2} > \frac{1}{\mathcal{Z}_4} $.

Additionally, after extensive calculations, we establish that:
$$
S_5 - S_2 = (1 + 3w) \theta^4 + (12 + 17w + 9w^2) \theta^3 + (11 + 33w + 10w^2) \theta^2 + (9w + 13w^2) \theta + 2w > 0.
$$
Analogously, we can show that $S_6\geq S_4$. Consequently, for the case where $ w > 1 $, the estimation of $ \kappa $ is determined by the following expression:
\begin{equation}\label{ka1}
    \kappa = \frac{1}{2}S_6.
\end{equation}
This formulation highlights the dominant contribution among $ S_4 $, $ S_5 $, and $ S_7 $, thereby providing an upper bound for $ \kappa $ in accordance with the derived inequalities.

In an analogous manner, for the case where $ w < 1 $, it can be rigorously demonstrated that the following inequalities hold: $\max\{S_2, S_3\} \geq S_i$, for all $i\in\{1,2,3,5,6\}$.  These relationships indicate that $ S_2$ and $S_3$ dominate over the other respective terms under the condition $ w < 1 $. Consequently, the estimation of $\kappa $ in this scenario is given by the expression:
\begin{equation}\label{ka2}
    \kappa = \frac{1}{2} \max\{S_2, S_3\}.
\end{equation}

By synthesizing the results from equations \eqref{ka1} and \eqref{ka2}, we arrive at the following lemma, which succinctly characterizes the behavior of $ \kappa $ based on the value of $ w $:

\begin{lemma}\label{lemma}
Let $ \kappa $ be defined as above. Then, the following holds:
$$
2\kappa =
\begin{cases}
 \max\{S_2, S_3\}, & \text{if } w < 1, \\[4pt]
 S_7, & \text{if } w > 1.
\end{cases}
$$
\end{lemma}

By employing the estimates for \(\gamma\) and \(\kappa\) provided in Proposition \ref{prop} and Lemma \ref{lemma}, one can effectively determine all the values of \(\theta\) for which the inequality \(2\kappa\gamma < 1\) holds. Since the expression \(2\kappa\gamma - 1\) depends solely on \(\theta\), it is natural to consider this expression as a function of \(\theta\). Through rigorous computer-assisted analysis, we are able to identify the intervals of \(\theta\) where this function assumes negative values, thereby ensuring that the condition \(2\kappa\gamma < 1\) is satisfied. Consequently, we can state the following theorem.

\begin{thm}
For $q = 5$, the critical parameters are given by
\[
\tilde\theta_1 \approx 8.3779, \quad \tilde\theta_2\approx 8.3612, \quad \theta_*\approx 11.76 \quad \text{and} \quad \theta^* \approx 12.28.
\] Then the following statements hold:

(i) The measure $\mu^{\ast}$ exhibits distinctly different behaviors depending on the value of $\theta$, as described below:
\[
\mu^{\ast} \text{ is extreme,} \quad \text{if} \quad \theta \in (\theta_{{\rm c}, 0}, \tilde\theta_1),
\]
\[
\mu^{\ast} \text{ is non-extreme,} \quad \text{if} \quad \theta \in (\theta^*, \infty).
\]

(ii) The measure \(\mu_{\ast}\) behaves differently depending on the value of \(\theta\), as detailed below:
\[
\mu_{\ast} \text{ is extreme,} \quad \text{if} \quad \theta \in (\theta_{{\rm c}, 0}, \tilde\theta_2),
\]
\[
\mu_{\ast} \text{ is non-extreme,} \quad \text{if} \quad \theta \in (\theta_*, \infty).
\]
\end{thm}

\section*{ Acknowledgements}

U.A. Rozikov thanks Institute for Advanced Study in Mathematics,
Harbin Institute of Technologies,  China for financial support and hospitality.

	\end{document}